\documentclass[12pt]{amsart}

\usepackage[marginpar=2cm]{geometry}
\usepackage{amsmath,amsfonts,amssymb,amsthm}
\usepackage{pinlabel}
\usepackage{graphicx}
\usepackage{mathtools}
\usepackage[all]{xy}
\usepackage{hyperref}
\usepackage[usenames,dvipsnames]{color}
\usepackage{color}
\xyoption{dvips}

\numberwithin{equation}{section}

\newtheorem{thm}{Theorem}[section]
\newtheorem{cor}[thm]{Corollary}
\newtheorem{question}[thm]{Question}
\newtheorem{conj}[thm]{Conjecture}
\newtheorem{prop}[thm]{Proposition}
\newtheorem{lem}[thm]{Lemma}

\theoremstyle{remark}
\newtheorem{rem}[thm]{Remark}
\newtheorem{example}[thm]{Example}
\theoremstyle{definition}
\newtheorem{defn}[thm]{Definition}
\theoremstyle{remark}
\newtheorem{remark}[thm]{Remark}
\theoremstyle{property}

\newcommand{\R}{\mathbb{R}}

\newcommand{\Z}{\mathbb{Z}}
\newcommand{\C}{\mathbb{C}}
\newcommand{\CP}{\mathbb{C}P}
\newcommand{\RP}{\mathbb{R}P}
\newcommand{\N}{\mathbb{N}}
\newcommand{\T}{\mathbb{T}}
\newcommand{\F}{\mathbb{F}}

\newcommand{\blk}{\color{black}}

\newcommand{\im}{\operatorname{im}}
\newcommand{\OP}{\operatorname}

\newcommand{\pt}{\mathrm{pt}}

\begin{document}

\title[Legendrians from Bohr--Sommerfeld covers of Lagrangian tori]{Legendrian submanifolds from Bohr--Sommerfeld covers of monotone Lagrangian tori}

\author{Georgios Dimitroglou Rizell}
\author{Roman Golovko}

\begin{abstract}
 By a result due to Ziltener, there exist no closed embedded Bohr--Sommerfeld Lagrangians inside $\CP^n$ for the prequantisation bundle whose total space is the standard contact sphere. On the other hand, any embedded monotone Lagrangian torus has a canonical nontrivial cover which is a Bohr--Sommerfeld immersion. We draw the front projections for the corresponding Legendrian lifts inside a contact Darboux ball of the threefold covers of both the two-dimensional Clifford and Chekanov tori (the former is the Legendrian link of the Harvey--Lawson special Lagrangian cone), and compute the associated Chekanov--Eliashberg algebras. Although these Legendrians are not loose, we show that they both admit exact Lagrangian cobordisms to the loose Legendrian sphere; they hence admit exact Lagrangian caps in the symplectisation, which are non-regular Lagrangian cobordisms. Along the way, we also compute bilinearised Legendrian contact homology of a general Legendrian surface in the standard contact vector space when all Reeb chords are of positive degree, as well as the augmentation variety in the case of tori.
\end{abstract}

\address{Department of Mathematics, Uppsala University, Box 480, SE-751 06, Uppsala, Sweden}
\email{georgios.dimitroglou@math.uu.se}

\address{Faculty of Mathematics and Physics, Charles University, Sokolovsk\'{a} 83, 18000 Praha 8, Czech Republic} \email{golovko@karlin.mff.cuni.cz}
\date{\today}
\thanks{}


\maketitle

\section{Introduction}
Prequantisation $S^1$-bundles $\pi \colon E \to M$ form an important class of contact manifolds $(E^{2n+1},\alpha)$ that have been well studied from many different  points of view\color{black}. By definition, the contact form $\alpha$ is a connection one-form for the $S^1$-bundle, and the curvature is a symplectic two-form $\omega \in \Omega^2(M)$ on $M^{2n}.$ For that reason, there is a close relationship between the symplectic geometry of $(M,\omega)$ and the contact geometry of $(E,\ker \alpha).$ For instance, every Legendrian immersion inside $E$ projects to a  Lagrangian immersion inside $(M,\omega)$ which satisfies the Bohr--Sommerfeld condition  (see Definition \ref{defn:BS} and Lemma \ref{lem:prequant}).  Our goal here is to study certain embedded Legendrians $\Lambda \subset (E,\ker\alpha)$ whose projection $\pi(\Lambda) \subset M$ again has an embedded \emph{image} (but which is possibly  multiply covered). In fact, we are mainly interested in the case $M=\CP^n$ and the line-bundle $\mathcal{O}(-1),$ which produces the standard round contact sphere $(E,\alpha)=(S^{2n+1},\alpha_{\OP{st}})$. The connection between Bohr--Sommerfeld Lagrangian immersions in the projective plane and Legendrians in the standard contact sphere was also studied in recent work by \cite{LiftingImmersions} by  Baldridge--McCarthy--Vela-Vick.\color{black}

It is a standard fact that any prequantisation bundle admits symplectic fillings (which need not be exact). In the case of a complex line bundle over a K\"{a}hler variety, this is the standard fact that the unit disc-bundle of any negative complex line bundle has a strictly pseudoconvex boundary; its boundary is a prequantisation bundle. A similar construction works in the symplectic case as well.
\begin{example} In the case of the bundle $\mathcal{O}(-1)$ over $\CP^n$ the unit-circle bundle becomes the Hopf fibration $S^{2n+1}\to\CP^n$, and the associated unit-disc bundle is the blow-up of $D^{2n+2}$ at the origin. We can of course blow down the exceptional divisor to produce the standard \emph{exact} symplectic filling of $S^{2n+1}$, i.e.~the standard symplectic disc $D^{2n+2}$. One slightly confusing point here is that the contact form on the boundary of the unit-disc differs by a sign from the naturally induced connection one-form of the unit-circle bundle associated to $\mathcal{O}(-1)$. From this point of view, the prequantisation bundle $S^{2n+1} \to \CP^n$ could be said to be naturally induced from the unit-sphere bundle on $\mathcal{O}(1)$ instead, however this is only a matter of conventions.
\end{example}

Mohnke has shown that any Legendrian in the boundary of a subcritically fillable contact manifold admits a Reeb chord \cite{MohnkeChord}. Ziltener later improved this result to the following quantitative statement in the case of the round sphere:\blk
\begin{thm}[\cite{Chord}]
\label{noembeddedBSLagr}
Any closed Legendrian submanifold inside the round contact $(S^{2n+1},\alpha_{\OP{st}})$ admits a Reeb chord of length $(0,\pi/2]$ (where $\pi/2$ is half the minimal period of a Reeb orbit). In particular, there exists no closed Bohr--Sommerfeld Lagrangian embeddings inside the standard symplectic $\CP^n$ for the prequantisation bundle $\mathcal{O}(-1).$
\end{thm}
Although there are no embedded Bohr--Sommerfeld Lagrangians for the prequantisation bundle $S^{2n+1} \to \CP^n$, there still exist \blk plenty of Bohr--Sommerfeld immersions that are \emph{multiple covers} of \blk Lagrangian embeddings. Here we are interested mainly in the case of monotone tori, which all admit Bohr--Sommerfeld covers by the following result shown in Section \ref{sec:bscovers}.  (See Definition \ref{defn:monotone} for the notion of monotonicity.) \color{black}
\begin{thm}
\label{thm:bscovers}
Any monotone Lagrangian torus inside $\CP^n$ has a canonically defined $n+1$-fold cover which is a Bohr--Sommerfeld immersion for the prequantisation bundle $\mathcal{O}(-1).$ Moreover, the corresponding Legendrian lift is embedded inside $(S^{2n+1},\xi_{\OP{st}}),$ has vanishing Maslov class, and only Reeb chords of positive degree. (Here we mean the degree in the Bott sense as explained in Section \ref{sec:maslov}.)
\end{thm}
\begin{proof}
The lift is constructed in Proposition \ref{prop:lift}. The Maslov class vanishes by Proposition \ref{prop:maslov}. The Bott degree of the chords was finally computed in Proposition \ref{prop:bottdegree}.
\end{proof}
We then turn our eyes to some particular examples of Bohr--Sommerfeld covers of embedded Lagrangians inside $\CP^2.$ Since $S^{2n+1} \setminus \{\pt\}$ is contactomorphic to a Darboux ball $(\R^{2n+1},dz-ydx)$,  see e.g. \cite[Proposition 2.1.8]{Geiges}, \color{black} any Legendrian in the sphere can be described by its associated front projections to the $(z,\mathbf{x})$-plane inside this Darboux ball. Recall that the front-projection recovers the Legendrian embedding. However, since the aforementioned contactomorphism is not strict,  finding the front of the Legendrian inside this  Darboux ball is not an easy task.
For that reason, we instead \blk take a different path, and produce an explicit contact isotopy of the Legendrian into a small contact-form preserving Darboux ball with respect to the prequantisation contact form $\alpha.$ One advantage of placing the Legendrian inside a small Darboux ball is that it makes its Legendrian contact homology as defined by Eliashberg--Givental--Hofer and Chekanov \cite{IntroToSFT, DiffGradedAlgebraLegLinks} computable by alluding to Ekholm's theory of gradient flow trees \cite{MorseFlowTrees}. Legendrian contact homology is a Legendrian invariant which is a homology of a differential graded algebra (DGA for short) called the {\bf Chekanov--Eliashberg algebra}; in the setting considered here the precise construction was carried out in \cite{LegendrianContactPxR} by Ekholm--Etnyre--Sullivan.

We give particular attention to the Legendrian link of the Harvey--Lawson cone from \cite[Example III.3.A]{HarveyLawson}. This is a conical special Lagrangian inside $\C^3,$ whose intersection with the standard contact sphere $S^5$ is a Legendrian torus whose projection to $\CP^2$ is a threefold cover of the monotone Clifford torus. This example was also studied in \cite[Example 4.1]{LiftingImmersions} from the perspective of contact geometry. After a Legendrian isotopy into a small contact Darboux ball, we obtain the Legendrian $\Lambda_{\OP{Cl}}$ whose front projection is shown in Figure \ref{fig:Clifford}; it is the symmetric figure-8 curve with two horizontal cusps rotated along its $z$-axis of symmetry. The Chekanov--Eliashberg algebra of $\Lambda_{\OP{Cl}}$ was computed by the first author in \cite{KnottedLegendrianSurface}; see Section \ref{sec:comp-cliff}.
\begin{figure}[t]
\vspace{3mm}
\centering
\labellist
\pinlabel $x$ at -5 13
\pinlabel $y$ at 72 43
\pinlabel $z$ at 25 93
\pinlabel $\Lambda_{\OP{Cl}}$ at 310 43
\endlabellist
\includegraphics{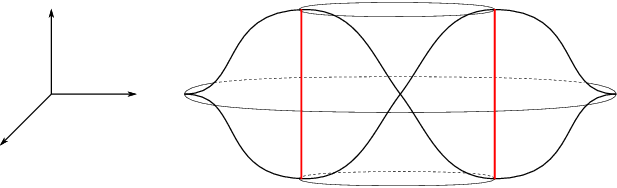}
\caption{Front projection of the Legendrian lift $\Lambda_{\OP{Cl}}$ of the threefold Bohr--Sommerfeld cover of the Clifford torus placed inside a Darboux ball.} \label{fig:Clifford}
\end{figure}
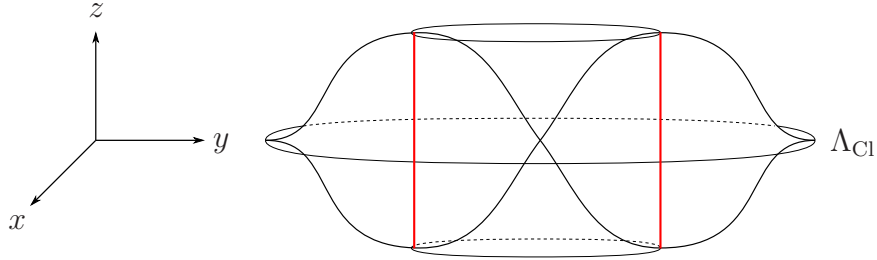
A different front associated to this Legendrian was described by Treumann--Zaslow in \cite[Section 3.1]{TreumannZaslow}, who placed the torus as a kind of satellite around the standard Legendrian sphere. The front inside the standard neighbourhood $J^1S^2$ of the standard Legendrian two-sphere then has a caustic which is the tetrahedron cubic graph on $S^2.$ Our front projection is simpler, at least in the sense that the number of Reeb chords for a small and generic perturbation is minimal for our representative. The computations of the augmentation varieties of $\Lambda_{\OP{Cl}}$ from either of \cite{KnottedLegendrianSurface,TreumannZaslow} show that this Legendrian torus, i.e.~the Legendrian which corresponds to the Harvey--Lawson cone, does not admit any exact Lagrangian filling.\blk

The second example that we study is the Legendrian lift of the threefold canonical Bohr--Sommerfeld cover of the monotone Chekanov torus inside $\CP^2$ \cite{Ch96}. The front projection of a representative $\Lambda_{\OP{Ch}}$ placed inside a small contact Darboux ball is shown in Figure \ref{fig:Chekanov} in Section \ref{sec:examples-plane}. Again the front is symmetric with respect to rotation around the $z$-axis.

Since the Legendrians $\Lambda_{\OP{Cl}}$ and $\Lambda_{\OP{Ch}}$ both satisfy a symmetry of their fronts with respect to rotation around the $z$-axis, they have Lagrangian projections that are well-behaved with respect to the standard Lefschetz fibration $\C^2 \to \C$ given by $(z_1,z_2) \mapsto z_1 \cdot z_2.$ (One must first translate the Legendrians to ensure that the axis of the $S^1$-symmetry is precisely $\{x_1=x_2=0\}.$) Namely, they project to generic immersions of a closed curve inside $\C^*$; see the curves $\gamma_{\OP{Ch}}$ and $\gamma_{\OP{Cl}}$ in the bottom right of Figures \ref{fig:homotopy-clifford} and \ref{fig:homotopy-chekanov}, respectively.  This symmetry of the Legendrians facilitates the task of finding the differentials of their Chekanov--Eliasbherg algebras. The reason is that the projection to $\C^*$ provides a kind of dimensional reduction. In the case when the Legendrian is an $S^1$-spun knot, this dimensional reduction was fully carried out by Ekholm--K\'{a}lm\'{a}n \cite{EkholmKalman}; the DGA of the spun torus is determined by the DGA of the knot. Unfortunately, our situation is more complicated, since the Legendrians are \emph{symmetric} $S^1$-spuns; see Remark \ref{rem:spun}.  \color{black}

We also compute the Chekanov--Eliashberg algebra of $\Lambda_{\OP{Ch}}$ and deduce that it is not Legendrian isotopic to $\Lambda_{\OP{Cl}}$; recall that Chekanov--Schlenk \cite{TwistTori} proved that the underlying Lagrangian torus also is not Hamiltonian isotopic to the Clifford torus inside $\CP^2.$ To conclude:
\begin{thm}[Theorem \ref{thm:frontsoflifts}]
The threefold canonical Bohr--Sommerfeld covers of the Clifford and Chekanov tori in $\CP^2$  have \color{black} embedded Legendrian lifts to the prequantisation space $S^5\to\CP^2$ that are not Legendrian isotopic. Moreover, both tori have vanishing Maslov classes and are subloose but not loose.
\end{thm}
We call a Legendrian subloose if there exists an exact Lagrangian cobordism to a loose Legendrian; see Definition \ref{def:subflex} below in Section \ref{sec:subflex}. Note that sublooseness in particular implies that the Legendrian admits no exact Lagrangian filling.  Subloose but not loose Legendrian submanifolds first appeared in the work of the first author \cite[Remark 2.1.5]{PhdThesis} (even though this terminology was not used there). \color{black}

\begin{rem}
\label{rem:subflexible}
The notion of subloose is intimately connected to the notion of a \emph{subflexible Weinstein domain} introduced in \cite{MurphySiegel} by Murphy--Siegel. In fact, in the course of producing the subflexible Weinstein domain in \cite[Theorem 3.7]{AffineFronts}, Casals--Murphy construct an example of a subloose Legendrian sphere whose front projection is resemblant of the front projection the torus in Figure \ref{fig:Clifford}; see \cite[Figure 32]{AffineFronts}. However, unlike the setting considered here, the Lagrangian cobordism from the non-loose sphere to the loose sphere considered there does not live in the symplectisation, but rather in a nontrivial symplectic cobordism.
\end{rem}
The subloose Legendrians play an important role for understanding the geography of Legendrian submanifolds. More precisely, the Chekanov-Eliashberg algebra of such submanifolds has an unusual behaviour: it  is acyclic with field coefficients, but is linearisable with the so-called Novikov coefficients. There are plenty of subloose Legendrians that can be obtained from the examples in \cite{KnottedLegendrianSurface}, e.g.~obtained by taking cusp-connect sums with appropriate families of Legendrian spheres. However, note that almost nothing is known when it comes to what augmentation varieties are realisable by such examples (in fact this is also the case for more general Legendrian tori).

We also consider the lift of the canonical twofold cover of the monotone product torus $S^1 \times S^1$ inside the monotone quadric surface $(\CP^1 \times \CP^1,\omega_{\OP{FS}}\oplus\omega_{\OP{FS}}).$ It turns out that its Legendrian lift is Legendrian isotopic to the conormal lift of the unknot inside $S^3$; see Theorem \ref{thm:unknot}. Recall that the conormal lift of any submanifold of $\R^n \subset S^n$ has a front projection associated to its representation inside the jet-space $(J^1\R^n,dz-ydx) \cong UT^*\R^n \subset (UT^*S^n,pdq)$ which can be recovered from the geometric properties of the smooth submanifold itself.

By the result of Vianna \cite{Vianna16} there are infinitely many different monotone Lagrangian tori inside $\CP^2$ up to Hamiltonian isotopy. In Section \ref{furtherdirections} we conjecture that the canonical Bohr--Sommerfeld covers of Vianna's different Lagrangian tori have Legendrian lifts in $S^5$ that remain different up to Legendrian isotopy.
\begin{rem}
This question is related to the question in \cite{AurouxInfinmanyLagrR6}, which asks whether the lifts of  Vianna's tori in $\CP^2$ to three-dimensional Lagrangian tori in $S^5 \subset (\C^3,\omega_0)$ remain in distinct Hamiltonian isotopy classes (as Lagrangians in $\C^3$). Note that these three-dimensional  Lagrangian tori can be obtained as the traces of our Legendrian tori under the periodic Reeb flow on $S^5$. 
\end{rem}
Vianna's infinite family of tori exhibit intricate Floer homological properties. As objects in the Fukaya category they constitute the infinite charts on the cluster variety that is the mirror to $\CP^2$ (in the sense of homological mirror symmetry); see recent work of Pascaleff--Tonkonog \cite{PascaleffTonkonog}. One should expect that this structure has a rich counterpart for also the Legendrian lifts of their Bohr--Sommerfeld covers. This paper is the starting point for such an investigation, since we compute the Legendrian invariants for the first two tori that appear in the family. Recall that the first torus in the family is the well-known Clifford torus, whose Bohr--Sommerfeld cover has $\Lambda_{\OP{Cl}}$ as its Legendrian lift. When computing the DGA of the latter Legendrian, we recover Nadler's computation \cite{Nadler} which exhibits the mirror of the one-dimensional pair of pants.

The strategy used here for studying the Legendrian lifts of the threefold covers of the Clifford and Chekanov tori should be possible to apply  \color{black}  to all of the infinitely many monotone Lagrangian tori produced by Vianna \cite{Vianna16}, even though the explicit isotopy into a Darboux ball is more complicated in these cases. However, once this has been found, note that the Legendrian contact homology is possible to compute for Legendrians inside a Darboux ball by using Ekholm's theory of gradient flow trees \cite{MorseFlowTrees}. In this manner interesting enumerative invariants for these Legendrians are thus possible to compute (at least in theory). In particular, in Section \ref{sec:potential} we formulate a conjecture relating the augmentation polynomial and the superpotential. Since the underlying Lagrangian live in pairwise different Hamiltonian isotopy classes, as detected by their different superpotentials, we expect that two different monotone tori also have different Legendrian lifts. We expect that this family of subloose Legendrian tori exhibits interesting properties reflecting the invariants of Vianna's family. For instance, their augmentations varieties should   be related to the superpotentials of the corresponding tori in the following manner: the zero-locus of the superpotential is a suitable threefold unbranched cover of the augmentation variety.

In particular, in view of Vianna's infinite family of monotone tori, it seems believable that there is a very rich family of augmentation varieties associated to different Legendrian tori inside the standard contact sphere, with relations to the infinite number of cluster charts associated to these Lagrangian tori. We plan to study these questions in future investigations.
\blk

Along the way, we also provide the following reasonably general computation:
\begin{thm}[Theorem \ref{thm:lchgeneral}]
\label{lchgeneralintro}
 Let $\varepsilon_0,\varepsilon_1 \colon (\mathcal{A},\partial) \to (\F,0)$ be two graded augmentations for a Legendrian oriented genus $g\geq 0$ surface $\Lambda \subset (\R^5,\xi_{\OP{st}})$ of vanishing Maslov class and with all Reeb chords in positive degrees. Then
\begin{itemize}
\item
when  $\varepsilon_0 = \varepsilon_1$: $$LCH_k^{\varepsilon_0,\varepsilon_1}(\Lambda)=\begin{cases} \F, & k=2,\\ \F^{g}, & k=1, \\ 0, & k\neq 1,2,
\end{cases}$$
\item when $0\leq g\leq 1$ and $\varepsilon_0 \neq \varepsilon_1$: $LCH_k^{\varepsilon_0,\varepsilon_1}(\Lambda) = 0$ for all $k,$
\end{itemize}
is satisfied for the bilinearised Legendrian contact homology groups.
\end{thm}
(In other words, in the case $g=1,$ the $LCH_*$ groups for augmentations behave as the $\OP{Ext}$-groups for skyscraper sheaves on an algebraic curve under the above assumptions.) In addition, we obtain restrictions on the variety of augmentations of a Legendrian torus that satisfies the assumptions of Theorem \ref{thm:lchgeneral}.
\begin{thm}[Theorem \ref{thm:augvar}]
Let $\Lambda \subset (\R^5,\xi_{\OP{st}})$ be a Legendrian torus of vanishing Maslov class and with all Reeb chords in positive degrees. Its augmentation variety over $\C$ is then either empty,  or cut out by a single polynomial. \color{black}
\end{thm}
{The restriction on the degree of Reeb chords given by Theorem \ref{thm:bscovers} implies that the latter two theorems are applicable} to the Legendrian lift of the canonical Bohr--Sommerfeld cover of any monotone Lagrangian torus in $\CP^2.$

Then in Section \ref{sec:subflex} we construct an  infinite family of non-regular exact Lagrangian caps in the symplectisation of the contact Euclidean space of all odd dimensions $\geq 5$. This provides a partial negative answer to the question of Eliashberg--Ganatra--Lazarev \cite{FlexibleLagrangians, EliashbergSurvey}; see Question \ref{questionfromGanatra-Eliashberg-Lazarev}.

\begin{thm}[Theorem \ref{NonRegularCapsofSurfaces}]
  For any $g>0$ and $k_1,\ldots,k_r \ge 0$ there exists infinitely many different Legendrian isotopy classes of subloose Legendrian embeddings $\Lambda \subset (\R^{2k_1+\ldots+2k_r+5},\xi_{\OP{st}})$ of the manifold $S^{k_1} \times \ldots \times S^{k_r} \times \Sigma_g,$ where $\Sigma_g$ denotes the surface of genus $g,$ which:
  \begin{itemize}
  \item have vanishing Maslov classes;
  \item have Chekanov--Eliashberg algebras with Novikov coefficients $R=\C[H_1(\Lambda)]$ that admit (0-graded) augmentations; and
  \item admit orientable exact Lagrangian caps inside the symplectisation with vanishing Maslov classes.
  \end{itemize}
In particular, none of these Legendrian surfaces are loose.
\end{thm}
In addition, we provide a refinement of the question of Eliashberg--Ganatra--Lazarev, see Question \ref{refinedquestionfromGanatra-Eliashberg-Lazarev}.

\subsection{Criterion for looseness in open books}
We end with some discussions about criteria for a Legendrian to be stabilised, i.e.~to admit a Darboux chart in which the Legendrian coincides with a standard stabilisation. In the case when the Legendrian has nonempty boundary, this Darboux chart is by definition disjoint from the boundary. Recall that, in the case of a contact manifold of dimension at least five, a stabilised Legendrian is the same as a loose Legendrian.

The Legendrian disc admits a stabilisation in arbitrary dimension, as was shown in \cite[Remark 7.22(2)]{SteinWeinstein} by Cieliebak--Eliashberg. In particular, when the contact manifold has dimension greater than three, any Legendrian disc is loose. Since the Legendrian isotopy class of a Legendrian with boundary is preserved under the Legendrian version of boundary connected sum with a disc, we can now immediately conclude that
\begin{prop}
\label{prop:openloose}
Any Legendrian with nonempty boundary admits a stabilisation; in particular, in a contact manifold of dimension strictly greater than three, such a Legendrian is loose.
\end{prop}

In Section \ref{sec:disc} we investigate the case of the ``standard Legendrian disc''
$$\Lambda_0 \coloneqq \{x_{n+1} \ge 0\} \cap \Re \C^{n+1} \cap S^{2n+1} \subset (S^{2n+1},\alpha_{\OP{st}})$$
which arises as a subset of the standard Legendrian sphere. Namely, we construct an explicit Legendrian isotopy which is fixed near the boundary after which an explicit chart which contains a stabilisation can be seen. To describe this Legendrian isotopy we again use the standard Lefschetz fibration $\C^n \to \C$.

The case of the standard Legendrian disc inside the sphere is interesting, since it provides a new proof of Casals--Murphy result from \cite[Proposition 2.9]{AffineFronts} or, rather, the mild generalisation of it which is stated in Proposition \ref{prop:CM} below.

Proposition \ref{prop:CM} provides a criterion for looseness in terms of the relation between the Legendrian and an open book decomposition; we refer to \cite{Giroux} for the definition of an open book in the setting of contact geometry. Our setup is as follows: Assume that a contact manifold $(Y,\xi)$ admits an open book decomposition with page given by a Weinstein domain $(W,d\eta)$, and that $W$ contains a standard critical Weinstein handle
$$ (DT^*B^n,2p\,dq+q\,dp) \hookrightarrow (W,d\eta),$$
whose boundary satisfies
$$ST^*B^n=\partial(DT^*B^n)  \hookrightarrow \partial W,$$
along which the monodromy of the open book is trivial. The triviality of the monodromy on this subset of the page implies that there exists a contact embedding
$$\partial(D^2 \times DT^*B^n) \subset (Y,\xi),$$
where the contact structure on the former is induced by the restriction of the Liouville form on the domain
$$ (D^2 \times DT^*B^n,r^2\,d\theta\oplus p\,dq)$$
to its contact boundary (after smoothing the corners). Note that $\partial(D^2 \times DT^*B^n)$ is an \emph{open} contact manifold which admits the structure of an open book with trivial monodromy, whose page is a standard critical Weinstein handle (considered as an \emph{open} symplectic manifold with non-empty boundary where, as usual, the boundary becomes the binding of the open book).
\begin{prop}[Casals--Murphy]
\label{prop:CM}
Consider a Legendrian submanifold $\Lambda \subset (Y,\xi)$ of a contact open book of the above form. If this Legendrian is contained in a single page when intersected with the subset $\partial(D^2 \times DT^*B^n) \subset Y$, and if it moreover coincides with the core disc inside that page, i.e.~if
$$\Lambda \cap \partial(D^2 \times DT^*B^n) = \{1\} \times 0_{B^n} \subset Y$$
is satisfied, then $\Lambda$ admits a stabilisation.

In particular, when $Y=\partial(D^2 \times W)$ is the trivial open book with page a Weinstein domain $W$, any Legendrian which is contained inside a single page $\{1\} \times W \subset Y$, where it moreover is assumed to intersect some cocore of $W$ transversely in a single point, admits a stabilisation.\blk
\end{prop}
\begin{proof}
The open contact manifold $\partial(D^2 \times DT^*B^n)$ can be identified with $(S^{2n+1} \setminus \partial \Lambda_0,\xi_{\OP{st}})$, under which $\{1\} \times 0_{B^n}$ can be assumed to be identified with the standard open Legendrian ball $\Lambda_0 \setminus \partial \Lambda_0$. The first part of the result is then a direct consequence of Proposition \ref{prop:openloose}.

For the last claim, one just needs the following standard fact. Assume that a Lagrangian intersects the cocore of a handle transversely in a single point. After a Hamiltonian isotopy supported near the cocore, the Lagrangian can then be assumed to coincide with the core disc of the same handle inside some small neighbourhood of the cocore.\blk
\end{proof}
\blk

\section*{Acknowledgements}
This project started when the authors visited the University of Ottawa during the workshop on ``Lagrangian cobordisms,''  and we are grateful to the organisers of the workshop and the University of Ottawa for their hospitality. We would also like to thank Emmy Murphy who told us the expectation of the relation between the Clifford torus and the front considered here, Otto van Koert for useful discussions and Fr\'{e}d\'{e}ric Bourgeois for helpful  comments  on  the  first  version  of  the  paper  which led to the correction of Theorem \ref{lchgeneralintro}.
Finally, we would like to thank both anonymous referees for suggesting quite a few useful improvements, and also for pointing out inaccuracies and mistakes.  \color{black}

The first author is supported by the grant KAW 2016.0198 from the Knut and Alice Wallenberg Foundation.
The second author is supported by the ERC Consolidator Grant 646649 ``SymplecticEinstein'', GA\v{C}R EXPRO Grant 19-28628X and 4EU+/20/F3/21 Minigrant.

\section{Background}

 Here we provide recollections of different geometric notions which are crucial to this work, as well as to the Chekanov--Eliashberg algebra, which is a modern Legendrian invariant. \color{black}

\subsection{Prequantisation spaces and  Bohr--Sommerfeld Lagrangians}

A {\bf prequantisation space} is a contact manifold $(E,\alpha)$
equipped with a contact form $\alpha$ whose Reeb flow defines a free
$S^1$-action. It follows that $(M=E/S^1,\omega=d\alpha)$
naturally is a symplectic manifold via symplectic reduction. Phrased
differently, the $S^1$-bundle $\pi \colon E \to M$ has a connection 1-form
$\alpha$ with a curvature 2-form given by $\omega$. Recall that all prequantisation spaces are strongly symplectic fillable by a symplectic $D^2$-bundle over $(M,\omega)$. (In general, however, they are not fillable by exact symplectic manifolds.) \color{black}

We fix the convention that the lengths of the corresponding simple
Reeb orbits on $(E,\alpha)$ all are equal to $\pi,$ i.e.~the fibre is canonically identified with $S^1=\R/\pi\Z$ when parametrised by the Reeb flow. In this case the latter symplectic form satisfies
\begin{equation}
\label{eq:prequant}
[\omega] =\pi\cdot c_1(E) \in \pi\cdot \im(H^2(M,\Z) \to H^2(M,\R))
\end{equation}
where $c_1(E)$ denotes the first Chern class of the $S^1$-bundle.

It is straight-forward from the definitions that an $n$-dimensional immersion $\iota \colon L \looparrowright (M,\omega)$ is Lagrangian if and only if the pullback connection 1-form $\iota^*\alpha$ on the pullback bundle $\iota^*E \to L$ is flat.\blk

\begin{defn}
\label{defn:BS}
In the above setting, a Lagrangian immersion \blk $\iota \colon L \looparrowright (M,\omega)$ satisfies the {\em Bohr--Sommerfeld condition} for the prequantisation $S^1$-bundle $(E,\alpha) \to (M,\omega)$ if the pullback connection 1-form $\iota^*\alpha$ on the pullback bundle $\iota^*E \to L$ is trivial.
\end{defn}

\begin{rem}
\label{rem:sympaction}
In the case when $M$ is simply connected, the above property translates to the condition that any $\eta \in \iota_*(H_1(L)) \subset H_1(\iota(L))$ bounds integral symplectic area; by this we mean that
$$\int_u \omega \in \pi \cdot \Z$$
is satisfied for any $u \in H_2(M,\iota(L))$ on which the connecting homomorphism takes the value $\delta(u)=\eta$.

Furthermore, by the integrality assumption in Formula \eqref{eq:prequant}, in order to determine whether all $u \in H_2(M,\iota(L))$ with $\delta(u)=\eta$ have integral symplectic area, it suffices to verify the property for a single such class $u$.
\end{rem}
\blk When it is clear to which prequantisation bundle we are referring, we will in the following simply say that a Lagrangian immersion is Bohr--Sommerfeld.

The following lemma is immediate from the definition.
\begin{lem}
  \label{lem:prequant}
Any closed immersed Legendrian submanifold $\Lambda \subset (E,\alpha)$ projects to a Bohr--Sommerfeld Lagrangian immersion $\pi(\Lambda) \subset (M,\omega)$ and, conversely, any Bohr--Sommerfeld immersion lifts under $\pi$ to a Legendrian immersion.  For a Bohr--Sommerfeld immersion of a connected manifold, the lift is uniquely determined up to a global application of the Reeb flow.
\end{lem}
\begin{proof}
Since the connection is flat along any Lagrangian submanifold (the curvature vanishes by the Lagrangian condition), the existence of a lift is equivalent to the triviality of the monodromy. Recall that flat connections are classified by their monodromy.
\end{proof}

It is immediate that the Legendrian lift of a Bohr--Sommerfeld immersion $\iota \colon L \looparrowright (M,\omega)$ of a connected manifold $L$ is an embedding whenever it satisfies the following additional property: the connection 1-form $\alpha$ has a \emph{non-trivial} monodromy along any path of the form $\iota(\gamma) \subset M$, where $\gamma$ is a path in $L$ which connects two different preimages of a double point of $\iota$.
\begin{rem}
In the case when $M$ is simply connected, this property can be rephrased as the requirement that a disc in $\pi_2(M,\iota(L))$ whose boundary lifts to a path in $L$ which connects two different preimages of a double point of $\iota$ has symplectic area which does not take any of the values $\pi\cdot\Z$.
\end{rem}
It can be seen that a small generic perturbation through Bohr--Sommerfeld Lagrangian immersions makes the lift embedded. \color{black}

Even though Bohr--Sommerfeld immersions generically have embedded Legendrian lifts, it is important to note that the appearance of intersections in one-parameter families of Legendrian lifts is a generic phenomenon, and in general one can not expect to get rid of them.

Recall that the integrality condition on the symplectic form given in Formula \eqref{eq:prequant} is sufficient to guarantee the existence of a prequantisation bundle with a connection one-form whose curvature is equal to precisely that symplectic form. The isomorphism class of the bundle depends on the choice of integer lift of the integral symplectic form; when there is no torsion this lift is unique, and so is the bundle. For simplicity we will in the following assume that $\pi_1(M)=0.$ In this case the connection one-form is uniquely determined by the integral symplectic form itself up to isomorphism. \color{black}

The prequantisation spaces that we will study here are the following.
\begin{example}
\label{ex:prequant}
\begin{enumerate}
\item The standard round contact sphere
$$ S^{2n-1} \subset \C^n $$ equipped with the coordinates $(z_1=x_1+iy_1,\dots,z_n=x_n+iy_n)$
and the contact form
$$ \alpha_{\OP{st}} := \frac{1}{2}\sum_{i=1}^n(x_idy_i-y_idx_i).$$
The corresponding symplectic reduction is the quotient under the Hopf map, which gives $M=\CP^{n-1}$ equipped with the {\bf Fubini--Study K\"{a}hler form} $\omega_{\mathrm{FS}}$
for which the symplectic area of a line has been normalised to
$\int_\ell \omega_{\mathrm{FS}}=\pi.$
This is also the length of a minimal periodic Reeb orbit for the above contact form.
\item The unit cotangent bundle $(UT^*S^n,pdq)$ of the round $n$-sphere of radius $1/2$ has a completely periodic Reeb flow with minimal length of a periodic orbit equal to $\pi$.
The symplectic reduction gives $UT^*S^n$ the structure of a prequantisation space over the monotone projective quadric $(n-1)$-fold. In particular, when $n=3,$ this unit cotangent bundle is a prequantisation space over the monotone quadric surface $(\CP^1 \times \CP^1,\omega_{\mathrm{FS}} \oplus \omega_{\mathrm{FS}})$.
\end{enumerate}
\end{example}

The contact structures of the prequantisation spaces in the above examples both have vanishing first Chern classes $c_1(E,\alpha) \in H^2(E).$ This is of course obvious for the sphere, while it is also a general fact for all unit cotangent bundles. For completeness we also recall the following standard result: simply connected prequantisation spaces over monotone symplectic manifolds (the above prequantisation spaces are all of this type) have first Chern classes $c_1(E,\alpha)$ which are torsion; see Lemma \ref{lem:chern}.  \blk

Here follows basic examples of Lagrangian Bohr--Sommerfeld immersions inside the projective space and quadric.
 \color{black}

 \begin{example}
  \begin{enumerate}
\item The standard Legendrian sphere $\Lambda_{\OP{std}} \subset (S^{2n+1},\alpha_{\OP{st}})$ is the intersection of the standard contact sphere $S^{2n+1}$ and the real part $\Re\C^{n+1}.$ This Legendrian is the lift of the twofold cover of $\RP^n \subset \CP^n,$ and this two-fold cover is thus is a Bohr--Sommerfeld immersion. It can be readily seen to be Legendrian isotopic to a representative inside a contact Darboux ball which has a rotationally symmetric front projection given by the ``flying saucer'' with precisely one Reeb chord, and singularities consisting of a spherical cusp edge. The Lagrangian projection of this representative is the so-called (exact Lagrangian) Whitney immersion of a sphere with a single transverse double point.
\item The anti-diagonal
$$\{(z,\overline{z})\} \subset (\CP^1 \times \CP^1,\omega_{\mathrm{FS}} \oplus \omega_{\mathrm{FS}})$$
is an embedded Lagrangian sphere inside the monotone projective quadratic surface, and it is thus automatically Bohr--Sommerfeld.

More generally, all monotone projective quadrics contain Lagrangian spheres. Recall that the affine complex quadric of dimension $n-1$ is symplectomorphic to $(T^*S^{n-1},d(pdq))$  and thus it contains an embedded Lagrangian sphere: the zero section.  The affine quadric sits inside the projective quadric as the complement of a divisor (a quadric of one dimension less), and thus we get a Lagrangian sphere in the projective quadric as well. As described in Part (2) of Example \ref{ex:prequant}, there is a prequantisation bundle over the $n-1$--dimensional projective quadric with total space $UT^*S^n$. The Legendrian lift to $UT^*S^n$ of the aforementioned Lagrangian sphere (i.e.~the zero-section in the affine part)  \color{black}  can be readily seen to be Legendrian isotopic to a unit cotangent fibre $UT^*_{\OP{pt}}S^n.$
\end{enumerate}
\end{example}  \color{black}

\subsection{The Chekanov--Eliashberg algebra}
Here we give a brief recollection of the Chekanov--Eliashberg algebra $\mathcal{A}(\Lambda)$ of a Legendrian $\Lambda$ in the contact manifold $(\C^n \times \R,dz-ydx)$ as developed in \cite{LegendrianContactPxR}. More precisely, the algebraic formalism from \cite{KnotContHom}   will be used, which there was called the ``fully noncommutative Legendrian DGA.'' Roughly speaking, the noncommutativity refers to the fact that, unlike in the original definition of the invariant, the subalgebra $R=\F[H_1(\Lambda)]$ of ``Novikov coefficients'' is non-central by construction (i.e.~it does not commute with the Reeb chord generators); also c.f.~\cite{Noncomm}.
\begin{rem} Recall that we mainly are interested in the case when $\Lambda$ is a torus. Since the algebra $\F[H_1(\Lambda)]$ then is quasi-isomorphic to the DGA of chains on the based loop space, in this case the fully noncommutative Legendrian DGA computes the partially wrapped Fukaya category of the ball $B^4$ with $\Lambda \subset \partial B^4$ used as a stop; see recent work \cite{EkholmLekili} by Ekholm--Lekili.
\end{rem}

We proceed to give the precise construction of the algebra \color{black} . Let $\Lambda \subset \C^n \times \R$ be a closed Legendrian submanifold with a set $\mathcal{Q}$ of Reeb chords, which are assumed to be finite.
We can thus consider the free graded $R$-bimodule $A$ generated by $\mathcal{Q},$ where the grading is induced by the Conley--Zehnder index as in \cite{LegendrianContactPxR}. The underlying unital algebra of the Chekanov--Eliashberg DGA in our setting is the tensor ring
$$\mathcal{A}(\Lambda) \coloneqq \bigoplus_{k \ge 0} A^{\otimes_R k}$$
where
\begin{align*}
& A^{\otimes_R 0}=R,\\
& A^{\otimes_R k} \coloneqq \underbrace{A \otimes_R \cdots \otimes_R A}_k, \:\: k \ge 1,
\end{align*}
all are $R$-bimodules. The contributions from the homotopy classes of the pseudoholomorphic discs in the definition of the differential $\partial$ is then determined by auxiliary choices of \emph{capping paths} from each of the two endpoints of every Reeb chord to a fixed based point $\star \in \Lambda.$ We refer to \cite{KnotContHom,Noncomm} for more details.

Recall that an {\bf augmentation} is a unital DGA-morphism
$$ \varepsilon \colon \mathcal{A}(\Lambda) \to \F$$
which thus satisfies $\varepsilon \circ \partial=0.$ Here we will only consider {\bf graded augmentations} which by definition vanish on all generators in nonzero degrees. Observe that an augmentation restricts to a unital algebra map
$$ \varepsilon \colon \F[H_1(\Lambda)] \to \F,$$
which can be identified with a local system in $\F$ when $\pi_1(\Lambda)$ is abelian. Following \cite{FramedKnotContactHomology} we define the {\bf augmentation variety} to be the Zariski closure of the set of points in $\OP{Sp}(\F[H_1(\Lambda)]),$ i.e.~unital $\F$-algebra maps   $\F[H_1(\Lambda)] \to \F$,  \color{black}  which extend to an augmentation via the canonical inclusion $\F[H_1(\Lambda)] \subset \mathcal{A}(\Lambda).$

\begin{rem} While the set of augmentations naturally has the structure of a (possibly non-reduced) algebraic variety -- this is the so called ``total augmentation variety'' -- this is not necessarily the case for the set of algebra maps $\F[H_1(\Lambda)] \to \F$ which \emph{admit extensions} to an augmentation. The latter is the image of the full augmentation variety under a projection map, and since the image under such a projection is not necessarily itself a variety, we are forced to take the Zariski closure above to ensure that we obtain an algebraic subset.
\end{rem}

Given a pair of augmentations, Bourgeois and Chantraine defined the {\bf bilinearised Legendrian contact homology} in \cite{BilinearisedLCH}, which is a chain complex with underlying vector space   $\F^\mathcal{Q}$.  \color{black}  In the setting of the fully noncommutative Legendrian DGA we refer to \cite{Noncomm} for more details.

\subsection{Lagrangian cobordisms and subloose Legendrian submanifolds}
\label{sec:sublfex}
The class of loose Legendrian submanifolds of dimension two and more was introduced by Murphy in \cite{LooseLegendrianEmbeddings}. By definition a Legendrian is loose if one can find a so-called \emph{loose chart}, which is a contact Darboux ball where the Legendrian is in a particular position. In the aforementioned article an h-principle was established for this class of Legendrians, which implies that their classification up to Legendrian isotopy is determined by their formal Legendrian isotopy classes.

The Chekanov--Eliashberg algebra of a loose Legendrian can be computed, in some suitable standard model, to be acyclic. This is equivalent to the unit being a boundary. Note that an acyclic DGA admits no augmentations.

\begin{defn}
Given two closed Legendrian submanifolds $\Lambda_{-}$ and $\Lambda_{+}$ of a contact manifold $(M,\alpha)$.
An {\em exact Lagrangian cobordism from $\Lambda_-$ to $\Lambda_+$}  is a properly embedded submanifold $L \subset (\R \times M, d(e^t \alpha))$ in the symplectisation such that for some $T>0$
\begin{itemize}
 \item[(i)] $L\cap (- \infty, -T) \times M= (-\infty, -T)\times \Lambda_{-}$ and  $L\cap (T, +\infty)\times M = (T,+\infty)\times \Lambda_{+}$,
 \item[(ii)] $L \cap [-T,T] \times M$ is compact.
\item[(iii)]  there is a function $f_L\in C^{\infty}(L)$ such that
\begin{itemize}
\item $e^{t}\alpha|_{TL} = df_L$,
\item $f_L|_{(-\infty,-T)\times \Lambda_{-}}$, $f_L|_{(T,+\infty)\times \Lambda_{+}}$ are constant functions.
\end{itemize}
\end{itemize}
We call $(T,+\infty)\times \Lambda_+$ and $(-\infty,-T) \times \Lambda_-$ the {\em positive end} and the {\em negative end} of $L$, respectively.
\end{defn}

Recall the following fact established in \cite{RationalSFT}: an exact Lagrangian cobordism $L$ from $\Lambda_-$ to $\Lambda_+$ gives rise to a unital DGA morphism
$$ \Phi \colon (\mathcal{A}(\Lambda_+),\partial_+) \to (\mathcal{A}(\Lambda_-),\partial_-)$$
defined by an appropriate count of rigid pseudoholomorphic discs with boundary on the cobordism. Even though the result of \cite{RationalSFT} is written for $\Z_2$-coefficients only, it admits a natural extension to more general coefficients such as the fully noncommutative DGA with ``Novikov coefficients'' considered here. One caveat is that the Novikov coefficients must be taken from the cobordism, i.e.~we must take $R=\mathbb F[H_1(L)]$, in the definition of the Chekanov--Eliashberg algebras of both $\Lambda_\pm$; see e.g.~\cite[Section 8.1]{Cthulhu} for more details. Observe that there are geometrically induced unital algebra maps $\F[H_1(\Lambda_\pm)] \to \F[H_1(L)]$, and that the DGA with coefficients in $\F[H_1(L)]$ is determined by the DGA with coefficients in $\F[H_1(\Lambda_\pm)]$ via the obvious extensions of these algebra maps to the entire DGAs.
 \color{black}

When the field $\F$ has characteristic different from two, we need to fix the additional data of a \emph{spin structure} on the Legendrians $\Lambda_\pm$ in order to be able to define their Chekanov--Eliashberg algebras with these coefficients; see \cite{Ekholm:Orientations}. Additionally, in order to make the cobordism map well-defined in these characteristics, we must also fix a spin structure on the cobordism which restricts to the chosen spin structures on its ends. See recent work \cite{Karlsson:Cob} by Karlsson.
 \color{black}

In particular, note that an augmentation of $\Lambda_-$ can be \emph{pulled back} to an augmentation of $\Lambda_+$,  with the caveat that coefficients and spin structures must be chosen in the manner described above. \color{black}  In the case of an exact Lagrangian {\bf filling} of $\Lambda,$ i.e.~an exact cobordism with $\Lambda_-=\emptyset$ and $\Lambda_+=\Lambda,$ we obtained an augmentation of $\mathcal{A}(\Lambda)$ with coefficients in $R=\F.$

\begin{defn}
\label{def:subflex}
A Legendrian submanifold $\Lambda$ is said to be {\em subloose} if there exists an exact Lagrangian cobordism from $\Lambda$ to a loose Legendrian submanifold.
\end{defn}
\begin{remark}
Observe that according to the definition of a subloose Legendrian, every loose Legendrian submanifold  is subloose (taking a trivial exact Lagrangian cobordism $\R\times \Lambda$ of a loose Legendrian $\Lambda$, we see that $\Lambda$ is subloose).
\end{remark}

\begin{prop}
\label{sub-looseacyclic}

A subloose Legendrian has an acyclic Chekanov--Eliashberg algebra with coefficients in $R=\Z_2,$ i.e.~when Novikov coefficients are not used. In particular, it admits no exact Lagrangian fillings.

If there exists a spin structure on the subloose Legendrian which extends to a spin structure on some Lagrangian cobordism to a loose Legendrian, then the Chekanov--Eliashberg algebra with coefficients in a field $R=\F$ of arbitrary characteristic is acyclic as well.
\end{prop}
\begin{proof}
Consider the map in homology that is induced by the DGA morphism associated to the exact Lagrangian cobordism. This becomes a map from the trivial ring to the homology of the Chekanov--Eliashberg algebra of the subloose Legendrian. (By the above discussion it is important that we do not use Novikov coefficients here.) The result then follows from the following purely algebraic fact: if some given ring admits a unital map from the trivial ring, i.e.~a ring for which $1=0$, then the target ring must be trivial as well.
\end{proof}

In Section \ref{sec:subflex} we show that the Legendrian tori $\Lambda_{\OP{Cl}}$ and $\Lambda_{\OP{Ch}}$ both are subloose. This, however, does not mean that their Chekanov--Eliashberg algebras are uninteresting. As computed in Section \ref{sec:dga}, they even admit augmentations when Novikov coefficients are used;    in particular, they are not loose themselves.  \color{black}    Meanwhile, a priori they must have acyclic DGAs when coefficients $R=\F$ are used (i.e.~without the Novikov parameter) in view of Proposition \ref{sub-looseacyclic}. \color{black}

\section{Lagrangian isotopies in hypersurfaces}
\label{sec:lefschetz}

In this section we show how certain hypersurfaces of symplectic manifolds can be used as a tool for constructing Lagrangians as well as Lagrangian isotopies, by making the Lagrangians confined to the hypersurface. The point is that that Lagrangians which are contained inside well-behaved hypersurfaces can be understood via different forms of dimensional reduction.

Recall that the symplectic form restricted to a hypersurface $\Sigma \subset (M,\omega)$ of a symplectic manifold has a one-dimensional kernel $\ker (\omega|_{T\Sigma}) \subset T\Sigma$ called the {\bf characteristic distribution}. Any Lagrangian submanifold $L \subset \Sigma$ which is contained entirely in the hypersurface must be tangent to the characteristic distribution.  We will focus on the following two different techniques for studying such Lagrangians:

\begin{itemize}
\item[($\alpha$)]In the case when characteristic distribution integrates to an action on $\Sigma$ by the group $S^1$, symplectic reduction produces a fibre bundle
$$S^1 \to \Sigma \xrightarrow{f} \Sigma/S^1$$
and the Lagrangian submanifold $L \subset \Sigma$ projects to a Lagrangian of dimension $\dim L -1$ inside the symplectic base $\Sigma/S^1$ (the base is endowed with the symplectic structure induced by the symplectic reduction). Conversely, any Lagrangian in the base produces a Lagrangian inside $\Sigma$ by taking the full preimage under $f$.
\item[($\beta$)] In the case when there is a symplectic fibration $g \colon \Sigma \to S^1$, i.e.~a smooth fibre bundle for which the fibres are symplectic, then the Lagrangian submanifold $L \subset \Sigma$ can be uniquely recovered by its intersection with any given fibre; the latter is a Lagrangian submanifold of the fibre of dimension $\dim L -1$, and its intersection with any other fibre is determined by parallel transport along the characteristic distribution. Conversely, any Lagrangian in the fibre gives rise to a Lagrangian inside $\Sigma$ by taking its trace under the parallel transport.
\end{itemize}

For our purposes we will mainly be interested in a particular smooth, properly embedded real codimension-one hypersurface
$$\Sigma_{(1,1),1}^{(1,1)} \subset \CP^2 \setminus \{[0:0:1]\}$$
whose affine part
$$\Sigma_{(1,1),1}^o = \Sigma_{(1,1),1}^{(1,1)} \cap \C^2$$
is a cone, and where $\Sigma_{(1,1),1}^{(1,1)}$ intersects the line at infinity $\ell_\infty$ in an embedded closed curve $$S_{(1,1),1}^{(1,1)} = \Sigma_{(1,1),1}^{(1,1)} \cap \ell_\infty \cong S^1.$$
We carry out the construction in the subsequent paragraphs. All of our Lagrangian isotopies will take place inside of $\Sigma_{(1,1),1}^{(1,1)}$. Inside the affine part $\Sigma_{(1,1),1}^o$ we will use the symplectic reduction approach ($\alpha$) to understand our Lagrangian submanifolds. Since this symplectic reduction becomes singular near the line $\ell_\infty \subset \CP^2$ at infinity, approach ($\beta$) based upon a symplectic fibration will instead be used there.

First we construct the hypersurface; this will be done in a slightly greater generality than needed for our applications. For any fixed
$$r>0 \:\:\: \text{and} \:\:\: \mathbf{a}=(a_1,a_2)\in\Z_{>0}^2,$$
where the latter is a primitive integer vector, we start by constructing the three-dimensional hypersurface with boundary
$$ \Sigma_{\mathbf{a},r} \coloneqq \{s\cdot (e^{i\theta_1} \sqrt{a_1}, e^{i\theta_2}\sqrt{a_2}); \: s\in \R_{>0}, \theta_1,\theta_2 \in \R\} \cap D^4_r \subset (\C^*)^2 \cap D^4_r$$
inside the closed 4-disc. Note that this hypersurface is properly embedded, conical, and smooth with boundary inside $D^4_r \setminus \{0\}$; furthermore, it is endowed with an obvious foliation by punctured pseudoholomorphic lines. In addition it is foliated by the Lagrangian product tori
$$S^1_{s\sqrt{a_1}} \times S^1_{s\sqrt{a_2}} \:\:\: \text{ for all } s \in \left(0,r/\sqrt{a_1+a_2}\right].$$
In particular, the boundary
$$\partial \Sigma_{\mathbf{a},r}=S^1_{r\sqrt{a_1/(a_1+a_2)}} \times S^1_{r\sqrt{a_2/(a_1+a_2)}}$$
is itself a Lagrangian torus.

Recall that the closed symplectic manifold $\CP^2$ can be obtained from $D^4$ by performing a symplectic reduction $\partial D^4 \to \CP^1$ of its boundary. Observe that the hypersurface with boundary
$$(\Sigma_{(1,1),1},\partial \Sigma_{(1,1),1}) \subset (D^4,\partial D^4)$$
has an image inside $\CP^2$ which again is a smooth hypersurface, but where the latter has empty boundary (the symplectic reduction collapses the torus boundary to an embedding of $S^1$).

We generalise this symplectic reduction in the following manner. Start with the hypersurface $\Sigma_{\mathbf{a},r}$ with non-empty boundary. Then fix the choice of a vector
$$\mathbf{n}=(n_1,n_2) \in \Z_{> 0}^2$$
subject to the condition $n_1a_1+n_2a_2\neq0$ where, in addition, we assume that $\mathbf{n}$ is a primitive integral vector. Then consider the domain
$$H^{\mathbf{n}}_{\mathbf{a},r}=\left\{ |z_1|^2/n_2+|z_2|^2/n_1 \le \frac{r^2}{a_1+a_2}(a_1/n_2+a_2/n_1); \: z_i \in \C^*\right\} \subset (\C^*)^2$$
for which we obtain the inclusion
$$(\Sigma_{\mathbf{a},r},\partial \Sigma_{\mathbf{a},r}) \subset (H^{\mathbf{n}}_{\mathbf{a},r},\partial H^{\mathbf{n}}_{\mathbf{a},r}).$$
Performing a symplectic reduction on the boundary $\partial H^{\mathbf{n}}_{\mathbf{a},r}$ of the domain again produces a 4-dimensional symplectic manifold, and the image of $\Sigma_{\mathbf{a},r}$ inside the latter is again a smooth hypersurface without boundary that we in the following denote by $\Sigma_{\mathbf{a},r}^{\mathbf{n}}$.

\begin{example}
Our sought hypersurface is $\Sigma_{(1,1),1}^{(1,1)}$. Note that this hypersurface naturally lives inside $\CP^2$, since it is the image of $\Sigma_{(1,1),1}$ under the symplectic reduction $D^4 \to \CP^2$ of the boundary $\partial D^4=S^3$; see Figure \ref{fig:moment_polytope}. \color{black}
\end{example}

We proceed to investigate the hypersurface $\Sigma_{\mathbf{a},r}^{\mathbf{n}}$ obtained after collapsing the boundary of $\Sigma_{\mathbf{a},r}$ under the symplectic reduction. Recall that the boundary $\partial \Sigma_{\mathbf{a},r}$ is a Lagrangian torus. The manifold $\Sigma_{\mathbf{a},r}^{\mathbf{n}}$ is obtained from $\Sigma_{\mathbf{a},r}$ by taking a quotient of its torus boundary
$$\partial\Sigma_{\mathbf{a},r}=S^1_{r\sqrt{a_1/(a_1+a_2)}} \times S^1_{r\sqrt{a_2/(a_1+a_2)}}$$
under the natural action of the subgroup
$$ S^1 \cong \{(e^{in_1t},e^{in_2t}); \: t \in \R\} \subset S^1 \times S^1$$
(i.e.~component-wise multiplication). Recall that $\Sigma_{\mathbf{a},r}^{\mathbf{n}}$ can be obtained as smooth hypersurface of the symplectic manifold obtained from $H^{\mathbf{n}}_{\mathbf{a},r}$ via symplectic reduction of its boundary. Consequently, the restriction of $\omega_0$ to
$$\Sigma_{\mathbf{a},r}^o \coloneqq \Sigma_{\mathbf{a},r} \setminus \partial \Sigma_{\mathbf{a},r} \subset \Sigma_{\mathbf{a},r}^{\mathbf{n}}$$
has an extension to a smooth closed two-form defined on the entire hypersurface $\Sigma_{\mathbf{a},r}^{\mathbf{n}}$. In Subsection \ref{sec:beta} below we will further investigate the behaviour of this two-form near the closed curve
$$S_{\mathbf{a},r}^{\mathbf{n}} \coloneqq \Sigma_{\mathbf{a},r}^{\mathbf{n}} \setminus \Sigma_{\mathbf{a},r}^o$$
which is the image of the torus $\partial\Sigma_{\mathbf{a},r}$ under the quotient.
 \color{black}

\subsection{Symplectic reduction via a Lefschetz fibration (Approach ($\alpha$))}
The Bohr--Somerfeld Lagrangians inside $(\CP^n,\omega_{\OP{FS}})$ that we consider here, as well as their Legendrian lifts to the prequantisation space $(S^{2n+1},\alpha_{\OP{st}}),$ can be efficiently understood via their images under the standard Lefschetz fibration
\begin{gather*}\C^n \to \C\\
  (z_1,\ldots,z_n) \mapsto z_1^2+\ldots+z_n^2
\end{gather*}
in an affine chart of the symplectic base $\CP^n.$ To consider Lagrangians in relation to a Lefschetz fibration has turned out to be a useful perspective, which goes back to the work \cite{Eliashberg:Georgia} by Eliashberg--Polterovich. We call a Lagrangian which projects to a curve under a symplectic fibration {\bf compatible} with the fibration.

We proceed to give details in the case $n=2;$ the general case needed in Section \ref{sec:disc} is treated analogously in the end of this subsection.

In dimension $n=2$ we choose coordinates so that the Lefschetz fibration becomes $(z_1,z_2) \mapsto z_1\cdot z_2.$ In order to fix notation, we endow $\C^2$ with the standard linear symplectic form $\omega_0=dx_1\wedge dy_1+dx_2 \wedge dy_2.$

The symplectic fibration $(z_1,z_2) \mapsto z_1^{a_1}\cdot z_2^{a_2}$ with a singularity at the origin restricts to a surjective and smooth $S^1$-fibration
$$ f_{\mathbf{a}} \colon \Sigma_{{\mathbf{a},r}} \to D_{\widetilde{r}}^2 \setminus \{0\}$$
onto the punctured disc of radius
 $$\widetilde{r}=\sqrt{a_1}^{a_1}\sqrt{a_2}^{a_2}(r/\sqrt{a_1+a_2})^{a_1+a_2}.$$
 \color{black}
The following lemma shows that the bundle-projection $f_{\mathbf{a}}$ of the $S^1$-fibration is a symplectic reduction; thus the Lagrangians that are contained inside $\Sigma_{\mathbf{a},r}$ can be studied via their images under $f_{\mathbf{a}}$. \color{black}

\begin{lem}
\label{lem:char}
The characteristic distribution $\ker(\omega|_{T\Sigma_{\mathbf{a},r}})$ is spanned by the infinitesimal generator of the action
$$ (z_1,z_2) \mapsto (e^{ia_2t}z_1,e^{-ia_1t}z_2),$$
by $t \in \R$. Consequently, the characteristic distribution is
tangent to the $S^1$-fibres of $f_{\mathbf{a}}.$
\end{lem}
In particular, by dimensional reasons, any (two-dimensional) Lagrangian immersion that is contained inside $\Sigma_{\mathbf{a},r}$ projects to a smooth immersed curve inside $D_{\widetilde{r}}^2 \setminus \{0\}$ under $f_{\mathbf{a}}$ and vice versa: since the fibres of $f_{\mathbf{a}}$ are tangent to the characteristic distribution, the preimage of any curve inside the base $D_{\widetilde{r}}^2 \setminus \{0\}$ under the same map is a Lagrangian immersion   contained inside the hypersurface  \color{black}  $\Sigma_{\mathbf{a},r}.$

In the following we will mainly be interested in $\Sigma_{\mathbf{a},r}$ for $\mathbf{a}=(1,1)$; in this case we drop the subscript $\mathbf{a}$ from both the hypersurface as well as the bundle projection. To be able to study the symplectic action properties of such Lagrangians, the following simple computation will be crucial.
\begin{lem}
  \label{lem:pullback}
The pull back to $\Sigma_r$ of the standard symplectic form $\omega_0$ on $\C^2$ is given by
$$\omega_0|_{T\Sigma_{r}}=f^*\,d((s/2)\,d\theta),$$
where $d((s/2)\,d\theta)$ is a symplectic form on $D^2_{\widetilde{r}} \setminus \{0\}$ of total area $\pi\widetilde{r}$ defined using the polar coordinates $(s,\theta)$.
\end{lem}
\begin{proof}
The standard symplectic form on $\C^2$ can be expressed as
$$\omega_0=d\left(\frac{s_1^2}{2}d\theta_1+\frac{s_2^2}{2}d\theta_2\right)$$
using the polar coordinates $(s_i,\theta_i)$ on each $\C$-factor. The map $f$ is given by
$$(s_1,\theta_1,s_2,\theta_2) \mapsto (s_1s_2,\theta_1+\theta_2)$$
in the same coordinates. Hence, the form $(s/2)\,d\theta$ pulls back to $(s_1s_2/2)(d\theta_1+d\theta_2)$ under $f$. Since $s_1=s_2$ holds on $\Sigma_{r}$ we obtain the equality
$$ \left.\left(\frac{s_1^2}{2}d\theta_1+\frac{s_2^2}{2}d\theta_2\right)\right|_{T\Sigma_{r}}=(s_1s_2/2)(d\theta_1+d\theta_2)|_{T\Sigma_{r}}$$
as sought.
\end{proof}

\subsection{A symplectic fibration near the divisor (Approach ($\beta$))}
\label{sec:beta}
Recall that $\Sigma_{\mathbf{a},r}^{\mathbf{n}}$ is obtained from the hypersurface $\Sigma_{\mathbf{a},r}$ via a quotient of its boundary, where this boundary is a Lagrangian torus. The image of the torus $\partial \Sigma_{\mathbf{a},r} \to \Sigma_{\mathbf{a},r}^{\mathbf{n}}$ under the quotient is an embedded closed curve $S_{\mathbf{a},r}^{\mathbf{n}} \subset \Sigma_{\mathbf{a},r}^{\mathbf{n}}.$ It has a neighbourhood given as a disc normal-bundle $D_\epsilon^2 \times S_{\mathbf{a},r}^{\mathbf{n}}$ with the following explicit description. Choose $\mathbf{c}=(c_1,c_2) \in \Z^2$ for which $n_1c_2-n_2c_1=1$ (here we use that $\mathbf{n}$ is primitive).
Consider the induced family
$$\left\{(s\sqrt{a_1/(a_1+a_2)}\cdot e^{i(n_1\varphi+c_1\theta)},s\sqrt{a_2/(a_1+a_2)}\cdot e^{i(n_2\varphi+c_2\theta)}); \:\: r-\epsilon \le s \le r \right\} \subset \Sigma_{\mathbf{a},r}$$
of annuli; here each annulus is parametrised by $(s,\varphi)$  while the family is parametrised by $\theta \in [0,2\pi)$. Since $n_1c_2-n_2c_1=1$ it follows that $\theta$ descends to a coordinate on $S_{\mathbf{a},r}^{\mathbf{n}}\cong S^1$ under the quotient.

The boundaries $\{s=r/\|\mathbf{a}\|\}$ of the annuli provide a smooth \blk  foliation of the torus $\partial\Sigma_{\mathbf{a},r}$ by the closed curves $\{\theta=c\}$. Since each leaf of the foliation is collapsed to a point under the quotient $\Sigma_{\mathbf{a},r} \to \Sigma_{\mathbf{a},r}^{\mathbf{n}}$ (which is given by the group action $(z_1,z_2) \mapsto (e^{in_1t}z_1,e^{in_2t}z_1)$); hence, the above family of annuli naturally gives rise to a family $D_\epsilon^2 \times S_{\mathbf{a},r}^{\mathbf{n}} \hookrightarrow \Sigma_{\mathbf{a},r}^{\mathbf{n}}$ of discs under the same quotient.

\begin{lem}
\label{lem:sympdiscs}
The discs $D_{\epsilon}^2 \times \{\theta\} \hookrightarrow \Sigma_{\mathbf{a},r}^{\mathbf{n}}$ induced by the above annuli are all symplectic, or equivalently, the characteristic distribution of $\Sigma_{\mathbf{a},r}^{\mathbf{n}}$ is transverse to each such disc. Moreover, the characteristic distribution is tangent to the curve $S_{\mathbf{a},r}^{\mathbf{n}} \subset \Sigma_{\mathbf{a},r}^{\mathbf{n}}$.
\end{lem}
\begin{proof}
Since $n_1a_1+n_2a_2 \neq 0$ holds by assumption, it follows from Lemma \ref{lem:char} that the characteristic distribution $\frac{d}{dt}(e^{ia_2t},e^{-ia_1t})$ on $\Sigma_{\mathbf{a},r}^o$ is nowhere parallel to $\frac{d}{dt}(e^{in_1t},e^{in_2t})$ and, consequently, everywhere transverse to the annuli $\{\theta=c\}$ inside
$$\Sigma_{\mathbf{a},r}^o = \Sigma_{\mathbf{a},r}^{\mathbf{n}} \setminus S_{\mathbf{a},r}^{\mathbf{n}} \subset \Sigma_{\mathbf{a},r}^{\mathbf{n}}.$$
In other words, the discs $D_{\epsilon}^2 \times \{\theta\}$ are all symplectic away from the origin (in this region they coincide with the aforementioned annuli).

Since the characteristic distribution on $\Sigma_{\mathbf{a},r}$ moreover is tangent to its boundary, it follows that the characteristic distribution on $\Sigma_{\mathbf{a},r}^{\mathbf{n}}$ is tangent to the core $S_{\mathbf{a},r}^{\mathbf{n}}$ (i.e.~the image of the boundary under the quotient) as sought.  This implies, in particular, that the discs $D_{\epsilon}^2 \times \{\theta\}$ are symplectic at the origin as well.  
\end{proof}
In view of the above lemma, the product structure induced by the identification of the neighbourhood of $S_{\mathbf{a},r}^{\mathbf{n}} \subset \Sigma_{\mathbf{a},r}^{\mathbf{n}}$ with an embedded solid torus $D_\epsilon^2 \times S_{\mathbf{a},r}^{\mathbf{n}}$ gives us a symplectic fibration
$$ g_{\mathbf{a}}^{\mathbf{n},\mathbf{c}} \colon \operatorname{Op}_{ \Sigma_{\mathbf{a},r}^{\mathbf{n}}}( S_{\mathbf{a},r}^{\mathbf{n}}) \to S_{\mathbf{a},r}^{\mathbf{n}} \cong S^1 $$
defined in a neighbourhood $\operatorname{Op}_{ \Sigma_{\mathbf{a},r}^{\mathbf{n}}}( S_{\mathbf{a},r}^{\mathbf{n}})$ of the curve $S_{\mathbf{a},r}^{\mathbf{n}} \subset \Sigma_{\mathbf{a},r}^{\mathbf{n}}.$ \blk The characteristic distribution gives rise to a well-defined parallel transport
$$ D_\epsilon^2 \times \{\theta_0\} \to D_\epsilon^2 \times \{\theta_1\} $$
induced by the path $\theta=\theta_0+t(\theta_1-\theta_0)$ in $S_{\mathbf{a},r}^{\mathbf{n}}$ which starts at $\theta=\theta_0$ and ends at $\theta=\theta_1$. This is a symplectomorphism which, since the base is one-dimensional, moreover is independent of the choice of path up to homotopy with fixed endpoints.

As mentioned in Lemma \ref{lem:sympdiscs} above, the curve $S_{\mathbf{a},r}^\mathbf{n} \subset \Sigma_{\mathbf{a},r}^\mathbf{n}$ is tangent to the characteristic distribution. In other words, the origin of the disc $D_\epsilon^2$ is preserved by the above parallel transport. More precisely, the parallel transport takes the form of the following explicit rotation:
 \color{black}
\begin{lem}
\label{lem:parallel}
Relative the polar coordinates $(s,\varphi)$ on the discs $D_\epsilon^2 \times \{\theta\}$, the aforementioned parallel transport takes the form
$$(s,\varphi) \mapsto \left(s,\varphi - \frac{c_1a_1+c_2a_2}{n_1a_1+n_2a_2}(\theta_1-\theta_0)\right),$$
where $n_1a_1+n_2a_2$ is the intersection number between the two curves $t\cdot(a_2,-a_1)$ (i.e.~an integral curve of the characteristic distribution) and $t\cdot\mathbf{n}$ on the torus $\R^2/\Z^2.$
\end{lem}
\begin{proof}
Consider the frame $\langle\partial_t,\partial_\varphi,\partial_\theta\rangle$ of the tangent bundle of the hypersurface $\Sigma_{\mathbf{a},r}$ induced by the above coordinates. In this frame the characteristic distribution (recall its expression from Lemma \ref{lem:char}) is given by
$$ \R\cdot((c_1a_1+c_2a_2)\partial_\varphi+(-n_1a_1-n_2a_2)\partial_\theta).$$
It follows that
$$-\frac{c_1a_1+c_2a_2}{n_1a_1+n_2a_2}\partial_\varphi+\partial_\theta$$
is a horizontal vector field for the parallel transport in question.
\end{proof}
Any Lagrangian submanifold contained inside $\Sigma_{\mathbf{a},r}^{\mathbf{n}}$ must be tangent to the characteristic distribution. The above lemma thus in particular implies that
\begin{cor}
\label{cor:symmetry} Any properly immersed Lagrangian
without boundary which is contained inside the solid torus
$D_\epsilon^2 \times S_{\mathbf{a},r}^{\mathbf{n}}$ is determined
uniquely by its intersection with the symplectic disc $D_\epsilon^2
\times \{\theta\}$; the latter is a curve which is invariant under
the monodromy map
$$(s,\varphi) \mapsto \left(s,\varphi - \frac{c_1a_1+c_2a_2}{n_1a_1+n_2a_2}2\pi\right).$$
The intersection of the Lagrangian with the remaining discs are then given as the images of this curve under the parallel transport map.

Conversely, given any curve inside $D_\epsilon^2 \times \{\theta\}$ which is invariant under the monodromy, the extension to a parallel manifold is a Lagrangian submanifold contained inside $D_\epsilon^2 \times S_{\mathbf{a},r}^{\mathbf{n}}$.
\end{cor}

\subsection{ The particular case $\mathbf{a}=\mathbf{n}=(1,1)$, $\mathbf{c}=(0,1)$.}
\label{sec:maincase}

For our purposes we will use the hypersurface $\Sigma_{(1,1),1}^{(1,1)} \subset \CP^2$ depicted in Figure \ref{fig:moment_polytope} for constructing our Lagrangian isotopies. We give a brief recollection of the above constructions in this particular case.

A regular homotopy of Lagrangians that are contained inside this hypersurface, and which are disjoint from the divisor $\ell_\infty$ at infinity, can be efficiently described by its projection under the symplectic reduction
$$f \colon \Sigma_{(1,1),1}^{(1,1)} \setminus \ell_\infty = \Sigma_{(1,1),1}^o \to \dot{D}^2_{1/2}.$$
Any such Lagrangian projects to an immersed curve, and any regular homotopy through such Lagrangians corresponds to a regular homotopy of the projection. Recall that the Bohr--Sommerfeld condition will be satisfied whenever the symplectic action of the Lagrangian immersion takes values in $\pi\cdot\Z$. This condition then translates to the fact that the symplectic action for the projection takes values in $\pi\cdot\Z$ with respect to the exact symplectic form from Lemma \ref{lem:pullback}.

In order to describe a Lagrangian regular homotopy that traverses the line at infinity $\ell_\infty \cap \Sigma_{(1,1),1}^{(1,1)}$ the projection $f$ cannot be used; here we instead use the symplectic fibration
$$g \colon \Sigma_{(1,1),1}^{(1,1)} \to S^1$$
induced by $\mathbf{c}=(0,1)$. In fact, we will only use this symplectic fibration in a small neighbourhood of $\ell_\infty$. Note that the induced monodromy map is given by $(s,\varphi) \mapsto (s,\varphi+\pi),$ i.e.~it is a rotation by $\pi$ radians.

\begin{figure}[h!]
\vspace{5mm}
\centering
\labellist
\pinlabel $\pi$ at 220 15
\pinlabel $x$ at 280 30
\pinlabel $\mathbf{n}=(1,1)$ at 205 140
\pinlabel $\mathbf{c}=(0,1)$ at 150 173
\pinlabel $\frac{\pi}{3}$ at 95 11
\pinlabel $\frac{\pi}{3}$ at 15 98
\pinlabel $y$ at 34 282
\pinlabel $\color{blue}L_{\OP{Cl}}$ at 118 78
\pinlabel $\pi$ at 15 224
\endlabellist
\includegraphics[scale=0.5]{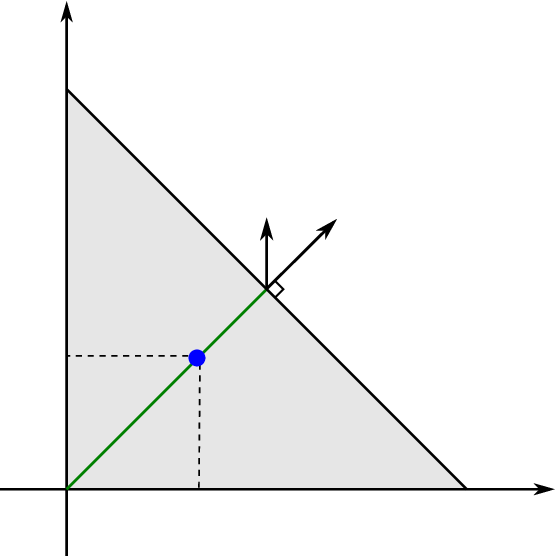}
\caption{The standard momentum polytope for $\CP^2,$ the fibre over the point $(\pi/3,\pi/3)$ (shown in blue) is the Clifford torus while the subset over the line $\{u=v\}$ (shown in green) is symplectomorphic to $\Sigma_{(1,1),1}^{\mathbf{n}}$ with $\mathbf{n}=(1,1)$.}
\label{fig:moment_polytope}
\end{figure}

\subsection{Higher dimensions $n>2$}

In the case $n > 2,$ the subset $\overline{\Sigma} \subset \CP^n$ that we are interested in is the quotient of the real $(n+1)$-dimensional submanifold
$$ \Sigma \coloneqq \{ z(x_1,\ldots,x_n); \: z \in S^1 \subset \C^*, x_i \in \R \} \cap D^{2n}$$
 under the symplectic reduction $D^{2n} \to \CP^n$   along the boundary $S^{2n-1}=\partial D^{2n}.$
In higher dimensions $\Sigma$ is of course not a hypersurface, but we can nevertheless still use the same approach. Again the characteristic distribution are the tangencies to the $S^{n-1}$-fibres of the symplectic Lefschetz fibration $z_1^2+\ldots+z_n^2$ restricted to $\Sigma.$ Observe that this description also makes sense in the case $n=2$; in fact, this hypersurface is equivalent to $\Sigma_{(1,1),1}^{(1,1)}$ after a suitable coordinate change.

\subsection{Loose charts and the Lefschetz fibration}
We now describe a loose chart that is well-behaved with respect to the restricted Lefschetz fibration $f \colon \Sigma \to \dot{D}^2_{1/2}$. 
Consider a Legendrian inside $(S^{2n+1},\alpha_{\OP{st}})$ whose Lagrangian projection in $\CP^n$ has the following description. Inside some $f^{-1}(U) \subset \Sigma$ for an open   simply connected    domain $U \subset \dot{D}^2_{1/2}$ with smooth boundary, the projection of $\Lambda \cap f^{-1}(U)$ under $  f \circ \pi$ is a single arc $\gamma$ that intersects $\partial U$ transversely in two points, where this arc moreover has a single transverse  double point; see Figure \ref{fig:loose}. Let $A \subset U \setminus \gamma$ be the connected component which is bounded in $\C \setminus \gamma$ while $B \subset U \setminus \gamma$ is the connected component which is adjacent to $A$ only at the double-point of $\gamma$; these different components are depicted in Figure~\ref{fig:loose}.
\begin{lem}
\label{lem:loose}
If the $(1/2)dr\wedge d\theta$-area of $A$ is strictly smaller than that of $B,$ then we can assume that $A$ is arbitrarily small after a Legendrian isotopy whose support is compact and contained inside
$$\pi^{-1}(f^{-1}(U)) \subset \pi^{-1}(\Sigma) \subset S^{2n+1}$$
Additionally, $\Lambda$ is loose under the same assumptions.
\end{lem}
\begin{proof}
The compactly supported Legendrian isotopy is explicitly constructed as the lifts of $f^{-1}(\gamma_t)$ for a suitable regular isotopy $\gamma_t$ of the curve $\gamma_0=\gamma$   which remains fixed near the boundary $\partial U$. We want the curves $\gamma_t$ to all be immersed with a single transverse double point, while the $(1/2)dr\wedge d\theta$-area satisfies the following properties:
\begin{itemize}
\item the region $A$ for the curve $\gamma_t$, for all $t>1$, is contained inside the ball $B^2_{1/t}(p) \subset \dot{D}^2_{1/2}$ of radius $1/t$ centered at some fixed point $p$ (i.e.~the region $A$ shrinks to a point under the deformation; in particular its area is shrinking indefinitely); and
\item two different curves in family $\gamma_t$ have signed area difference equal to zero (the area \emph{difference} does not depend on the choice of a primitive of the area form).
\end{itemize}
Note that the assumption that the area of $A$ is smaller than the area $B$ is crucial in order to make both bullet points satisfied simultaneously.
By Lemma \ref{lem:pullback} it follows that the Bohr--Sommerfeld condition is preserved, and that the corresponding Legendrian lifts can be taken to be supported inside $\pi^{-1}(f^{-1}(U))$ as sought. In particular, the Legendrians stay embedded.

What remains is to show the looseness. One readily shows that there exists a neighbourhood of the form
$$(D_\epsilon T^*S^{n-1} \times [-\epsilon,\epsilon]^3,-pdq-ydx+dz) \hookrightarrow \pi^{-1}(f^{-1}(B^2_{1/T}(p)) \subset (S^{2n+1},\alpha_{\OP{st}})$$
for some $T \gg 0$ in which $f\circ\pi$ becomes identified with the Lagrangian projection of the factor $[-\epsilon,\epsilon]^3$, i.e.~the three-dimensional contact Darboux ball.

After the aforementioned isotopy, we may thus assume that the Legendrian coincides with $0_{S^n} \times \Lambda_{\OP{stab}},$ for a stabilised Legendrian arc $\Lambda_{\OP{stab}} \subset ([-\epsilon,\epsilon]^3,dz-ydx).$   To that end, recall that the Lagrangian projection of a standard stabilisation is precisely the loop shown in Figure \ref{fig:loose} with $A$ being of smaller area than $B$. 
\end{proof}

\begin{figure}[t]
  \centering
  \vspace{3mm}
\labellist
\pinlabel $x$ at 195 92
\pinlabel $y$ at 92 198
\pinlabel $U$ at 73 140
\pinlabel $\color{blue}\gamma$ at 82 108
\pinlabel $B$ at 65 112
\pinlabel $A$ at 44 127
\endlabellist
\includegraphics{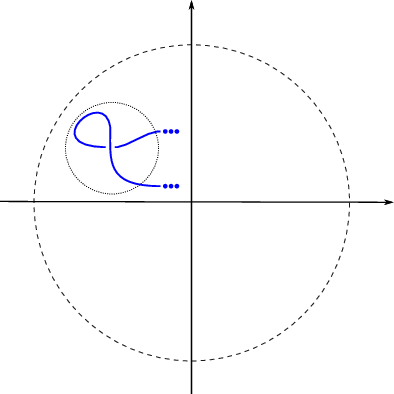}
\caption{A loose chart described by using the standard Lefschetz fibration, in the sense that $ f(\pi(\Lambda))=\gamma$ is an arc with a single self-intersection. The area of the region $A$ must be greater than the area of the region $B.$}
\label{fig:loose}
\end{figure}

\blk

\section{Canonical Bohr--Sommerfeld covers}
\label{sec:bscovers}

Here we present some fairly general considerations on how to obtain canonically defined Bohr--Sommerfeld covers from rational Lagrangian embeddings, and how to compute the Maslov classes of their lifts. Theorem \ref{thm:bscovers} then follows as a special case. Here we assume that $(E,\alpha)$ is a given prequantisation bundle over $(M,\omega)$ and the notion of Bohr--Sommerfeld Lagrangian will be taken with respect to this bundle.

In order to be able to reformulate the property of being Bohr--Sommerfeld in terms of integrality properties of the symplectic area of relative homology classes, we make the simplifying assumption that $\pi_1(M)=0$ throughout this section; c.f.~Remark \ref{rem:sympaction}.

\subsection{Bohr--Sommerfeld covers}

Any Lagrangian submanifold whose symplectic action class satisfies
\begin{equation}
  \label{eq:k}
[\omega] \in  \frac{\pi}{k}\cdot\im(H^2(M,L,\Z) \to H^2(M,L,\R))
\end{equation}
for some $k =1, 2,3,\ldots,$ admits a canonical $k$--fold cover which is a Bohr--Sommerfeld Lagrangian immersion determined in the following manner.

 \color{black} The symplectic action class $\sigma \colon H_2(M,L,\Z) \to \R$ descends to a well-defined morphism
$$ \overline{\sigma} \colon H_1(L,\Z) \to \left(\frac{\pi}{k}\cdot \Z\right)/(\pi\cdot \Z) \cong \Z_k,$$
which we assume is surjective.

Consider the subgroup $\ker \overline{\sigma} \subset H_1(L,\Z)$ and   the subgroup $G \subset \pi_1(L)$ given as the preimage of $\ker \overline{\sigma}$ under the abelinisation. (In the cases that we are interested in here we have $\pi_1(L)=H_1(L)$.) The corresponding $k$-fold cover $\widetilde{L} \to L \subset M$ is a canonically defined Lagrangian immersion $\widetilde{L} \looparrowright (M,\omega)$,  which is Bohr--Sommerfeld.  \color{black}

\begin{defn}
 Let $L \subset (M,\omega)$ be a Lagrangian embedding. For the minimal value $k = 1,2,3,\ldots$ with the property that Equation \eqref{eq:k} is satisfied, the induced Lagrangian immersion $\widetilde{L} \hookrightarrow (M,\omega)$ is called the {\em canonical $k$-fold Bohr--Sommerfeld cover of $L$.}
\end{defn}
From Lemma \ref{lem:prequant}, together with the main theorem of \cite{LiftingImmersions} (in order to ensure embeddedness), we deduce that:
\begin{prop}
\label{prop:lift}
Assume that $L$ is a closed embedded Lagrangian submanifold of $(M,\omega)$ that satisfies Equation \eqref{eq:k}. Its canonical $k$-fold Bohr--Sommerfeld cover $\widetilde{L} \looparrowright (M,\omega)$ then admits an embedded Legendrian lift $\Lambda_{\widetilde{L}} \hookrightarrow (E,\alpha)$  which is uniquely defined up to an application of the Reeb flow. \color{black}
\end{prop}

\subsection{Computing the Maslov class and degrees}
\label{sec:maslov}

In this subsection we provide tools for relating the Maslov class of a Legendrian submanifold
$$\Lambda \subset (E,\alpha) \xrightarrow{\pi} (M,\omega)$$
of a prequantum bundle and the Maslov class of its Lagrangian projection, which is a Lagrangian immersion
$$\iota \colon {\tilde L} \looparrowright (M,\omega).$$
The ultimate goal of this is to relate the Maslov class of the canonical Lagrangian Bohr--Sommerfeld cover $\iota \colon \tilde L \looparrowright (M,\omega)$ of a Lagrangian embedding $L \subset (M,\omega)$ and the Maslov class of its Legendrian lift $\Lambda \subset (E,\alpha)$ produced by Proposition \ref{prop:lift}.

By definition the first Chern class of a contact manifold $c_1(E,\alpha) \in H^2(E)$ is the first Chern class of the contact distribution $\ker \alpha$. For a prequantisation bundle it immediately follows that this class is the pull-back of the first Chern class $c_1(M,\omega) \in H^2(M)$ on $M$.

The Maslov class of a Legendrian in a contact manifold is a class $\mu_\Lambda \in H^2(E,\Lambda)$ which is defined as the Maslov class of the Lagrangian tangent planes of $T\Lambda \subset \ker \alpha$ inside the symplectic bundle $\ker \alpha$. As for the Chern class, in the case of a prequantisation bundle, this characteristic class is the pull-back of the Maslov class $\mu_{\pi(\Lambda)}$ of the Lagrangian immersion $\pi(\Lambda) \subset (M,\omega)$.

Again we make some simplifying assumptions. First, we assume that $\pi_1(E)=0$.
\begin{example}This is the case for the contact sphere as well as for the unit cotangent bundle of a sphere of dimension at least three.
\end{example}
In view of the long exact sequence of homotopy groups of a fibre bundle $S^1 \to E \xrightarrow{\pi} M$ the simple connectivity of the total space then implies that $\pi_1(M)=0$ as well. Furthermore, the long exact sequence of homotopy groups takes the form
\begin{equation}
\label{eq:les1}
\pi_2(S_1)=0 \to \pi_2(E) \xrightarrow{\pi_*} \pi_2(M) \xrightarrow{\frac{1}{\pi}\int_\cdot \omega} \Z \to 0=\pi_1(E),
\end{equation}
where the last map is equal to the curvature of $(E,\alpha) \to E$ i.e.~the symplectic area. Note that, under our assumptions $\pi_1(E)=\pi_1(M)=0$ of simple connectivity, we have $\pi_2(E) \cong H_2(E)$ and $\pi_2(M) \cong H_2(M)$ by the Hurewicz isomorphism.

Since we are mainly interested in monotone Lagrangians inside monotone symplectic manifolds, the following additional assumption is a natural simplification: we now assume that  the first Chern class of $(M,\omega)$ satisfies the proportionality $c_1(M)=N_\omega \cdot \omega/\pi$ for some $N_\omega \in \Z$; in particular, $c_1(M)$ vanishes on the classes of $H_2(M)$ of vanishing symplectic area. We can conclude that
\begin{lem}
\label{lem:chern}
If $\pi_1(E)=0$ and $c_1(M)$ is proportional to $\omega$, then the first Chern class of $(E,\alpha)$ vanishes on all of $H_2(E)=\pi_2(E)$.
\end{lem}
\begin{proof}
Since the Chern class $c_1(E,\alpha)$ is the pull-back of $c_1(M,\omega)$, the statement follows from the short exact sequence in \eqref{eq:les1} together with the assumption of proportionality between the Chern class and the symplectic form.

Also c.f.~\cite[Lemma 7.3]{BoothbyWangbundles} where it is shown that the first Chern class is torsion whenever the base is a simply connected monotone symplectic manifold.
\end{proof}

We also recall the following central definitions.

\begin{defn}
\label{defn:monotone}
A symplectic manifold $(M,\omega)$ is said to be {\em monotone} if the above constant $N_\omega>0$ is positive. More generally, a Lagrangian is {\em monotone} if its symplectic action class is a nonnegative multiple of its Maslov class. In particular, this automatically implies that the ambient symplectic manifold also is monotone.\end{defn}
In fact, the symplectic manifolds that we are interested in here are the following monotone ones.
\begin{example}
\begin{enumerate}
\item In the case $(S^{2n+1},\alpha_{\OP{st}}) \to (\CP^n,\omega_{\OP{FS}})$ we have $N_{\omega_{\OP{FS}}}=n+1$; recall that the area of a line is equal to $\pi$ with our convention.\label{ex:monotoneplane}
\item In the case $(UT^*S^n,pdq) \to (X,\omega)$ with $n \ge 3$, where $X$ is a monotone projective quadric $n-1$-fold, we have $N_\omega=n-1$. Recall that $\omega$ is normalised so that $X$ is uniruled by lines of area $\pi$ with our convention.\label{ex:monotonequadric}
\end{enumerate}
\end{example}

In the case when the Legendrian $\Lambda \subset (E,\alpha)$ is the lift of a Lagrangian immersion $\tilde L \looparrowright (M,\omega)$ to the prequantum bundle their Maslov classes have the following direct relation.
\blk
\begin{prop}
\label{prop:maslov}
For any $A\in H_2(E; \Lambda)$ the class $\pi_{\ast}(A)\in H_2(M;\Pi(\Lambda))$ has vanishing symplectic area. Since the Maslov class satisfies $$ \mu_{\Lambda}(A)=\mu_{\pi(\Lambda)}(\pi_*(A)),$$ it follows that the Legendrian lift of the Bohr--Sommerfeld cover of a monotone Lagrangian has vanishing Maslov class.
\blk
\end{prop}
\begin{proof}
As in the proof of Lemma \ref{lem:chern}, in view of the proportionality of Maslov class and symplectic area for chains with boundary on $L$, it suffices to show that $\pi_*(A)$ has vanishing symplectic area. This is the case by the definition of the curvature of $(E,\alpha) \to M$ which is equal to symplectic area; c.f.~the short exact sequence \eqref{eq:les1}.
\end{proof}

Next we compute the degree of the Reeb chords on the Legendrian lift $\Lambda$ of the canonical Bohr--Sommerfeld cover of an monotone Lagrangian embedding $L \subset (M,\omega)$. For a computation of the degree of the periodic Reeb orbits in the same geometric setting we refer to \cite[Theorem 3.1]{Hong:Conley} by Hong; in particular, monotonicity implies that the degrees of the periodic Reeb orbits are strictly positive.

In the following we assume that $\Lambda$ is connected, which together with the assumption $\pi_1(E)=0$  implies that for a Reeb chord $c$ of $\Lambda$, $[c] \in \pi_1(E,\Lambda)$ is the trivial class. The latter fact combined with the vanishing $\mu_\Lambda=0$ of the Maslov class established in Proposition \ref{prop:maslov} above implies that the degree is a well-defined integer $|c| \in \Z.$

We briefly recall the definition of the degree, which is given by $|c|=\OP{CZ}(\gamma)-1$ in terms of the Conley--Zehnder index, where $\gamma$ is a path on $\Lambda$ that connects the endpoint of $c$ with the starting point of $c$. The Conley--Zehnder index depends on the choice of trivialisation of the contact distribution along the cycle $c + \gamma$. We choose this trivialisation to be induced by a choice of chain $D$ inside $E$ with boundary $\partial D =c + \gamma$. Note that, for the prequantisation form $\alpha$, every Reeb chord $c$ on $\Lambda$ lives in a Bott manifold of dimension precisely $\dim \Lambda$ (here we use the assumption that the image $\pi(\Lambda)$ is embedded); the Bott version of the above Conley--Zehnder index must thus be used here. In this case the Bott manifold is of maximal dimension, which implies that $\OP{CZ}(\gamma)$ can be computed in the following manner: Consider the trivialisation of the contact distribution along $c + \gamma$ induced by $D$; the Conley--Zehnder index $\OP{CZ}(\gamma)$ is given by the Maslov class of the loop of Legendrian tangent planes $T\Lambda$ along $\gamma$ extended over the Reeb chord $c$ via the Reeb flow. (Since the Bott manifold is of full dimension the Reeb flow takes the Legendrian tangent plane at the starting point of $c$ onto the Legendrian tangent plane at its end point, so we indeed get a closed loop of Legendrian planes.)

\begin{remark}
Remember that the degree $|c|$ of a Reeb chord is equal to the expected dimension of the moduli space of unparametrised pseudholomorphic half-planes in the symplectisation with boundary on $\R \times \Lambda \subset \R \times E$, whose asymptotics are constrained to the Reeb chord $c$.
\end{remark}
\blk

\begin{prop}
\label{prop:bottdegree}
Denote by $\Lambda \subset (E,\alpha)$ a connected Legendrian lift of the canonical $k$-fold Bohr--Sommerfeld cover of a monotone Lagrangian embedding $L \subset (M^{2n},\omega).$ Any Reeb chord $c$ on $\Lambda$ is of length $\ell(c)=\pi \cdot l/k$ for some $l \in \N$; the Bott degree of a chord of length $\pi \cdot l/k$ is equal to
$$|c|=2N_\omega\cdot l/k-1 \in \Z$$
where $N_\omega$ is the minimal Chern number of $M$ as above. (Observe that $N_\omega/k$ is the minimal Maslov number of $L$ by the monotonicity; thus it is an integer.)

\end{prop}
\begin{proof}
Take a chain $D \subset E$ with boundary $c + \gamma$ such that $\gamma \subset \Lambda$ is a path which connects the two endpoints of the chord $c$. The Conley--Zehnder index $\OP{CZ}(\gamma)$ computed in the trivialisation induced by $D$ is equal to
$$\mu_L(\pi_*(D))=2N_\omega\cdot l/k$$
as a direct consequence of the monotonicity of $L$, and the area property $\int_{\pi_*(D)}\omega=\pi\cdot l/k.$ (Note that $[\pi_*(D)] \in H_2(M,L)$ and that the curvature of $(E,\alpha)$ coincides with the symplectic area.)
\end{proof}

\section{Some concrete Lagrangian tori and their lifts}
\label{sec:examples}

It is now time to restrict attention to some particular examples of Legendrian tori produced by Proposition \ref{prop:lift} as lifts of canonical Bohr--Sommerfeld covers.

\subsection{Lifts of monotone tori in the quadric surface}

The quadric surface is cut out by the polynomial equation
$\overline{\{z_1^2+z_2^2+z_3^2=1\}} \subset \CP^3$. We endow this
variety with the restriction of the Fubini--Study symplectic form,
which descends to the standard monotone symplectic form
$$ (\overline{\{z_1^2+z_2^2+z_3^2=1\}}, \omega_{\OP{FS}}) \cong (\CP^1 \times \CP^1,\omega_{\CP^1} \oplus \omega_{\CP^1}))$$
under a suitable identification. The affine part of this variety $\{z_1^2+z_2^2+z_3^2=1\} \subset
\C^3$ is symplectomorphic to an open subset $(DT^*S^2,d(pdq))$ of
the cotangent bundle. We consider the prequantisation space $E$ which is the unit cotangent bundle  $(E,\alpha)=(UT^*S^3,pdq)$  \color{black}  of the round three-sphere of radius $1/2.$  (The choice of radius depends on our previously fixed convention that the minimal period of a closed Reeb orbit is equal to $\pi$.) \color{black}  This bundle can be identified with the $S^1$-bundle corresponding to the complex line-bundle $\mathcal{O}(-1,-1)$ over the conic.

Any monotone Lagrangian torus inside $\CP^1 \times \CP^1$ has a symplectic action class which is contained inside the subset
$$\frac{\pi}{2}\cdot H^2(\CP^1 \times \CP^1,L,\Z) \subset H^2(\CP^1 \times \CP^1,L,\R),$$
and which, moreover, assumes the value $\frac{\pi}{2}$ on some homology class. This is a consequence of the integrality of the  Maslov class together with the fact that $c_1=N_\omega\cdot\omega/\pi$ with $N_\omega=2$; see Example \ref{ex:monotonequadric} above with $n=3$. \color{black}  In fact, the symplectic action class cannot be an integer multiple of $\pi$ by the result \cite[Theorem 1.21]{CM17} of Cieliebak--Mohnke.  In other words, there must exist a class of symplectic area precisely $\pi/2$ and thus of Maslov index two (in view of the monotonicity). \color{black}

The most basic example is the monotone product torus $S^1 \times S^1 \subset
\CP^1 \times \CP^1,$ which we now proceed to direct our attention to. (This torus also goes under the name of ``the Clifford torus'' some times, but in order to avoid confusion we will not call it by that.)

Recall that given a knot $K$ in $S^3$, one can define the conormal lift of it in the following way. First one takes $L_K$ defined by
$$L_K=\{(q,p)\ | \  q\in K, \langle p,v \rangle=0\ \mbox{for all}\ v\in T_qK\}\subset T^{\ast}S^3.$$
It is a standard exercise to check that $L_K$ is a Lagrangian submanifold of $T^{\ast}S^3$. Then one takes the
unit cotangent bundle $UT^{\ast} S^3$ of unit covectors of $S^3$ with respect to some
metric. The conormal lift $\Lambda_K$  is given by $\Lambda_K=L_K\cap UT^{\ast} S^3$, and it is a Legendrian torus in $UT^{\ast} S^3$.
\begin{thm}
  \label{thm:unknot}
Under the identification of $E$ with $(UT^*S^3,pdq)$ the Legendrian
lift of the twice covered monotone product torus is Legendrian isotopic to
the conormal lift of the unknot.
\end{thm}
\begin{proof}
The monotone product torus $S^1 \times S^1 \subset \CP^1 \times \CP^1$ given as the product of equators is invariant under the standard Hamiltonian $\T^2$-action on the conic (which is toric). One can readily check that this $\T^2$-action lifts to strict contactomorphisms of the prequantisation space, and that the Legendrian lift of the canonical Bohr--Sommerfeld cover is invariant under the latter $\T^2$-action.

From the point of view of $UT^*S^3,$ this $\T^2$-action is induced by   first restricting the standard toric $\T^2$-action on $\C^2$ to the unit sphere $S^3$, and then lifting it the cotangent bundle $T^*S^3$. \color{black}  (Observe that this $\T^2$-action is by isometries on the round three-sphere.) We refer to e.g.~\cite[Section 6.1]{ContactToric} for a careful treatment of this contact toric structure on $UT^*S^3.$ It follows that any Legendrian which is invariant under this $\T^2$-action must thus live above a (possibly degenerate) torus in the base of the form $S^1_a \times S^1_b \subset S^3$ for $a^2+b^2=1,$  where $S^1_r \subset \C$ denotes the unit circle of radius $r$.  \color{black}  In other words, the Legendrian lift of the Bohr--Sommerfeld cover is the conormal lift of one of these tori.

Finally, since the one-sided conormal lifts of all these tori can be seen to be Legendrian isotopic to the conormal lift of an unknot $S^1_1 \times \{0\} \subset S^3$ by a suitable application of the Reeb flow induced by the round metric, the result now follows.
\end{proof}
\begin{rem}
Alternatively, the result can be proven following the same recipe as in the proof of Theorem \ref{thm:frontsoflifts} below in Section \ref{sec:frontsoflifts}. There the Legendrian isotopy is obtained by describing a homotopy of Bohr--Sommerfeld Lagrangian immersions via their images under a Lefschetz fibration. In this case one then uses the Lefschetz fibration which on the affine part $\{z_1^2+z_2^2+z_3^2=1\} \subset \C^3$ of the conic is equal to the restriction of $(z_1,z_2,z_3) \mapsto z_3.$ In fact, the present situation is significantly simpler compared to the situation in Theorem \ref{thm:frontsoflifts}, due to the fact that this defines a Lefschetz fibration on the full conic, and not just its affine part. In order to identify the conormal lift of the unknot, we refer to \cite{KnotContHom}, where it is shown that a suitable representative has a Lagrangian projection inside the conic that lives above an immersed figure-8 curve that encircles both critical values in the Lefschetz fibration.   The proof above instead uses a more direct approach.  \color{black}
\end{rem}

There are infinitely many monotone Lagrangian tori inside the quadric surface by a result \cite{Vianna17} due to Vianna. We do not know if any other of these tori have a lift that is isotopic to the conormal lift of a smooth knot.

In any case, the above Legendrian lift of the product torus is not subloose since the conormal lift of the unknot admits an exact Lagrangian filling inside $(T^*S^3,d(pdq))$; e.g.~the Lagrangian conormal of the knot. In fact, since the canonical Bohr--Sommerfeld covers of monotone Lagrangians in the conic are \emph{two}-fold covers, it is not difficult to show the following.
\begin{prop}
  \label{prp:monotonefilling}
  Any monotone Lagrangian torus $L \subset (\CP^1\times\CP^1,\omega_{\OP{FS}}\oplus\omega_{\OP{FS}})$ has a two-fold canonical Bohr--Sommerfeld cover whose Legendrian lift admits a monotone Lagrangian filling inside the complex line bundle $\mathcal{O}(-1,-1)$ which is contained in the preimage of $L$ under the bundle projection.
\end{prop}
\begin{proof}
The prequantisation bundle is the unit-sphere bundle of the complex line bundle
 $\mathcal{O}(-1,-1) \to \CP^1\times\CP^1,$ and in particular the corresponding disc-bundle is an a monotone symplectic filling of this contact manifold. (Recall the general fact that the unit-sphere bundle of a negative line bundle is a pseudconvex boundary of the disc bundle.) The Lagrangian filling itself can then be taken to be the full $\R$-subbundle of the $\C$-bundle defined over the Lagrangian $L$, whose associated $S^0$-bundle inside the unit sphere bundle (i.e.~the prequantisation bundle) coincides with the Legendrian lift. Note that this Lagrangian filling transversely intersects the base $\CP^1 \times \CP^1$ precisely in the Clifford torus.
\end{proof}
We believe, but do not prove, that these monotone Lagrangian fillings also are obstructions to being subloose.

\subsection{Lifts of monotone tori in the projective plane}
\label{sec:examples-plane}

By Theorem \ref{noembeddedBSLagr} the symplectic action class of any embedded Lagrangian cannot take values only in $\Z\cdot\pi.$  Recall that $c_1=N_\omega\cdot\omega/\pi$ with $N_\omega=3$ \blk in the case of $\CP^2$ (see Example \ref{ex:monotoneplane} above with $n=2$). If a Lagrangian in $\CP^2$ is oriented and monotone, it thus follows that its symplectic action class cannot take values only in $\Z\cdot\pi/2$ either; a class of symplectic area equal to a half-integer multiple of $\pi$ would have an odd Maslov index (this only happens when the Lagrangian is non-orientable). A further consideration of the integrality of the Maslov class in conjunction with $N_\omega=3$ we now conclude that  \color{black}
a monotone Lagrangian torus $L$ inside the projective plane has a
symplectic action class which lives inside the subset
$$ \frac{\pi}{3}\cdot H^2(\CP^2,L,\Z) \subset H^2(\CP^2,L,\R),$$
and which, moreover, assumes the value $\frac{\pi}{3}$ on some homology class. \color{black}  In other words, an appropriate three-fold cover of any monotone Lagrangian torus is Bohr--Sommerfeld for the prequantisation space
$$ S^1 \to S^5 \to \CP^2.$$

\begin{rem}
The Audin conjecture holds for Lagrangian tori in the projective planes $\CP^n$ by work of Cieliebak--Mohnke \cite{CM17}, i.e.~they admit a relative homology class of Maslov index two. From this strong fact, combined with the fact $N_\omega=n+1$, it then more generally follows that any monotone torus in $\CP^n$ has a symplectic action class which lives inside the subset
$$ \frac{\pi}{n+1}\cdot H^2(\CP^n,L,\Z) \subset H^2(\CP^n,L,\R),$$
and which, moreover, assumes the value $\frac{\pi}{n+1}$ on some homology class. \color{black}
\end{rem}

The first two well-known examples of monotone Lagrangian tori are the Clifford and the Chekanov torus \cite{Ch98,TwistTori}. In terms of the fibration
$$f \colon \Sigma \to \dot{D}^2_{1/2}$$
from Section \ref{sec:lefschetz}, i.e.~the restriction of the Lefschetz fibration $(z_1,z_2) \mapsto z_1 \cdot z_2$ to the subset $\Sigma \subset \CP^2,$ these two tori have the following descriptions.
\begin{itemize}
\item The {\bf Clifford torus} is the preimage of an embedded closed curve in $(\dot{D}^2_{1/2},(1/2)d(r\,d\theta))$ of symplectic area $\pi/3$ which \emph{encircles} the origin; while
\item The {\bf Chekanov torus} is the preimage of an embedded closed curve in $(\dot{D}^2_{1/2},(1/2)d(r\,d\theta))$ of symplectic area $\pi/3$ which does \emph{not} encircle the origin.
\end{itemize}
See Lemma \ref{lem:pullback} for more details concerning the area considerations.

The Legendrian lift of the Bohr--Sommerfeld cover of the Clifford torus is already well studied, since it is the link of the singularity of the Harvey--Lawson cone \cite{HarveyLawson}. Its contact topology was been studied by Nadler \cite{Nadler}, Treumann--Zaslow \cite{TreumannZaslow}, as well as Baldridge--McCarthy--Vela-Vick \color{black} \cite{LiftingImmersions}. By Theorem \ref{thm:frontsoflifts} below, this torus is, in addition, Legendrian isotopic to the ``knotted'' Legendrian torus considered  in \cite{KnottedLegendrianSurface} by the first author.

\begin{thm}
\label{thm:frontsoflifts}
The threefold Bohr--Sommerfeld coverings of the Clifford and Chekanov tori in $\CP^2$ have Legendrian lifts in the prequantisation space $S^5\to\CP^2$ that are Legendrian isotopic into a Darboux ball with front projections as shown in Figures \ref{fig:Clifford} and \ref{fig:Chekanov}.
\end{thm}

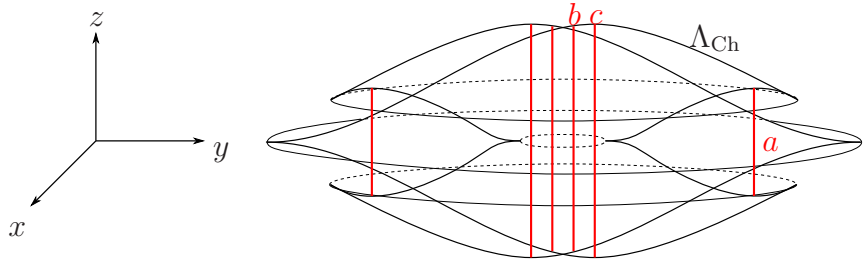
\begin{figure}[t]
\centering
\labellist
\pinlabel $x$ at -5 11
\pinlabel $y$ at 72 40
\pinlabel $z$ at 25 90
\pinlabel $\color{red}a$ at 279 43
\pinlabel $\color{red}b$ at 205 92
\pinlabel $\color{red}c$ at 213 91
\pinlabel $\Lambda_{\OP{Ch}}$ at 258 83
\endlabellist
\includegraphics{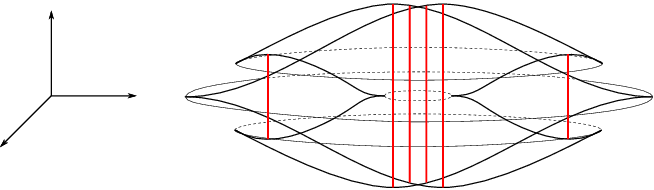}
\caption{Front projection of the Legendrian lift of the threefold Bohr--Sommerfeld cover of the Chekanov torus placed inside a Darboux ball. There are three Bott $S^1$-families of Reeb chords.}
\label{fig:Chekanov}
\end{figure}

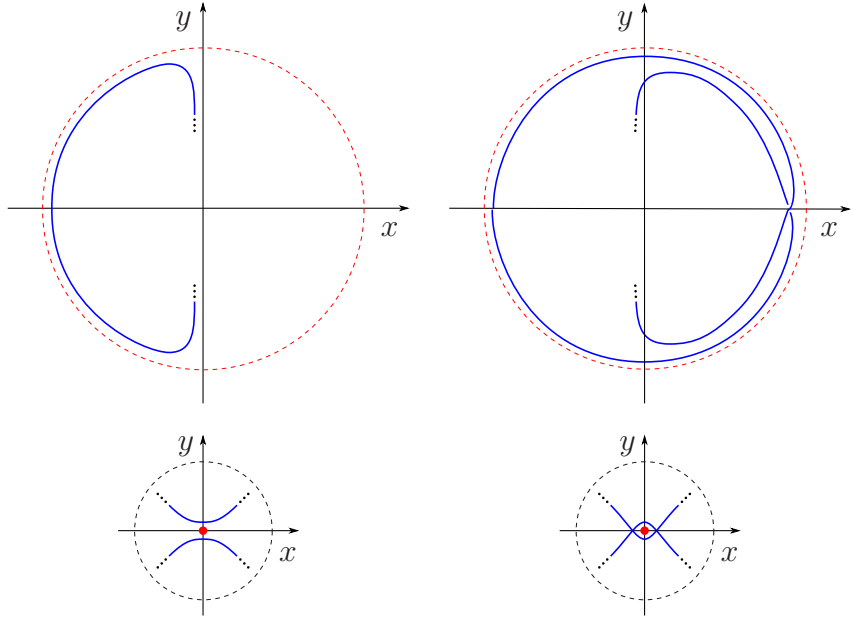
\begin{figure}[h!]
\centering
\labellist
\pinlabel $x$ at 180 182
\pinlabel $x$ at 387 182
\pinlabel $y$ at 291 280
\pinlabel $y$ at 83 280
\pinlabel $x$ at 132 29
\pinlabel $x$ at 340 29
\pinlabel $y$ at 84 80
\pinlabel $y$ at 292 80
\endlabellist
\includegraphics[scale=0.8]{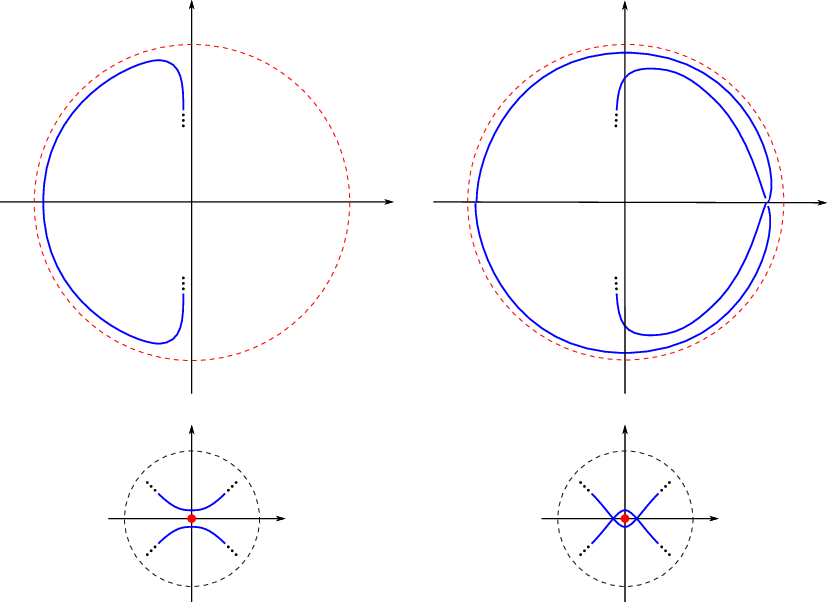}
\caption{Top left: the image under $f \colon \Sigma \to \dot{D}^2_{1/2}$ of a Lagrangian cylinder inside $\overline{\Sigma}.$ Top right: the cylinder after a Lagrangian regular homotopy that slides it over the embedded circle $S_{(1,1),1}^{(1,1)} \subset \overline{\Sigma}$ (which lives over the boundary of $\dot{D}^2_{1/2}$ as shown in red). The bottom pictures show the respective intersection of the Lagrangian in a fibre of the normal bundle $g\colon D_\epsilon^2 \times S_{(1,1),1}^{(1,1)} \to S_{(1,1),1}^{(1,1)}$ inside $\overline{\Sigma}$. {   We refer to Section \ref{sec:maincase} for the definition of the fibrations $f$ and $g$. In particular, recall that the intersection is invariant under multiplication with $-1$.}}
\label{fig:overinfinity}
\end{figure}

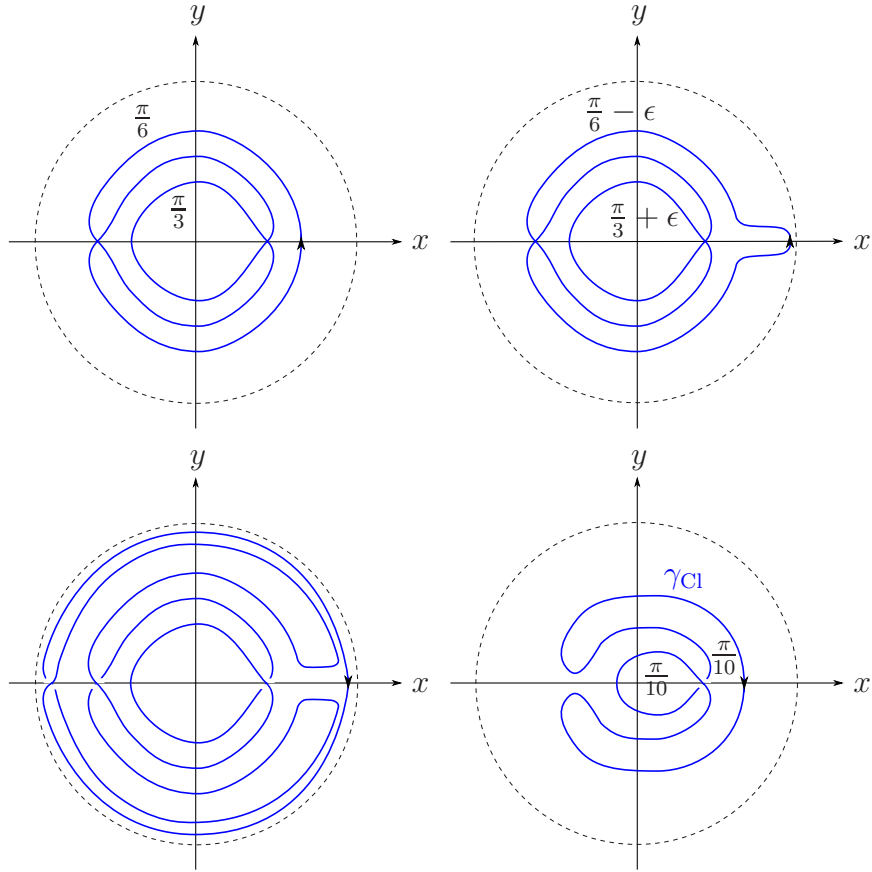
\begin{figure}[h!]
\vspace{3mm}
\centering
\labellist
\pinlabel $\frac{\pi}{3}$ at 80 310
\pinlabel $\frac{\pi}{6}$ at 63 353
\pinlabel $\frac{\pi}{6}-\epsilon$ at 288 356
\pinlabel $\frac{\pi}{3}+\epsilon$ at 299 306
\pinlabel $\frac{\pi}{10}$ at 305 90
\pinlabel $\frac{\pi}{10}$ at 337 100
\pinlabel $x$ at 402 296
\pinlabel $x$ at 402 87
\pinlabel $x$ at 194 87
\pinlabel $x$ at 194 296
\pinlabel $y$ at 297 403
\pinlabel $y$ at 297 193
\pinlabel $y$ at 89 403
\pinlabel $y$ at 89 193
\pinlabel $\color{blue}\gamma_{\OP{Cl}}$ at 318 137
\endlabellist
\includegraphics[scale=0.8]{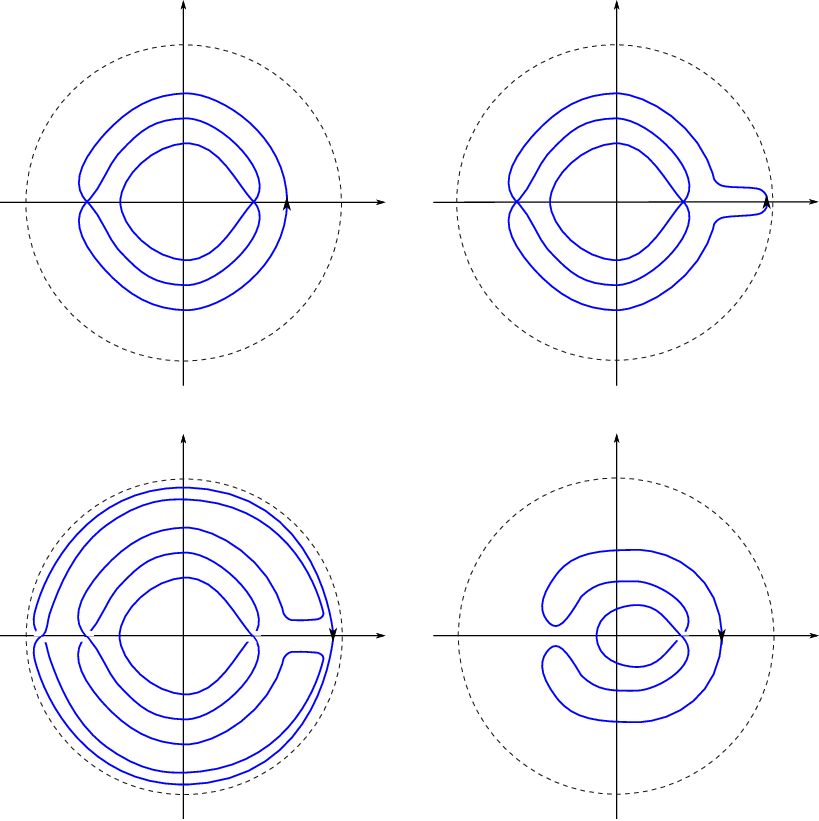}
\caption{The sequence of curves depicts the projection to $\dot{D}^2_{1/2}$ of a homotopy of Bohr--Sommerfeld immersions contained inside $\overline{\Sigma}\subset \CP^2$ under the fibration$f \colon \Sigma \to \dot{D}^2_{1/2}.$ The numbers denote approximate symplectic areas of the regions with respect to the symplectic form from Lemma \ref{lem:pullback}, for which $\dot{D}^2_{1/2}$ has total area $\pi/2.$   The embeddedness of the Legendrian lifts readily follows since an arc that connects two sheets at double point never bounds a region whose signed symplectic area takes value in $\pi\cdot\Z$ throughout the regular homotopy.  \color{black} }
\label{fig:homotopy-clifford}
\end{figure}

\begin{figure}[h!]
\vspace{3mm}
\centering
\labellist
\pinlabel $\frac{\pi}{3}$ at 103 320
\pinlabel $\frac{\pi}{6}$ at 50 307
\pinlabel $\frac{\pi}{10}$ at 309 71
\pinlabel $\frac{\pi}{10}$ at 332 92
\pinlabel $x$ at 402 296
\pinlabel $x$ at 402 87
\pinlabel $x$ at 194 87
\pinlabel $x$ at 194 296
\pinlabel $y$ at 297 403
\pinlabel $y$ at 297 193
\pinlabel $y$ at 89 403
\pinlabel $y$ at 89 193
\pinlabel $\color{blue}\gamma_{\OP{Ch}}$ at 320 126
\endlabellist
\includegraphics[scale=0.8]{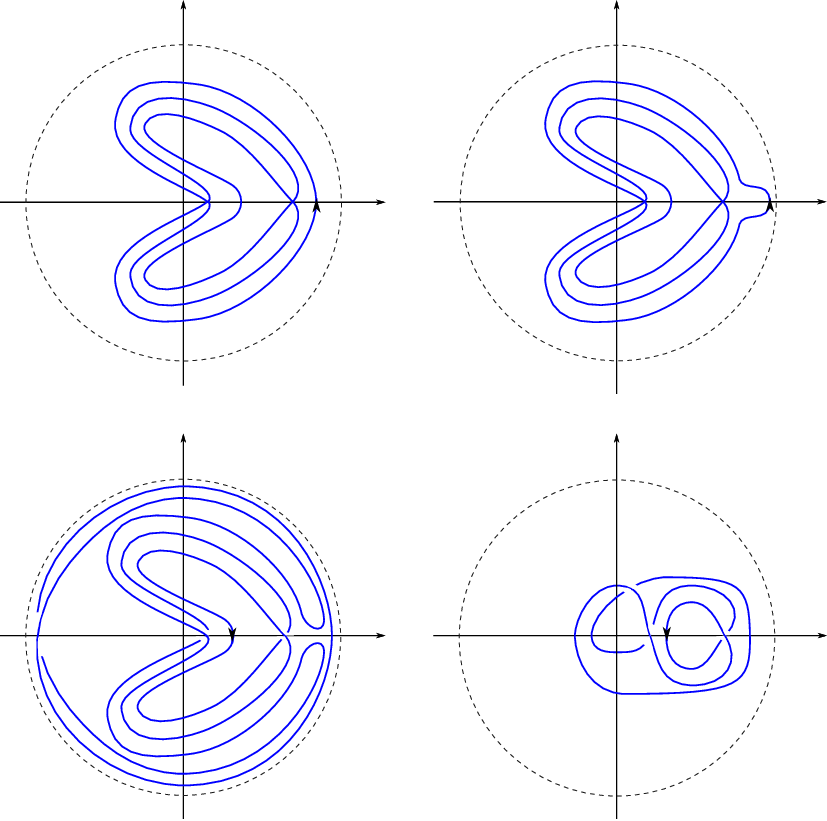}
\caption{A homotopy of Bohr--Sommerfeld immersions of tori inside $\overline{\Sigma}$ which starts at the canonical threefold Bohr--Sommerfeld cover of the Chekanov torus and ends at the Lagrangian projection of $\Lambda_{\OP{Ch}}.$ The curves depict the projections under the fibration $f \colon \Sigma \to \dot{D}^2_{1/2}.$   The embeddedness of the Legendrian lifts readily follows since an arc that connects two sheets at double point never bounds a region whose signed symplectic area takes value in $\pi\cdot\Z$ throughout the regular homotopy.}
\label{fig:homotopy-chekanov}
\end{figure}

\subsection{Proof of Theorem \ref{thm:frontsoflifts}}
\label{sec:frontsoflifts}

Both the Clifford and Chekanov tori live inside the  hypersurface $\overline{\Sigma} \subset \CP^2$ described in Subsection \ref{sec:maincase}  \color{black} , and they project to simple closed curves under the standard Lefschetz fibration $f \colon \Sigma \to \dot{D}^2_{1/2}$ described in the same; the former encircles the origin, while the latter does not. The sought Legendrian isotopies will be constructed inside the prequantisation bundle above the same hypersurface $\overline{\Sigma} \subset \CP^2$   by alluding to the techniques from Section \ref{sec:lefschetz} \color{black} . Here we describe the corresponding Lagrangian projections, which thus is a Lagrangian regular homotopy $\widetilde{L}_t$ of Bohr--Sommerfeld immersions inside the same. While doing this one must of course also take precaution so that no self-intersections arise in the Legendrian lifts.

Recall that a self-intersection arises precisely when the potential for the symplectic action form (this potential is a well-defined function with values in $\Z/\pi\Z$\blk) takes the same value at the two different sheets at a self-intersection of the Lagrangian immersion $\widetilde{L}_t$. In view of Lemma \ref{lem:pullback}, this translates to the property that the symplectic area enclosed by the sub-arc in $f(\widetilde{L}_t) \subset \dot{D}^2_{1/2}$ which starts and ends at the two sheets bounds a region whose signed symplectic area takes value in $\pi\cdot\Z$. \color{black}

Recall that any immersed Lagrangian torus inside $\Sigma$ projects to an immersed closed curve inside $\dot{D}^2_{1/2}$ under the fibration $f$ by Lemma \ref{lem:char}, and that conversely preimages of such curves are Lagrangian immersions. The sought regular homotopies are described in Figures \ref{fig:homotopy-clifford} and \ref{fig:homotopy-chekanov}, respectively, by regular homotopies of immersed curves inside $\dot{D}^2_{1/2},$ with one caveat: at one moment we must let our Lagrangian immersion pass through the intersection $S_{(1,1),1}^{(1,1)}$ of the line at infinity of $\CP^2$ and $\overline{\Sigma}.$ In that region $f$ is no longer suitable for describing the Lagrangian, and we instead use Corollary \ref{cor:symmetry}. We proceed with some more details.

We start on the top left in either of Figures \ref{fig:homotopy-clifford} or \ref{fig:homotopy-chekanov}. Here we see the threefold canonical Bohr--Sommerfeld covers of the respective monotone Lagrangian tori that have been generically perturbed inside the subset $\overline{\Sigma}$ through Lagrangian immersions.

Going from the top right projection to the bottom left projection we must let the Legendrian pass over the line at infinity of $\CP^2,$ i.e.~the Lagrangian intersects $S_{(1,1),1}^{(1,1)} \subset \overline{\Sigma}.$ The fact that the Lefschetz fibration $(z_1,z_2) \mapsto z_1 \cdot z_2$ defined on $\C^2$ obtains a singular fibre (a twofold branched cover of the line at infinity) when extended to $\CP^2$ makes this move slightly complicated. Instead of using $f$ we therefore pass to the description given in Corollary \ref{cor:symmetry} of a neighbourhood of the intersection $S_{(1,1),1}^{(1,1)}$ of the line at infinity and $\overline{\Sigma}$; see Figure \ref{fig:overinfinity}.

In the end we produce the Lagrangian projection living over the immersed curves $\gamma_{\OP{Ch}}$ and $\gamma_{\OP{Cl}}$ shown on the bottom right in Figures \ref{fig:homotopy-clifford} and \ref{fig:homotopy-chekanov}, respectively. Note that these curves bound signed symplectic area equal to zero for the symplectic form on $\dot{D}^2_{1/2}$ from Lemma \ref{lem:pullback}. Shrinking these curves are thus obviously a regular homotopy through exact, and hence Bohr--Sommerfeld, Lagrangian immersions inside $\C^2.$ Shrinking these curves sufficiently, the corresponding Legendrians end up inside a small contact Darboux ball.

We end by showing how to make the identification of the front-projections of the Legendrian tori in the Darboux balls constructed above, with the Legendrian tori which have the $S^1$-symmetric front projections depicted in Figures \ref{fig:Clifford} and \ref{fig:Chekanov}. Note that a priori it is clear that  the Legendrians produced here have front projections that satisfy an $S^1$-symmetry around the $z$-axis; this is due to the fact that the underlying Lagrangian immersions are invariant under the $S^1$-action $(z_1,z_2)\mapsto (e^{it}z_1,e^{-it}z_2)$.

First, consider the curves $\gamma_{\OP{Cl}}$ and $\gamma_{\OP{Ch}}$ inside $\dot{D}^2$ depicted in the bottom right of Figures \ref{fig:homotopy-clifford} and \ref{fig:homotopy-chekanov}, which correspond to the projections of the lift of the Clifford and Chekanov torus, respectively, isotoped into the Darboux ball. Suitable double covers of these curves admit lifts under the two-fold nontrivial cover $\dot{D}^2\to \dot{D}^2$ (i.e. the two-fold cover of $D^2$ branched at the origin) to curves $\tilde{\gamma}_{\OP{Cl}}$ and $\tilde{\gamma}_{\OP{Ch}}$ in $\dot{D}^2$ which are invariant under the involution of $D^2$ given by multiplication with $-1$; see Figure \ref{fig:liftedcurves}. Note that the winding number around the origin of $\gamma_{\OP{Cl}}$ (resp. $\gamma_{\OP{Ch}}$) is equal to 1 (resp. 0), which implies that the double cover which can be lifted to a closed curve which is invariant under multiplication by $-1$ must be the nontrivial cover (resp. the trivial cover).

 Then consider the front projection of the Legendrian knot in $\R^3$ which corresponds to the Legendrian lifts of the exact Lagrangian immersions $\tilde{\gamma}_{\OP{Cl}}$ and $\tilde{\gamma}_{\OP{Ch}}$ inside $\R^2$. Since these Lagrangian projections are invariant under multiplication with $-1$, it follows that the Legendrian knots given as their lifts have front projections that are invariant under the involution $x \mapsto -x$. The sought two-dimensional fronts of the Legendrian tori $\Lambda_{\OP{Cl}}$ and $\Lambda_{\OP{Ch}}$ can finally be obtained by performing a symmetric $S^1$-spin of the fronts of these knots around the $z$-axis.  \color{black}
\qed

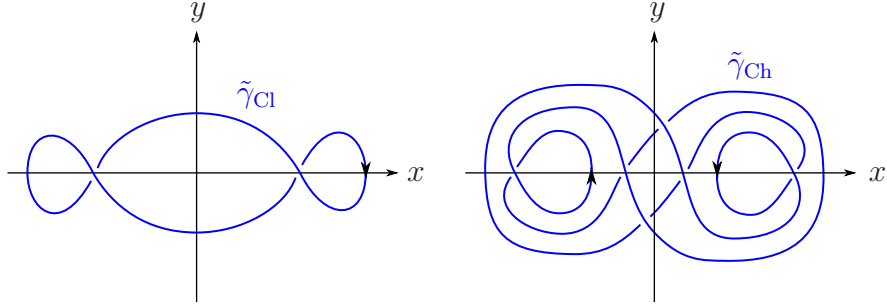
\begin{figure}[h!]
\vspace{3mm}
\centering
\labellist
\pinlabel $x$ at 328 49
\pinlabel $x$ at 154 49
\pinlabel $\color{blue}\tilde{\gamma}_{\OP{Cl}}$ at 94 79
\pinlabel $\color{blue}\tilde{\gamma}_{\OP{Ch}}$ at 280 90
\pinlabel $y$ at 245 110
\pinlabel $y$ at 72 110
\endlabellist
\includegraphics[scale=1.0]{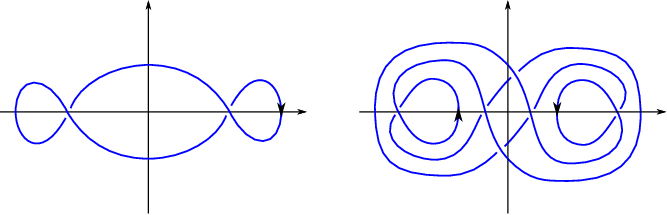}
\caption{The curves are two-fold covers of the curves $\gamma_{\OP{Cl}}$ and $\gamma_{\OP{Cl}}$ depicted in Figures \ref{fig:homotopy-clifford} and \ref{fig:homotopy-clifford} which have been lifted to $D^2$. These lifts are exact Lagrangian immersions which are invariant under scalar multiplication by $-1$, and their Legendrian lifts are thus invariant under the contactomorphism $(x,y,z) \mapsto (-x,-y,z).$}
\label{fig:liftedcurves}
\end{figure}

\section{Computations of the Chekanov--Eliashberg algebras}
\label{sec:dga}

For Legendrian surfaces inside the standard contact vector space $(\R^5,dz-ydx)$ the full Chekanov--Eliashberg algebra as constructed in \cite{LegendrianContactPxR} can readily be computed by using Ekholm's theory of gradient flow-trees \cite{MorseFlowTrees}. Here we compute this invariant for the Legendrian tori $\Lambda_{\OP{Cl}}$ and $\Lambda_{\OP{Ch}}$ when considered as Legendrian submanifolds inside a strict contact Darboux ball. The DGA inside the full $(S^5,\alpha_{\OP{st}})$ is of a more complicated form since it has infinitely many generators;   it is nevertheless quasi-isomorphic to the DGA computed in the Darboux ball, as was shown in~\cite{Refined}. \color{black}
\begin{remark}
Any Legendrian isotopy can be assumed to miss a generic point after a small perturbation. By \cite[Proposition 2.1.8]{Geiges} it may thus be assumed to actually be confined to the contact Darboux ball itself.  A priori, the Chekanov--Eliashberg algebra computed in a contact Darboux ball is thus an invariant of Legendrian isotopy inside the entire contact sphere as well   (i.e.~one does not need to rely on the aforementioned quasi-isomorphism for this). \color{black}
\end{remark}

We have found representatives for the Legendrians $\Lambda_{\OP{Cl}}$ and $\Lambda_{\OP{Ch}}$ that live in a Darboux ball $(\C^2 \times \R,dz-ydx)$ and which moreover satisfy a rotational symmetry around the $z$-axis. This is equivalent to the fact that they live inside $\Sigma \subset \C^2$ as considered in Section \ref{sec:lefschetz}, and can thus be described by their projections under $f \colon \Sigma \to \dot{D}^2_{1/2}.$ Their projections are the closed immersed curves $\gamma_{\OP{Cl}}$ and $\gamma_{\OP{Ch}} \subset \C^*$ shown at the bottom right in Figures \ref{fig:homotopy-clifford} and \ref{fig:homotopy-chekanov}, respectively. This symmetry will significantly facilitate the computations of their Chekanov--Eliashberg algebras.

\begin{remark}
\label{rem:spun}
In the case when the origin is contained in the unbounded component of $\C \setminus \gamma,$ the corresponding  $S^1$-symmetric Legendrian torus which projects to $\gamma$ is what usually is called the {\bf $S^1$-spun of a Legendrian knot}. More precisely, it is the $S^1$-spun of the Legendrian knot in $\R^3$ with Lagrangian projection given by the curve $\gamma \subset \C$;  \color{black}  see \cite{EkholmKalman} for the definition, where Ekholm--K\'{a}lm\'{a}n moreover give a description of the full Chekanov--Eliashberg algebra of such a Legendrian torus entirely in terms of the Chekanov--Eliashberg algebra of the knot alone. In other words, this means that the differential of the torus is determined by counts of polygons in $\C$ with boundary on $\gamma$ and precisely one positive puncture;  in particular, this uses Chekanov's formulation of the Chekanov--Eliashberg algebra \cite{DiffGradedAlgebraLegLinks}. \color{black}
\end{remark}

In the cases of interest here the origin of $\C$ is not contained in the unbounded component of $\C \setminus \gamma$ (i.e.~the complement of the projection of the Legendrian under the Lefschetz fibration $f$). Consequently, this case corresponds not to the aforementioned $S^1$-spun of a Legendrian knot as described in \cite{EkholmKalman}, but rather the generalisation to a so-called {\bf symmetric $S^1$-spun} of a Legendrian knot which (to the knowledge of the authors was) first appeared in \cite[Section 5.5.2]{CellularI}. Recall that the symmetric $S^1$-spun can be performed to any Legendrian knot which is invariant under the contactomorphism $(x,y,z) \mapsto (-x,-y,z)$; this is e.g.~the case here, as shown in Figure \ref{fig:liftedcurves}. Computing the full DGA in this case is  significantly  \color{black}  harder, but it can be done e.g.~using the cellular DGA by Rutherford--Sullivan \cite{CellularI,CellularII}. This was done by Li \cite[Proposition 6.2]{Li} in a few particular cases. While we manage to compute the full DGA of $\Lambda_{\OP{Cl}}$ here, we only perform a partial computation of the DGA of $\Lambda_{\OP{Ch}}$; we postpone the computation of its full DGA to future work.

\subsection{Generators and grading}
\label{sec:computing}

Let $\gamma \subset \C^*$ be a generic immersion of a closed curve which is the image of   (the Lagrangian projection of)  \color{black}  a Legendrian torus $\Lambda \subset (\C^2 \times \R,dz-ydx)$ under the above standard Lefschetz fibration $f \colon \C^2 \to \C$. The Legendrian condition translates to an exactness property of this curve (i.e.~that it bounds zero symplectic area)  with respect to the symplectic form on $\C^*$ described in Lemma \ref{lem:pullback}. \color{black}

The Chekanov--Eliashberg algebra of $\Lambda$ is the unital noncommutative DGA $(\mathcal{A}(\Lambda),\partial)$ freely generated by the Reeb chords of a generic perturbation of $\Lambda$ (in order to make all Reeb chords transverse) over the commutative group ring $\F[H_1(\Lambda)].$  The Chekanov--Eliashberg algebra always admits a well-defined grading in $\Z$
when the group ring coefficients are used. Recall that the coefficients themselves are graded by the value of the Maslov class on the corresponding cycle. In particular this means that reducing coefficients via the canonical algebra map $\F[H_1(\Lambda)]\to\F$ yields a Chekanov--Eliashberg algebra with coefficients in $\F$ which has well-defined grading only modulo the image of the Maslov class. \color{black}  \color{black}

\emph{A choice of basis of $H_1(\Lambda)$:} We begin by pin-pointing a $\Z$-basis $\langle \mu,\lambda \rangle =H_1(\Lambda).$ This gives an identification $\F[H_1(\Lambda)] \cong \F[\mu^{\pm 1},\lambda^{\pm 1}]$ with the ring of Laurent polynomials in two variables.

First we take $\mu \in H_1(\Lambda)$ that is represented by (a suitably  oriented \color{black} ) $S^1$-fibre of $\Lambda$ under the Lefschetz fibration;   in other words, $\mu$ is represented by a curve which is a simple orbit of the $S^1$-symmetry. A representative of $\mu$ can e.g.~be taken to be  \color{black}  parallel to one of the circular cusp edges in the front projection of $\Lambda.$ Then we take $\lambda \in H_1(\Lambda)$ to be represented by a simple closed curve inside $\Lambda$ which  projects under the Lefschetz fibration $f$ to the curve $\gamma \subset \C^*$ covered precisely once. Note that, even after fixing an orientation, the latter property determines the class $\lambda \in H_1(\Lambda)$ only up to the addition of a term of the form $k\cdot\mu \in H_1(\Lambda),$  $k \in \Z.$

Recall that the front projection of $\Lambda$ in $\R^2 \times \R$ satisfies a rotational $S^1$-symmetry around the $z$-axis; the singular locus thus consists of circular cusp edges together with double-cones whose cone points are situated on the $z$-axis. Any curve which represents $\lambda$ and projects to $\gamma$ thus satisfies the property that it intersects each circular cusp edge transversely exactly once. It is useful to make the representative of $\lambda$ even more well-behaved with respect to the front projection of $\Lambda$. Namely, we require that the representative of $\lambda$ to be a piecewise smooth embedding which is contained entirely inside the hyperplane $\{y=0\} \subset \R^2 \times \R$. This determines the class $\lambda$ uniquely \emph{except} for an ambiguity at each cone-point of the front; when traversing a double-cone there are two possible choice one can make, with different choices resulting in the addition of a term $\pm\mu$ to the class $\lambda$. \color{black}

\emph{The Reeb chords:} Closed Legendrians that satisfy the $S^1$-symmetry considered here never have transverse Reeb chords; the preimage of a double point in the base $\C$ of the Lefschetz fibration is always a \emph{circle} of Reeb chords. In the case when the intersection point of the projection is a transverse double point, the circle of Reeb chords is still generic in the Bott sense. For a generic perturbation constructed by a Morse function on the Bott manifold (i.e.~$S^1$) having precisely two critical points, each Bott manifold gives rise to precisely two Reeb chords $x$ and $\hat{x}$ corresponding to the minimum and maximum of the Morse function, respectively. The grading of these generators are then $|\hat{x}|=|x|+1$ and $|x|=\OP{CZ}(x)-1,$ where $\OP{CZ}(x)$ moreover coincides with the Bott version of the Conley--Zehnder index for the corresponding $S^1$-family of Reeb chords. For a small generic perturbation we thus get precisely two Reeb chords $x$ and $\hat{x}$ for any double point of $\gamma.$ Recall that the gradings of the Novikov parameters $\mu,\lambda$ are given by their respective Maslov indices.

The following result allows us to compute the gradings entirely in terms of data on the projection $\gamma \subset \C.$
\begin{lem}
  \begin{enumerate}
  \item The Maslov index of the Novikov parameter \color{black} $\mu$ vanishes, while the Maslov index of $\lambda$ is equal to twice the tangent winding number $\tau \in \Z$ of the closed curve $\gamma$  (i.e.~the winding number of the tangent map $\dot{\gamma} \in \R^2 \setminus \{0\}$ with respect to the origin) \color{black} .
  \item The Conley--Zehnder index $\OP{CZ}(x)$ is equal to the Conley--Zehnder index of the corresponding transverse self-intersection of $\gamma \subset \C$ as defined for immersed curves with transverse double-points in \cite{DiffGradedAlgebraLegLinks}.
  \end{enumerate}
\end{lem}
\begin{proof}

 (1): Recall that
$$(\C^2 \setminus \{z_1z_2=0\},\omega_0) \supset \Sigma=f^{-1}(\dot{D}^2)$$
is symplectomorphic to a neighbourhood of the zero section of $(T^*\T^2,d(p\,dq))$, by a symplectomorphism which takes the Lagrangian Clifford torus $S^1 \times S^1 \subset \C^2$ to the zero section $0_{\T^2} \subset T^*\T^2$. There Maslov index in this non-simply connected symplectic domain can be defined by the following two \emph{different} natural symplectic trivialisations: the trivialisation of $\C^2$ and the canonical trivialisation of $T^*\T^2$. When computed in $(\C^2,\omega_0)$, each of $S^1 \times \{\pt\}$ and $\{\pt\} \times S^1$ are cycles of Maslov index \emph{two}, while they both have Maslov index \emph{zero} when computed in $T^*\T^2$. Note that the Maslov cycle used to define the Maslov class when $T^*\T^2$ is endowed with its canonical trivialisation is the Lagrangian distribution of cotangent fibres (i.e.~the Maslov index is defined by intersections of the Lagrangian Gauss map with the vertical tangent planes).

The Maslov class of a Lagrangian $L \subset \Sigma \subset \C^2$ computed relative the latter Maslov cycle amounts to taking the signed count $r$ of radial tangencies of the projection $\gamma \subset \C^*$ of the Lagrangian under $f \to \dot{D}^2$. If one instead wants the Maslov class induced by the trivialisation of $\C^2$ one has to add $2w$, where $w$ is the winding number of $\gamma$ around the origin. Since the expression $r+2w$ readily can be seen to be the tangent winding number of $\gamma$, this proves the claim. \color{black}

(2): In view of the formula for the Maslov index in (1), the degrees can be computed analogously to the degrees of the Reeb chords on an $S^1$-spun of a Legendrian knot in terms of the degrees of the Reeb chords of the knot itself; see \cite{EkholmKalman} for this computation.
\end{proof}

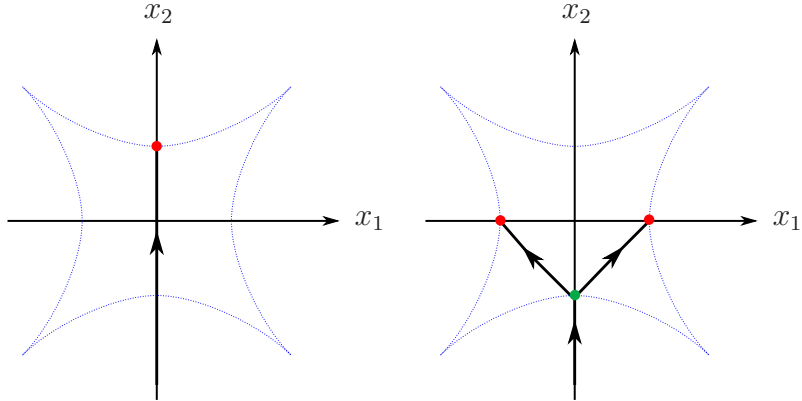
\begin{figure}[htp]
\vspace{3mm}
\centering
\labellist
\pinlabel $x_1$ at 196 45
\pinlabel $x_2$ at 143 97
\pinlabel $x_1$ at 91 45
\pinlabel $x_2$ at 38 97
\endlabellist
\includegraphics[scale=1.5]{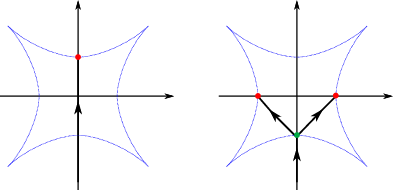}
\vspace{3mm}
\caption{The blue curve depicts the caustic of the front projection of a generic resolution of the front cone that involves four swallow-tail singularities. We also depict the two possibilities for a rigid partial gradient flow-tree that involves a single front-cone; the tree on the left has a single vertex being (an edge-vertex at the cusp edge), while the tree on the right has three vertices (one three-valent vertex of type $Y_1$ and two edges-vertices at two cusp-edges).}
\label{fig:coneres}
\end{figure}

\subsection{Computations for $\Lambda_{\OP{Cl}}$}
\label{sec:comp-cliff}

The Chekanov--Eliashberg algebra $(\mathcal{A}(\Lambda_{\OP{Cl}}),\partial)$ of the Legendrian torus $\Lambda_{\OP{Cl}} \subset (\R^5, \alpha_{st})$ was computed by the first author in \cite{KnottedLegendrianSurface}. Here we redo this computation in the present algebraic setting (in which the Novikov coefficients $\C[H_1(\Lambda_{\OP{Cl}})]$ do not commute  with the Reeb chord generators).

The $S^1$-symmetric version of $\Lambda_{\OP{Cl}}$ has a single $S^1$-Bott manifold of Reeb chords that correspond to the unique double point of the curve shown at bottom right in Figure \ref{fig:homotopy-clifford}. A small generic perturbation of the $S^1$-symmetric front of $\Lambda_{\OP{Cl}}$ produces two generic chords $a$ and $\hat{a},$ where we compute $|\hat{a}|=|a|+1$ and $|a|=1.$   Recall that the Novikov parameters have grading $|\mu|=|\lambda|=0$ (the Maslov class of $\Lambda_{\OP{Cl}}$ vanishes). \color{black}
\begin{thm}
\label{thm:dgacliff}
  For the spin structure on $\Lambda$ given by the trivialisation of $T\Lambda$ which is induced by the Lie group structure on $\Lambda$, and for  \color{black}  suitable choices of capping paths and basis $\langle \mu,\lambda \rangle = H_1(\Lambda)$ we have
  \begin{align*}
& \partial a=\partial a = 1+\lambda(1+\mu),\\
& \partial \hat{a}=a-\mu a \mu^{-1}.
  \end{align*}
In particular, the augmentation variety of $(\mathcal{A}(\Lambda_{\OP{Cl}}),\partial)$ is equal to the one-dimensional complex pair of pants
$$\OP{Sp}(\C[\mu^{\pm1},\lambda^{\pm1}]/\langle 1+\lambda(1+\mu)\rangle)$$
and for two augmentations $\varepsilon_0, \varepsilon_1 \colon H_0(\mathcal{A}) \to \C$ the bilinearised Legendrian contact homology group satisfies
$$ LCH_*^{\varepsilon_0,\varepsilon_1}(\Lambda_{\OP{Cl}}) = \begin{cases} H_{*-1}(S^1;\C), & \varepsilon_0=\varepsilon_1, \\ 0, & \varepsilon_0 \neq \varepsilon_1.
\end{cases}$$
\end{thm}
\begin{proof}
In \cite{KnottedLegendrianSurface} the differential
$$ \partial a = 1+\lambda(1+\mu)$$
was computed for some suitable (but unspecified) choice of spin
structure (and choice of sign of the generator ``$a$''), where the
terms are in bijection with the rigid pseudoholomorphic discs with
exactly one puncture.  The term ``1'' corresponds to the rightmost
immersed teardrop bounded by $\gamma_{\OP{Cl}}$ (c.f.~Figure
\ref{fig:homotopy-clifford}), i.e.~which does not contain the
origin. The capping paths of $a$ are chosen to be the two curves in
$\Lambda$ that correspond to the edge of this flow-tree, and the
base point is thus located precisely at the cusp-edge. The remaining
two terms both correspond to the immersed teardrop to the left
bounded by $\gamma_{\OP{Cl}}$, i.e.~that pass through the origin.
After generically resolving the front-cone singularity as described
in \cite{KnotContHom} or \cite{KnottedLegendrianSurface}, these
flow-trees flow towards the cone-point where they then proceed in
either of the two manners shown in Figure \ref{fig:coneres}.

In order to show that the signs of all three discs are the same for the Lie group spin structure we argue as follows.  First recall that the sign of a pseudoholomorphic disc is induced by a so-called coherent orientation of the moduli space. This coherent orientation, in turn, depends on the choice of spin structure of the Legendrian. The sign of the disc depends on this spin structure in the following precise manner. Assume that we have fixed choices of capping paths of the Reeb chords. Adjoining suitable such capping paths at the punctures of the disc closes it up into a closed curve with boundary on the Legendrian. If the spin structure along this closed curve is changed (while keeping the same orientation), then the coherent orientation (i.e.~the sign obtained when counting the disc) is changed; see \cite[Section 4.4]{Ekholm:Orientations} for more details. \color{black}

First, we claim that the discs counted by $\partial a$  that contribute to the term $\lambda$, \color{black}  i.e.~which arises from the resolution shown on the left in Figure \ref{fig:coneres}, and the disc that corresponds to the term ``$1$'' both are counted with the same sign; hence these signs can both be taken to be $+1$ without loss of generality. The claim follows by the computation of DGA for the $S^1$-spun of the standard unknot with a single Reeb chord, since the two aforementioned discs can be identified with the two discs that contribute to $\partial a'=1-\lambda$ for the unique generator $a'$ in degree one and the spin structure that extends over the solid torus in which $\lambda$ is a longitude (this is different from the Lie group spin structure along $\lambda$). To see the sign in the case of the $S^1$-spun of the unknot, note that the latter spin structure extends to the standard exact Lagrangian filling of the aforementioned $S^1$-spun by a solid torus in the symplectisation, in which the two discs become the endpoints of a one-dimensional moduli space of discs with boundary on the filling.

Second, we claim that the disc counted by $\partial a$ that corresponds to the term $\lambda\mu$ also is counted with the same sign as the previous discs. If not,  after \emph{changing} the Lie group spin structure precisely along the cycle $\mu$, the differential would become $\partial a=1+\lambda(1+\mu)$ for this new spin structure, and thus $\partial(a)=0$ when reducing the coefficients to $\Z_3$ by $\mu,\lambda\mapsto 1$. Note that such a change applied to the Lie group spin structure produces the spin structure that extends to the exact Lagrangian cobordism to the loose unknot constructed in Section \ref{sec:subflex} (see Figure \ref{fig:loosesphere}).  Since the spin structure extends, the augmentation with $\Z_3$-coefficients pulls back to an augmentation of the loose sphere, which is a contradiction (the loose sphere has an acyclic DGA for any choice of coefficients, and thus does not admit augmentations with any choice of coefficients); see~Proposition \ref{sub-looseacyclic}.

It should be noted that the argument above only pin-points the signs of the discs relative to each other. However, as far as the isomorphism class of the DGA is concerned, this is sufficient. Namely, to obtain a global change of signs, it suffices to make a change of variables $a \mapsto -a$ in the DGA.
 \color{black}

In the setting of the ``fully noncommutative Legendrian DGA'' considered here we also compute
$$ \partial \hat{a}=a-\mu a \mu^{-1},$$
where the two terms correspond to the two flow-lines in the Bott $S^1$-manifold of Reeb chords that appear after the generic perturbation. (The signs of the two terms can be determined by $\partial^2\hat{a}=0,$ after a choice of sign of the generator $\hat{a}.$) Here we have taken the capping paths of $\hat{a}$ to follow the $S^1$-manifolds of Reeb chords (before the Morse perturbation) to the corresponding endpoint of the chord $a,$ concatenated with the corresponding capping path for the chord $a.$

For the computation of the bilinearised Legendrian contact homology, we can either allude to Theorem \ref{thm:lchgeneral} below or perform the following computation by hand.

For the explicit computation, we first note that the bilinearised differential satisfies
$$\partial^{\varepsilon_0,\varepsilon_1}(\hat{a})=(1-\varepsilon_0(\mu)\varepsilon_1(\mu)^{-1})a.$$
Finally, two augmentations are the same if and only if they take the same value on the variable $\mu \in \C[H_1(\Lambda_{\OP{Ch}})],$ as follows from the relation $\lambda^{-1}=-(1+\mu).$
\end{proof}
In other words, the homology groups $LCH_*$ behave like the $\OP{Ext}$ groups for different skyscraper sheaves on the augmentation variety. This should be compared to Nadler's computation \cite{Nadler} based upon the technique of sheaves with microsupport on the Lagrangian cone over $\Lambda_{\OP{Cl}} \subset \partial   B^6.$

\subsection{Computations for $\Lambda_{\OP{Ch}}$}
\label{sec:comp-ch}
Here we compute the homology of the Chekanov--Eliashberg algebra $(\mathcal{A}(\Lambda_{\OP{Ch}}),\partial)$ of $\Lambda_{\OP{Ch}} \subset (\R^5, \alpha_{st})$ in degree zero. This is sufficient for computing its augmentation variety, and to distinguish it from the Legendrian torus $\Lambda_{\OP{Cl}}$ up to Legendrian isotopy. The computation of the full DGA is postponed to future studies. For the bilinearised Legendrian contact homologies, we refer to the general computation in Section \ref{sec:lchgeneral}.

The $S^1$-symmetric version of $\Lambda_{\OP{Ch}}$ has the three $S^1$-Bott manifolds of Reeb chords corresponding to the double point shown in the bottom right of  Figure \ref{fig:homotopy-chekanov} \color{black} . The Bott indices are $|a|=1$, $|c|=3,$ and $|b|=2,$ for the rightmost, middle, and leftmost double points, respectively. A small generic perturbation thus gives us a Legendrian with six Reeb chords $a,\hat{a},$ $c,\hat{c},$ $b,\hat{b},$ with
$$ |a|=|\hat{a}|-1=1, \:\:\: |c|=|\hat{c}|-1=3, \:\:\: \text{and} \:\:\: |b|=|\hat{b}|-1=2.$$  Again recall that the Novikov parameters are graded by $|\mu|=|\lambda|=0$. \color{black}
\begin{thm}
  For the   spin structure on $\Lambda$ given by the trivialisation of $T\Lambda$ which is induced by the Lie group structure on $\Lambda$, and for  \color{black}  suitable choices of capping paths and basis $\langle \mu,\lambda \rangle = H_1(\Lambda)$ we have
$$ \partial a=1+\lambda(1+\mu)^2,\:\:\partial\hat{a}=a-\mu a \mu^{-1},$$
and in particular $H_0(\mathcal{A},\partial)=\C[\mu^{\pm1},\lambda^{\pm1}]/\langle 1+\lambda(1+\mu)^2\rangle.$
\end{thm}
\begin{proof}
The version of the Legendrian that satisfies the $S^1$-symmetry shown in Figure \ref{fig:Chekanov} is degenerate for two different reasons: first, the Reeb chords come in $S^1$-Bott manifolds and, second, it has two double-cone singularities in its front projection both located inside the axis of the rotational symmetry.

We will compute the rigid gradient flow-trees \cite{MorseFlowTrees} after a generic perturbation which, by the same work, computes the differential for a suitable choice of almost complex structure. Since we only care about such flow-trees with a single positive puncture at the chord $a$ of degree $|a|=1,$ these flow-trees will have no additional (negative) punctures.

If we keep the $S^1$-Bott family of Reeb chords, then this is the same as a rigid gradient flow-tree with a positive puncture at the Bott family of chords itself, but with a fixed asymptotic constraint.

We start with the flow-trees that do not enter the region near the axis of symmetry. Here there is single rigid gradient flow-tree with a single positive puncture at $a,$ which is as shown in Figure \ref{fig:chekanov-disk1}. This flow-tree has precisely two vertices (one puncture at $a$ and one so-called edge-vertex at the cusp-edge) and one edge. We take the capping paths of $a$ to coincide with the two arcs on $\Lambda$ that correspond to the edge of the flow-tree; the basepoint is hence located precisely at the cusp-edge. With this choice, together with an appropriate choice of sign for the generator, we obviously get a disc that contributes to a term
$$ \partial(a)=\ldots + (\pm 1) + \ldots$$
of the differential.

We then consider gradient flow-trees which enter the region near the axis of symmetry. By action reasons, and using the fact that the positive puncture is unique, it is not possible for such a gradient flow-tree to have an edge that involves two sheets corresponding to \emph{different} cones. (The difference in $z$-coordinate for two such sheets is always greater than the length of $a.$) In other words, the only edges that we need to consider near the cone-region in this case are those which are associated to precisely one of the two cones. Such edges are given simply by a gradient flow towards the cone-point, and the vertex at the cone-point will be called a {\bf cone-vertex}. See Figure \ref{fig:chekanov-disk2} for an example of a partial flow tree with precisely two such vertices.

In the perturbed Legendrian which resolves the cone-point to a generic front, the rigid partial gradient flow-trees with cone-vertices in the sense described above can be resolved to honest rigid gradient flow-trees by the following process. First use the perturbations of the Legendrian described in \cite{KnotContHom} and \cite{KnottedLegendrianSurface} in order to resolve the front-cone singularities to a generic front. In this perturbation, every cone-vertex has precisely the two possibilities of a completion to an honest gradient flow tree that are shown in Figure \ref{fig:coneres}. Conversely, every rigid gradient flow tree corresponds to a unique partial rigid flow-tree after reintroducing the degenerate front-cone.

In view of the above, we are able to postpone the perturbation that makes the front generic near the cone points, and simply start by finding the rigid partial gradient flow-trees that have cone-vertices and a puncture at a fixed point on the Bott manifold of Reeb chords $a.$ There is precisely one such partial flow-tree; see Figure \ref{fig:chekanov-disk2}. By considering all possible completions of the latter partial flow-tree after the perturbation of the front cone, we thus find the remaining four rigid gradient flow-trees that contribute to the four terms
$$\partial a=\ldots + (\pm)\lambda\pm\lambda\mu\pm\lambda\mu\pm\lambda\mu^2 + \ldots$$
of the differential.

We have now found all five rigid flow-trees that contribute to $\partial a.$ As in the proof of Theorem \ref{thm:dgacliff} we can show that the disc contributing to the term ``$1$'' and the disc whose resolutions at both cone-points is as shown on the left in Figure \ref{fig:coneres} are counted with the same signs. A neck-stretching argument near the cone regions then show that the local sign contribution from the two resolutions shown in Figure \ref{fig:coneres} are the same at both cone points, from which we conclude that
$$ \partial a=1+\lambda(1 \pm \mu)^2$$
(as opposed to $ \partial a=1+\lambda(1 - \mu)(1+\mu)$). As in the proof of Theorem \ref{thm:dgacliff} the fact that $\Lambda_{\OP{Ch}}$ is subloose then pin-points the final sign in the last factor.

The computation of $\partial \hat{a}$ is the same as for $\mathcal{A}(\Lambda_{\OP{Ch}})$ in the proof of Theorem \ref{thm:dgacliff}.
\end{proof}

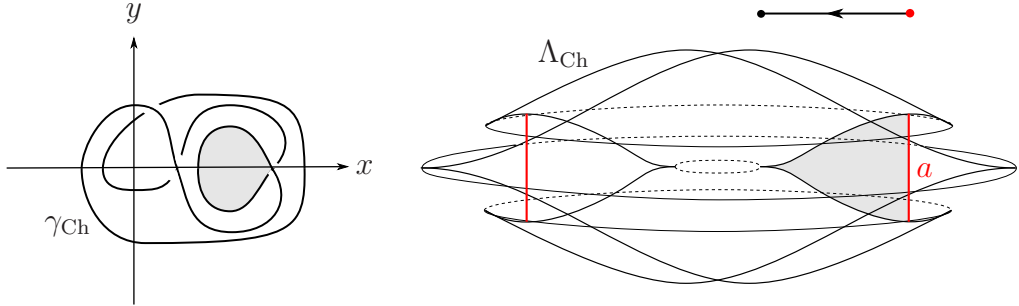
\begin{figure}[t]
\centering
\labellist
\pinlabel $\gamma_{\OP{Ch}}$ at 23 30
\pinlabel $x$ at 135 52
\pinlabel $y$ at 48 110
\pinlabel $\color{red}a$ at 346 51
\pinlabel $\Lambda_{\OP{Ch}}$ at 210 95
\endlabellist
\includegraphics{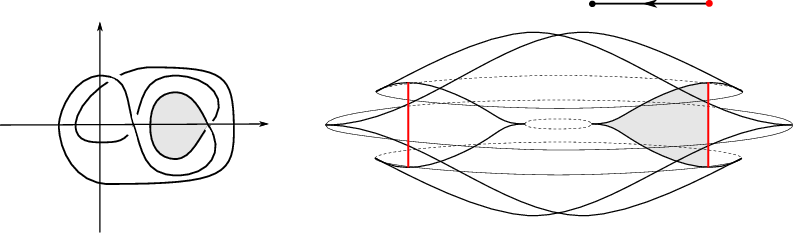}
\caption{A rigid gradient front-tree with precisely one edge and two vertices: one vertex being a positive puncture at $a,$ and one edge-vertex at the central cusp-edge.}
\label{fig:chekanov-disk1}
\end{figure}

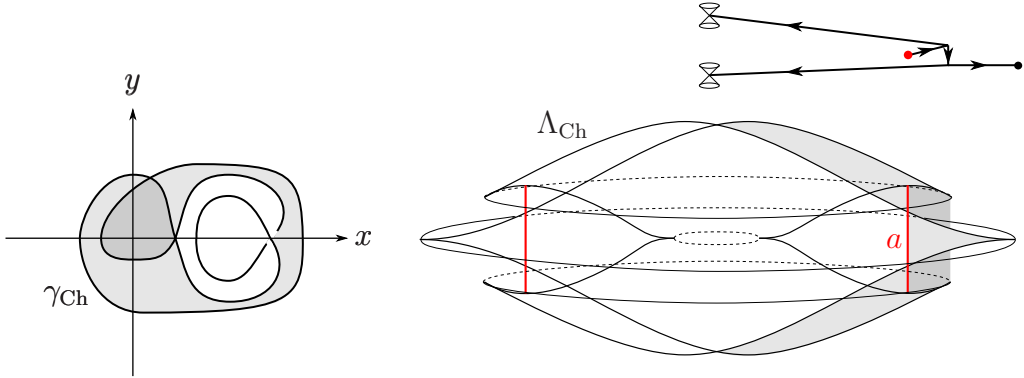
\begin{figure}[t]

  \centering
\labellist
\pinlabel $\gamma_{\OP{Ch}}$ at 23 30
\pinlabel $x$ at 135 52
\pinlabel $y$ at 48 110
\pinlabel $\gamma_{\OP{Ch}}$ at 23 30
\pinlabel $x$ at 135 52
\pinlabel $y$ at 48 110
\pinlabel $\color{red}a$ at 335 51
\pinlabel $\Lambda_{\OP{Ch}}$ at 210 95
\endlabellist
\includegraphics{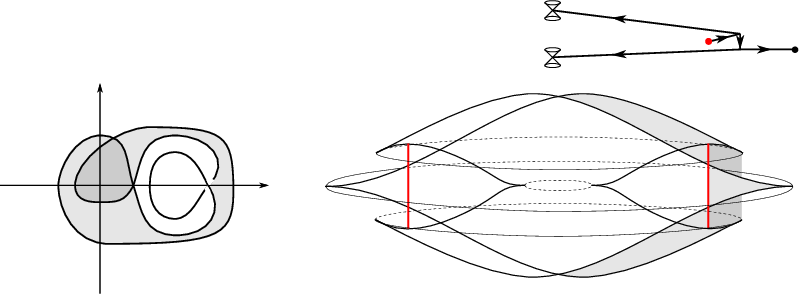}
\caption{A rigid partial gradient front-tree with one positive puncture at $a,$ one vertex at the outer cusp-edge, one cone-vertex at each  \color{black}  of the two cone points, and two three-valent vertices of type $Y_1.$}
\label{fig:chekanov-disk2}
\end{figure}

\subsection{A general computation of bilinearised LCH}
\label{sec:lchgeneral}
Even if we do not compute the full DGA of $\Lambda_{\OP{Ch}}$ we can still compute its bilinearised Legendrian contact homologies by a general result that we establish. Assume for now that we have a general Legendrian embedding of an oriented genus $g\geq 0$ surface $\Lambda \subset (\R^{5},\xi_{\OP{st}}).$ Further assume that $\Lambda$ has vanishing Maslov class and only Reeb chords of positive degree;  these assumptions are met by the Legendrian tori considered in this paper. \color{black}  The duality long exact sequence \cite{DualityLeg} by Ekholm--Etnyre--Sabloff and its generalisation \cite{BilinearisedLCH} by Bourgeois--Chantraine to the setting of bilinearised contact homology can then be used to compute all different bilinearised contact homologies.
\begin{thm}

  \label{thm:lchgeneral}

  Let $\varepsilon_0,\varepsilon_1 \colon (\mathcal{A},\partial) \to (\F,0)$ be two graded augmentations for a Legendrian oriented genus $g\geq 0$ surface $\Lambda \subset (\R^5,\xi_{\OP{st}})$ of vanishing Maslov class and with all Reeb chords in positive degrees. Then
\begin{itemize}
\item
when  $\varepsilon_0 = \varepsilon_1$: $$LCH_k^{\varepsilon_0,\varepsilon_1}(\Lambda)=\begin{cases} \F, & k=2,\\ \F^{g}, & k=1, \\ 0, & k\neq 1,2,
\end{cases}$$
\item when $0\leq g\leq 1$ and $\varepsilon_0 \neq \varepsilon_1$: $LCH_k^{\varepsilon_0,\varepsilon_1}(\Lambda) = 0$ for all $k,$
\end{itemize}
is satisfied for the bilinearised Legendrian contact homology groups.
\end{thm}
\begin{rem}
In the present situation two different augmentations are automatically inequivalent in the sense of \cite{BilinearisedLCH}, due to the fact that there are no Reeb chords of negative degree.  Namely, there can exist no DG-homotopy between two different augmentations merely for degree reasons (there cannot exist a chain homotopy even in the usual sense). \color{black}
\end{rem}
\begin{proof}
Let $\Sigma_g$ be an oriented genus $g\ge 0$ surface. First, assume that $0\leq g\leq 1$. Denote by $H_{k}(\Sigma_g;\varepsilon_0 \otimes \varepsilon_1^{-1})$ the Morse homology of $\Sigma_g$ with coefficients in $\F$ which is induced by the local system which takes the value $\varepsilon_0(\alpha)\varepsilon_1(\alpha)^{-1} \in \F$ on the class $\alpha \in H_1(\Sigma_g;\F).$ Recall the standard fact that this homology vanishes in all degrees unless the latter local system is trivial, i.e.~unless $\varepsilon_0=\varepsilon_1,$ while in the latter case we simply get the standard singular homology $H_k(\Sigma_g;\F)$ of the surface. (Of course when $g=0$ all local systems are trivial.)

 The duality long-exact sequence \cite[Theorem 1.5]{BilinearisedLCH} applied in this situation is equal to
 $$ \cdots \to H_{k+1}(\Sigma_g;\varepsilon_0\otimes \varepsilon_1^{-1}) \to LCH^{1-k}_{\varepsilon_1,\varepsilon_0}(\Lambda) \to LCH_k^{\varepsilon_0,\varepsilon_1}(\Lambda) \to H_{k}(\Sigma_g;\varepsilon_0\otimes \varepsilon_1^{-1})  \to \cdots $$
  The assumption on the gradings of the Reeb chords of $\Lambda$ immediately implies the vanishing $LCH_k^{\varepsilon_0,\varepsilon_1}(\Lambda)=0$ whenever $k\le 0$ and that $LCH^{1-k}_{\varepsilon_1,\varepsilon_0}(\Lambda)=0$ whenever $k \ge 1.$ Combined with the vanishing of the homology of the oriented genus $g=1$ surface with nontrivial local systems, the conclusion now follows in the case when $\varepsilon_0 \neq \varepsilon_1.$  Here we point out that two different graded augmentations automatically induce different local systems in the present situation; the reason is that the degree zero part of the DGA coincides with the Novikov ring $\F[H_1(\Lambda)]$ by assumption. \color{black}

What remains is the case when $\varepsilon_0=\varepsilon_1.$ Now we consider all $g\geq 0$. The immediate conclusion of the above vanishing result is that
\begin{align}
\label{eq:vanishing}
LCH^k_{\varepsilon_0,\varepsilon_0}(\Lambda)=0\ \mbox{when}\ k \neq 1,2,
\end{align}
and hence we can write part of the duality long exact sequence as
\begin{align*}
0=LCH^{-1}_{\varepsilon_0,\varepsilon_0}(\Lambda)\to LCH_2^{\varepsilon_0,\varepsilon_0}(\Lambda) \to H_{2}(\Sigma_g;\F) \to LCH^{0}_{\varepsilon_0,\varepsilon_0}(\Lambda)=0,
\end{align*}
which implies that
\begin{align*}
LCH^{\varepsilon_0,\varepsilon_0}_2(\Lambda) = \F.
\end{align*}

Recall that
\begin{equation}
\label{eq:adjoint}
  \dim LCH^{\varepsilon_0,\varepsilon_0}_*(\Lambda)=\dim LCH_{\varepsilon_0,\varepsilon_0}^*(\Lambda)
\end{equation}
is satisfied in general.

We now claim that $\dim LCH^{\varepsilon_0,\varepsilon_0}_1(\Lambda)=g.$
In order to show it, consider the following part of the duality long exact sequence, where  zero terms appear because of Equations \ref{eq:vanishing} and \ref{eq:adjoint}
\begin{align}\label{eq:secondformles}
 0\to LCH_{1}^{\varepsilon_0,\varepsilon_0}(\Lambda) \to  H_{1}(\Sigma_g;\F)\to LCH^{1}_{\varepsilon_0,\varepsilon_0}(\Lambda)\to  0.\end{align}
Using Equation \ref{eq:adjoint} and the fact that  $H_{1}(\Sigma_g;\F) \cong \F^{2g}$, we can rewrite long exact sequence \ref{eq:secondformles} as
\begin{align}\label{eq:finalles}
 0\to LCH_{1}^{\varepsilon_0,\varepsilon_0}(\Lambda) \to \F^{2g}\to LCH_{1}^{\varepsilon_0,\varepsilon_0}(\Lambda)\to  0.\end{align}
Since all  terms in long exact sequence \ref{eq:finalles} are $\F$-vector spaces, we see that
$LCH_{1}^{\varepsilon_0,\varepsilon_0}(\Lambda) \cong \F^{g}$. This finishes the proof.
\end{proof}

In addition, it is also possible to obtain restrictions on the variety of augmentations of a Legendrian torus satisfying the assumptions of Theorem \ref{thm:lchgeneral}.
\begin{thm}
\label{thm:augvar}
Let $\Lambda \subset (\R^5,\xi_{\OP{st}})$ be a Legendrian torus of vanishing Maslov class and with all Reeb chords in positive degrees. Its augmentation variety over $\C$ is then either empty, or   cut out by a single polynomial.
\end{thm}
\begin{proof}
First we argue that the augmentation variety is not all of $\OP{Sp}(\C[H_1(\Lambda)]) \cong (\C^*)^2.$ It suffices to show that there exists a Reeb chord of degree one \color{black} with a nontrivial boundary (whose image necessarily is an element of $\C[H_1(\Lambda)]$). We prove that there exists such a nonvanishing boundary with $\Z_2$-coefficients, which simplifies the considerations essentially, and implies the claim.

Argue by contradiction and assume that all such boundaries vanish. Thus the canonical unital DGA morphism $\varepsilon_0 \colon (\mathcal{A},\partial) \to \Z_2$ is clearly an augmentation. Consider the map
$$ LCH_{1}^{\varepsilon_0,\varepsilon_0}(\Lambda) \xrightarrow{\rho}  H_{1}(\T^2;\Z_2)$$
that arises in the duality long exact sequence \cite{DualityLeg}, which is defined by the choice of an auxiliary generic Morse function and Riemannian metric on $\Lambda.$ By \cite[Theorem 3.6(3) and Section 3.3.3]{DualityLeg} the value $\langle \rho(c),s\rangle$ for a Reeb chord $c \in LCH_{1}^{\varepsilon_0,\varepsilon_0}(\Lambda)$ and a critical point $s \in H_{1}(\T^2;\Z_2)$ of index one is given as the count of generalised pseudoholomorphic discs which consist of:
\begin{itemize}
\item a pseudoholomorphic polygon in $\R^4$ with boundary on the Lagrangian projection $\Pi_{\OP{Lag}}(\Lambda) \subset \R^4$ and a single positive puncture at the double point that corresponds to $c,$ together with
\item a negative gradient flow-line in $\Lambda$ that connects the boundary of the polygon to $s,$
\end{itemize}
such that the configuration moreover is rigid. Under suitable identifications, this count can be seen to be equal to the gradient $\nabla\partial(c) \in \C[\lambda^{\pm1},\mu^{\pm1}]$ of the Laurent polynomial $\partial(c)$ evaluated at $(1,1) \in \Z_2^2.$ The assumption that $\partial(c)$ vanishes for all Reeb chords $c$ of degree one is now seen to contradict the fact that $\rho$ is an inclusion of a one-dimensional $\Z_2$-vector space, as was shown in the proof of Theorem \ref{thm:lchgeneral} above.

Now consider an augmentation $\varepsilon_0 \colon (\mathcal{A},\partial) \to \C.$ We will show that there exists a one-dimensional component of the augmentation variety that passes through the corresponding point $\varepsilon_0 \in \OP{Sp}(\C[H_1(\Lambda)]).$ Choose a basis $\langle a_0,a_1,\ldots,a_m\rangle $ of the degree-one subspace of the complex $(LCC_*^{\varepsilon_0,\varepsilon_0}(\Lambda),\partial^{\varepsilon_0,\varepsilon_0}),$ which automatically consists of cycles, together with a basis $\langle b_1,\ldots,b_m\rangle $ of the degree-two subspace, such that $\partial^{\varepsilon_0,\varepsilon_0}(b_i)=a_i$ moreover is satisfied for $i\ge 1.$  That this is possible to achieve follows from Theorem \ref{thm:lchgeneral} above together with the assumption that $\Lambda$ is a torus. \color{black}

On the level of the DGA, after making the Novikov generators $\mu,\lambda$ commute with the Reeb chord generators, we thus conclude that
$$\partial b_i=\sum_{j=0}^m P^i_j(\mu,\lambda)a_j,$$
where $P^i_j \in \C[\mu^{\pm1},\lambda^{\pm1}]$ for $i=1,\ldots,m;$ $j=0,1,\ldots,m$ are Laurent polynomials that satisfy $\varepsilon_0(P^i_j)=\delta^i_j.$

Then consider the localisation
$$\C[\mu^{\pm1},\lambda^{\pm1}] \to \C[\mu^{\pm1},\lambda^{\pm1}][(P^1_1\cdots P^m_m)^{-1}],$$
and the induced DGA $(\widetilde{A},\widetilde{\partial})$ with coefficients in $\C[\mu^{\pm1},\lambda^{\pm1}][(P^1_1\cdots P^m_m)^{-1}]$ together with the canonical unital DGA morphism $(\mathcal{A},\partial) \to (\widetilde{\mathcal{A}},\widetilde{\partial}).$ Since none of the polynomials $P^i_i$ vanish at the point corresponding to $\varepsilon_0,$ the augmentation $\varepsilon_0$ descends to an augmentation $\widetilde{\varepsilon_0} \colon (\widetilde{A},\widetilde{\partial}) \to \C.$ In other words, we have
\begin{equation}
  \label{eq:deltilde}
  \widetilde{\partial}b_i=\sum_{j=0}^m P^i_ja_j, \:\: i \in [1,m],
\end{equation}
where $P^i_i \in \C[\mu^{\pm1},\lambda^{\pm1}][(P^1_1\cdots P^m_m)^{-1}]$ is a unit and $\widetilde{\varepsilon_0}(P^i_j)=0$ if $i \neq j.$

By induction we assume that we can perform a $\C[\mu^{\pm1},\lambda^{\pm1}][(P^1_1\cdots P^m_m)^{-1}]$-linear change of coordinates of the $\C[\mu^{\pm1},\lambda^{\pm1}][(P^1_1\cdots P^m_m)^{-1}]$-submodule generated by $\langle a_1,\ldots,a_m \rangle$ after which, in addition  \color{black}  to Equality \eqref{eq:deltilde}, also $\widetilde{\partial} a_i=0$ is satisfied for the (possibly empty) set of values $i \in [m_0+1,m]$ and some $m_0 \in [1,m].$ Since $\widetilde{\partial}^2b_{m_0}=0$ is satisfied, we can express
$$\widetilde{\partial}a_{m_0}=-(P^{m_0}_{m_0})^{-1}\sum_{j < m_0} P^{m_0}_j\widetilde{\partial}a_j,\:\:\:i=1,\ldots,m.$$
Hence we can make $\widetilde{\partial}a_{m_0}=0$ satisfied as well after the change of coordinates
$$ a_{m_0} \mapsto a_{m_0} + (P^{m_0}_{m_0})^{-1}\sum_{j < m_0} P^{m_0}_ja_j.$$
By induction we may thus readily assume that $\widetilde{\partial}a_i=0$ is satisfied for all $i \in [1,m]$ in addition to Equality \eqref{eq:deltilde}.

In conclusion, we have shown that the augmentation variety is cut out by the single function $\widetilde{\partial}a_0 \in \C[\mu^{\pm1},\lambda^{\pm1}][(P^1_1\cdots P^m_m)^{-1}]$ inside the principal open subset
$$\{P^1_1\cdots P^m_m \neq 0\} \subset \OP{Sp}(\C[\mu^{\pm1},\lambda^{\pm1}]),$$
which implies the claim.
\end{proof}

\section{Sublooseness and existence of exact caps}
\label{sec:subflex}

There exists an elementary exact Lagrangian cobordism from either of $\Lambda_{\OP{Cl}},\Lambda_{\OP{Ch}} \subset (S^5,\xi_{\OP{st}})$ to the loose sphere. This elementary cobordism arises as the  exact Lagrangian 2-handle  \color{black}  attachment of a single handle as constructed in \cite{Surgery},  induced by an ambient Legendrian 1-surgery.  \color{black}

We recall the construction  of Legendrian 1-surgery on a Legendrian surface; since the surgery is performed on a sphere of codimension one, it is a \emph{critical} surgery according to the terminology from \cite{Surgery}. Assume that we are given  \color{black}  a Legendrian disc $D \subset (S^5,\xi_{\OP{st}})$ that intersects a Legendrian surface $\Lambda_-$ cleanly precisely along its boundary.  The disc determines  \color{black}  an embedded Legendrian $\Lambda_+$ which,  as a topological space, \color{black}  is obtained from $\Lambda_-$ by surgery along the embedded codimension one sphere $\partial D \subset \Lambda_-.$ There is also an associated exact Lagrangian handle attachment cobordism that goes from $\Lambda_-$ (at the negative end) to $\Lambda_+$ (at the positive end). Note that the handle attachment cobordism can be assumed to be a trivial cylinder outside of an arbitrarily small neighbourhood of $\R \times D \subset (\R \times S^5,d(e^t\alpha_{\OP{st}}))$ in the symplectisation.

Recall that the local model inside a neighbourhood of the disc $D$ of the Legendrian surface before and after an ambient 1-surgery both can be taken to have fronts that satisfy a rotational $S^1$-symmetry. For that reason, \color{black}  the ambient surgery can be described using the Lefschetz fibration $f$ from Section \ref{sec:lefschetz}. On the left in Figure \ref{fig:surgery} is the local model for the Legendrian $\Lambda_-$ together with the Legendrian surgery disc, while the effect of the surgery is shown on the right side. In terms of the front projection, the Legendrian $\Lambda_-$ is a double-cone with the surgery disc $D$ projecting to the cone-point, while the result $\Lambda_+$ of the surgery resolves the cone-point to two smooth sheets that intersects transversely along an embedded $S^1$; see  the middle part  \color{black}  of Figure \ref{fig:loosesphere}.

To relate the front projection of the surgery model and the image of its Lagrangian projection under the Lefschetz fibration $f$ we again argue as in the last paragraphs of Section \ref{sec:frontsoflifts}. That is, we first lift the curves shown in Figure \ref{fig:surgery} (the image of the Lagrangian projection under the Lefschetz fibration) under the two-fold cover of $D^2$ branched at the origin. We view these lifts as an exact Lagrangian immersion of a union of two paths in $D^2$. We take the Legendrian lift of this Lagrangian immersion to thus obtain two embedded Legendrian arcs in $\R^3$; these Legendrian arcs have a front projection which is symmetric under the reflection $x \mapsto -x$. Finally perform a symmetric $S^1$-spinning construction to obtain the Legendrian surfaces which model the ambient 1-surgery.

\begin{remark}
The front version of the ambient surgery described above thus replaces a double cone with two immersed sheets that intersect transversely in a circle with a unique maximum-type Reeb chord; see Figure \ref{fig:loosesphere}. The more common picture of this surgery is to, first, start with a circular cusp-edge and, then, replace it by two disjoint embedded sheets with a minimum-type Reeb chord in the middle. In fact, these two surgeries are related by a Legendre transform and, in particular, these local models are contactomorphic. (Another way to see this is to perform a rotation of Figure \ref{fig:surgery} while considering the corresponding Legendrian lifts. A rotation by $\pi$ radians interchanges the two different models.)
\end{remark}
 \color{black}

\begin{figure}[htp]
\centering
\vspace{3mm}
\labellist
\pinlabel $x$ at 195 92
\pinlabel $y$ at 93 197
\pinlabel $x$ at 403 92
\pinlabel $y$ at 300 197
\pinlabel $\color{blue}\Lambda^-$ at 55 105
\pinlabel $\color{blue}\Lambda^+$ at 263 105
\pinlabel $D$ at 80 100
\endlabellist
\includegraphics{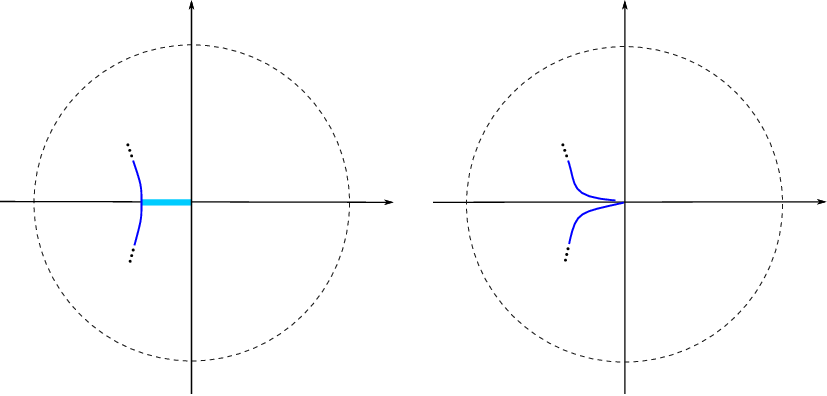}
\caption{The local model of the (critical) Legendrian ambient surgery described using the Lefschetz fibration $f$ from Section \ref{sec:lefschetz}. To the left: the Legendrian $\Lambda_-$ before the surgery together with the Legendrian surgery disc $D.$ To the right: the effect of the surgery $\Lambda_+,$ which is a Legendrian with one additional Reeb chord (which in this case projects to the origin).  The two branches of the Legendrian after the surgery project to two Lefschetz thimbles which are disjoint except for their common starting point, i.e.~the origin. \color{black} }
\label{fig:surgery}
\end{figure}
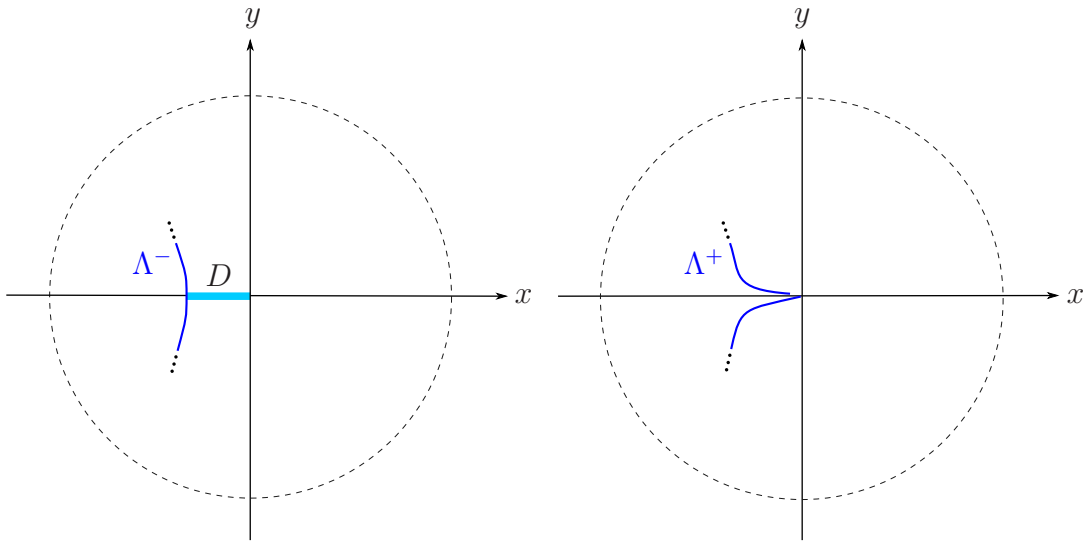

\emph{The lift of the Clifford torus:} In this case both the surgery disc and the result of the Legendrian ambient surgery are shown in Figure \ref{fig:loosesphere}. The result is clearly the loose two-sphere, since it coincides with an $S^1$-spun of a stabilised arc in some neighbourhood as in the proof of Lemma \ref{lem:loose} (alternatively, one can apply Lemma \ref{lem:loose} directly to the Lefschetz fibration).

\emph{The lift of the Chekanov torus:} In this case we take the surgery disc to live above the lower of the two cone-points in the front projection Figure \ref{fig:Chekanov}. The result of the Legendrian ambient surgery is then shown on the left side of  Figure \ref{fig:ChekanovLoose}. After an $S^1$-spun of a Reidemeister-II move (this preserves the Legendrian isotopy class) we see that the result of the surgery indeed is the loose two-sphere as well; in the Legendrian on the right in Figure \ref{fig:ChekanovLoose} one can clearly see the $S^1$-spun of a stabilised arc as described in the proof of Lemma \ref{lem:loose} (alternatively, one can apply Lemma \ref{lem:loose} directly to the Lefschetz fibration).

\begin{figure}[htp]
\centering
\labellist
\pinlabel $\Lambda_{\OP{Cl}}$ at 12 75
\pinlabel $\Lambda_{\OP{loose}}$ at 250 75
\pinlabel $D_1$ at 118 45
\endlabellist
\includegraphics[scale=0.9]{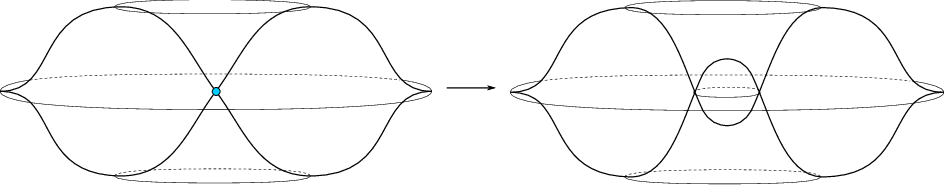}
\caption{Above the blue dot inside the cone-point there lives an embedded Legendrian disc $D_1 \subset (S^5,\xi_{\OP{st}})$ which intersects the Legendrian cleanly along its boundary. The ambient surgery on $\Lambda_{\OP{Ch}}$ by using the Legendrian surgery disc $D_1.$ The resulting manifold is obviously the loose two-sphere $\Lambda_{\OP{loose}}.$}
\label{fig:loosesphere}
\end{figure}

\begin{figure}[htp]
  \centering
  \labellist
  \pinlabel $\Lambda_{\OP{loose}}$ at 260 80
\endlabellist
\includegraphics[scale=0.9]{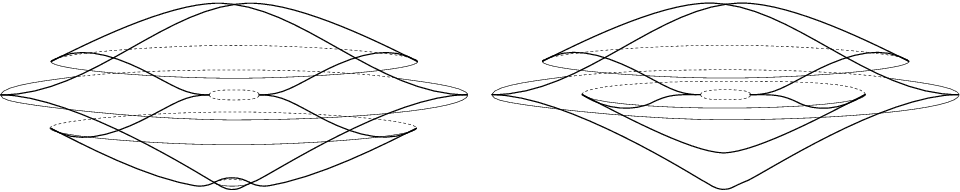}
\caption{On the left: The result of an ambient surgery on $\Lambda_{\OP{Ch}}$ shown in Figure \ref{fig:Chekanov} by using the Legendrian surgery disc $D_2$ contained over the bottom cone point. On the right: after an $S^1$-spun Reidemeister-II move in the front projection (this is a Legendrian isotopy) we see an $S^1$-spun of a stabilisation. Hence the Legendrian is the loose two-sphere.}
\label{fig:ChekanovLoose}
\end{figure}

Since it is well-known that  the loose sphere does not admit any exact Lagrangian filling, neither does any of the two former Legendrian tori $\Lambda_{\OP{Cl}}$ or $\Lambda_{\OP{Ch}}.$ In addition, the loose sphere was shown to admit an exact Lagrangian cap in \cite{EkholmMurphy}. By concatenation of exact Lagrangian cobordisms, it follows that both tori $\Lambda_{\OP{Cl}}$ and $\Lambda_{\OP{Ch}}$ also admit exact Lagrangian caps.

\subsection{Regularity of exact Lagrangian cobordisms}

The class of regular Lagrangians in a Weinstein manifold was first introduced by Eliashberg--Ganatra--Lazarev in \cite{FlexibleLagrangians} and was then studied in \cite{FlexibleLagrangians,EliashbergSurvey,LazarevHprincipleRegularLagrangians}.

We start with the definition of a regular exact Lagrangian
cobordism.

\begin{defn}
Given a Weinstein cobordism $(W, \omega, X, \phi)$, where $\omega$
denotes a symplectic form, $X$ is an expanding Liouville vector
field for $(W, \omega)$, $\phi:W\to \R$ is defining for $W$ and
Lyapunov for $X$, an exact Lagrangian cobordism $L\subset W$ is
called {\em regular} if $(W, \omega, X, \phi)$ can be  deformed  to
a Weinstein cobordism $(W, \omega', X', \phi')$ through Weinstein
structures for which $L$ is Lagrangian and $X'$ is tangent to $L$.
\end{defn}

The following question about regularity of exact Lagrangian
cobordisms has been asked by Eliashberg--Ganatra--Lazarev in
\cite{FlexibleLagrangians}, and by Eliashberg in
\cite{EliashbergSurvey}.
\begin{question}[\cite{FlexibleLagrangians, EliashbergSurvey}]
\label{questionfromGanatra-Eliashberg-Lazarev} Given a Weinstein cobordism $(W, \omega, X,
\phi)$ such that
$\partial_- W\neq \emptyset$  and an exact Lagrangian cobordism $L\subset (W, \omega, X,
\phi)$ such that $\partial_- L\neq \emptyset$ is not loose, is $L$
regular?
\end{question}
Since an exact Lagrangian cap inside a symplectisation $L \subset \R \times Y$  is null-homologous as a class in the relative homology group, it cannot be regular by \cite[Lemma 2.4]{FlexibleLagrangians}. The existence of subloose, but not loose, Legendrians combined with the existence of exact Lagrangian caps \cite{LagrangianCaps} by Eliashberg--Murphy thus provides a negative answer to Question \ref{questionfromGanatra-Eliashberg-Lazarev} above. (One could also allude to the example of the subloose sphere by Casals--Murphy referred to in Remark \ref{rem:subflexible}.) We elaborate on this construction in the following result. \color{black}
\begin{thm}
  \label{NonRegularCapsofSurfaces}
  For any $g>0$ and $k_1,\ldots,k_r \ge 0$ there exists infinitely many different Legendrian isotopy classes of subloose Legendrian embeddings $\Lambda \subset (\R^{2k_1+\ldots+2k_r+5},\xi_{\OP{st}})$ of the manifold $S^{k_1} \times \ldots \times S^{k_r} \times \Sigma_g,$ where $\Sigma_g$ denotes the surface of genus $g,$ which:
  \begin{itemize}
  \item have vanishing Maslov classes;
  \item have Chekanov--Eliashberg algebras with Novikov coefficients $R=\C[H_1(\Lambda)]$ that admits (0-graded) augmentations; and
  \item admit orientable exact Lagrangian caps \color{black} inside the symplectisation with vanishing Maslov classes.
  \end{itemize}
In particular, none of these Legendrian surfaces are loose.
\end{thm}

\begin{rem}
It can further be shown that there are infinitely many diffeomorphism types of orientable exact Lagrangian caps for the above subloose Legendrians; see \cite{EkholmMurphy}.
\end{rem}

In view of the above we would like to make the following reasonable refinement of
Question \ref{questionfromGanatra-Eliashberg-Lazarev}:

\begin{question}
\label{refinedquestionfromGanatra-Eliashberg-Lazarev}
Let $L$ be an exact Lagrangian cobordism in a Weinstein cobordism $(W,
\omega, X, \phi)$ such that  $\partial_- L \neq \emptyset$ is not subloose. Is $L$
regular?
\end{question}

\subsection{Proof of Theorem \ref{NonRegularCapsofSurfaces}}
\label{ProofOfTheMainTheorem}

\subsubsection{The case of surfaces}

We start by considering the subloose surface of genus $g>0$ obtained by taking the cusp-connect sum of $g$ number of copies of $\Lambda_{\OP{Cl}}$; for the definition of cusp-connect sum see \cite{Surgery}. These Legendrian surfaces were also considered in \cite{KnottedLegendrianSurface}.  
\begin{remark}An important feature of the cusp-connected sum is that it preserves all of the three following properties: admitting an augmentation, admitting an exact Lagrangian filling, and admitting a loose chart (i.e.~being loose).
\end{remark}
 \color{black}
Denote the Legendrian embedding above \blk by $\Sigma_g \subset (\R^5,\xi).$ Observe that this Legendrian surface satisfies the statements of the theorem. What then remains is to make modifications to yield an infinite set of Legendrians that are pairwise different up to Legendrian isotopy. \color{black}

For any $N>0$ we then produce a Legendrian $2$-sphere $\Lambda_N$ which admits an exact spin Lagrangian filling $L_N$ of vanishing Maslov class such that
$$\dim \bigoplus\limits_i H_i(L_N;\C) \ge N$$
is satisfied. For instance, $\Lambda_N$ can be taken to be the cusp connected sum $N$ copies of the Legendrian sphere from \cite[Example 9.4.2]{Cthulhu}. By Seidel's isomorphism (see e.g.~\cite{FloerConc,Cthulhu}) it follows that the augmentation $\varepsilon_{L_N}$ induced by this filling has a Legendrian contact cohomology $$\bigoplus\limits_i LCH^{i}_{\varepsilon_{L_N}}(\Lambda_N) \cong \bigoplus\limits_i H_{i}(L_N;\C)$$
with complex coefficients whose total rank is at least $N$ as well.

Recall the long exact sequence from \cite[Theorem 1.1]{Cthulhu}  which involves the linearised Legendrian contact homology groups of the two Legendrians $\Lambda_\pm$ related by an exact Lagrangian cobordism from $\Lambda_-$ to $\Lambda_+$, and the singular homology of the cobordism. In the particular case when the Lagrangian cobordism is a standard Lagrangian 1-handle attachment cobordism $L$ from two Legendrians $\Lambda_-=\Lambda_1 \cup \Lambda_2$ contained in separate Darboux balls to their cusp connected sum $\Lambda_-=\Lambda_1 \# \Lambda_2$, this long exact sequence specialises to
\begin{align*}
& \ldots \to H_{3-k}(L,\Lambda_-) \to LCH^k_{\varepsilon}(\Lambda_1) \oplus LCH^k_{\varepsilon}(\Lambda_2) \to LCH^k_{\varepsilon_+}(\Lambda_1 \# \Lambda_2 ) \to\\
& \to H_{3-(k-1)}(L,\Lambda_-) \to \ldots,
\end{align*}
where $\varepsilon_+$ is the pull-back of $\varepsilon$ under the DGA-morphism induced by $L$, and
$$H_{3-k}(L,\Lambda_-)=\begin{cases}
\F, & k=2\\
0, & o.w.
\end{cases}$$
From this long exact sequence it thus  \color{black}  follows that there exists an augmentation $\varepsilon$ of the cusp-connect sum $\Lambda \subset (\R^5,\xi_{\OP{st}})$ of $\Lambda_N$ and $\Sigma_g$ for which the linearised Legendrian contact cohomology $LCH^*_{\varepsilon}(\Lambda)$ with coefficients in $\C$ has total rank equal to at least $N-1.$

Recall that the set of isomorphism classes of linearised Legendrian contact cohomologies for all different augmentations is a Legendrian isotopy invariant \cite{DiffGradedAlgebraLegLinks}. Further, the dimension of the linearised contact cohomology for any augmentation has an a priori upper bound given by the number of Reeb chords for any given generic Legendrian representative. It is thus clear that we can construct infinitely many Legendrian embeddings $\Lambda$ in the above manner by choosing $N>0$ larger and larger.

\subsubsection{The existence of caps}
Since we saw above that $\Lambda_{\OP{Cl}}$ admits an orientable exact Lagrangian cobordism to the loose sphere, the same is true for $\Lambda$ as well.  Namely, performing $g$ surgeries on the $g$ number of $\Lambda_{\OP{Cl}}$-summands produces a connected sum of a number $g$ of loose Legendrian spheres contained in different Darboux balls. (Recall that cusp connected sum preserves the property of being loose.) The sought exact Lagrangian cobordism is the associated handle attachments. (For a cobordism to a loose genus $g-1$ surface, a single such Legendrian 1-surgery would clearly have been sufficient, but here we are interested in a cobordism to a \emph{sphere}.) \color{black}  The existence of exact Lagrangian caps for the loose sphere was established in \cite{LagrangianCaps} by Eliashberg--Murphy; also see the work \cite{EkholmMurphy} by Ekholm--Eliashberg--Murphy--Smith. The sought cap of $\Lambda$ is then given by the concatenation of the exact Lagrangian cobordism to the loose sphere and the exact Lagrangian cap of the loose sphere.

\subsubsection{The higher dimensional case}
We apply the spherical front-spinning construction by the second author \cite{FrontSpinningConstruction} to the above Legendrian surfaces and exact Lagrangian cobordisms to produce the subloose embeddings and their exact Lagrangian caps.

In order to deduce the existence of augmentations with Novikov coefficients we use the partial computation \cite[Theorem 3.1]{EstimatingReebChordsCharacteristicAlgebra} of the DGA of the $S^m$-spun of a Legendrian.

In order to deduce that the spuns remain in different Legendrian isotopy classes we use K\"{u}nneths formula \cite[Theorem 2.5]{FloerConc}, by which
$$LCH^*_{\varepsilon_{\Sigma_{S^m}L_N}}(\Sigma_{S^m}\Lambda_N)\cong LCH^*_{\varepsilon_{L_N}}(\Lambda_N) \otimes H_*(S^m).$$
Then we iterate this procedure and compute $LCH^*_{\varepsilon_{\Sigma_{S^{k_r}}\dots\Sigma_{S^{k_1}}L_N}}(\Sigma_{S^{k_r}}\dots\Sigma_{S^{k_1}}\Lambda_N)$. Note that this isomorphism holds also with coefficients in $\C$ given that the filling $ L_N$ is spin.
In addition, observe that \cite[Proposition 1.1]{FrontSpinningConstruction} implies that the spherical front spun of the cusp-connected sum is the cusp-connected sum of the spherical front spuns of the components.  Recall that a Legendrian obtained by performing a cusp-connected sum is related to the original Legendrian by an exact Lagrangian cobordism. \color{black}  Using this and the long exact sequence described in \cite[Theorem 1.1]{Cthulhu}, similarly to the 2-dimensional case,
we observe that  there exists infinitely many different Legendrian isotopy classes of subloose Legendrian embeddings $\Lambda \subset (\R^{2k_1+\ldots+2k_r+5},\xi_{\OP{st}})$ of the manifold $S^{k_1} \times \ldots \times S^{k_r} \times \Sigma_g$.
\qed

\section{Looseness of the standard Legendrian disc}
\label{sec:disc}

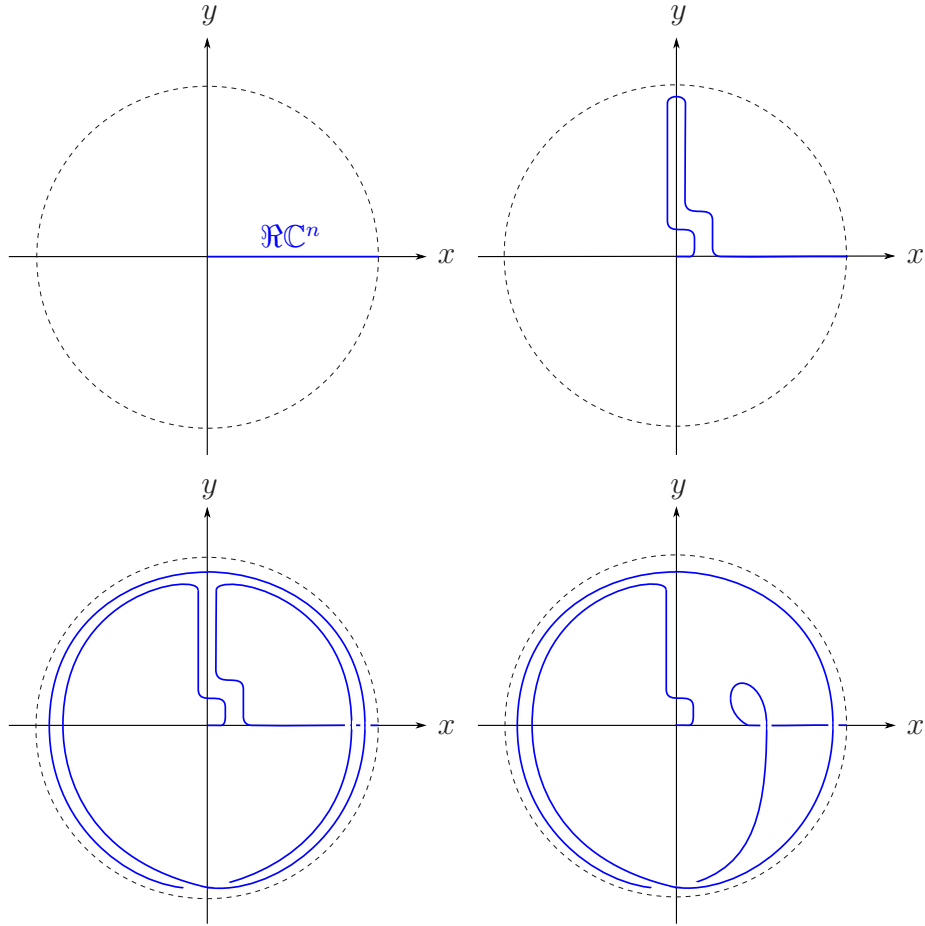
\begin{figure}[h!]
\vspace{3mm}
\centering
\labellist
\pinlabel $x$ at 402 296
\pinlabel $x$ at 402 87
\pinlabel $x$ at 194 87
\pinlabel $x$ at 194 296
\pinlabel $y$ at 297 403
\pinlabel $y$ at 297 193
\pinlabel $y$ at 89 403
\pinlabel $y$ at 89 193
\pinlabel $\color{blue}\Re\C^n$ at 125 305
\endlabellist
\includegraphics[scale=0.85]{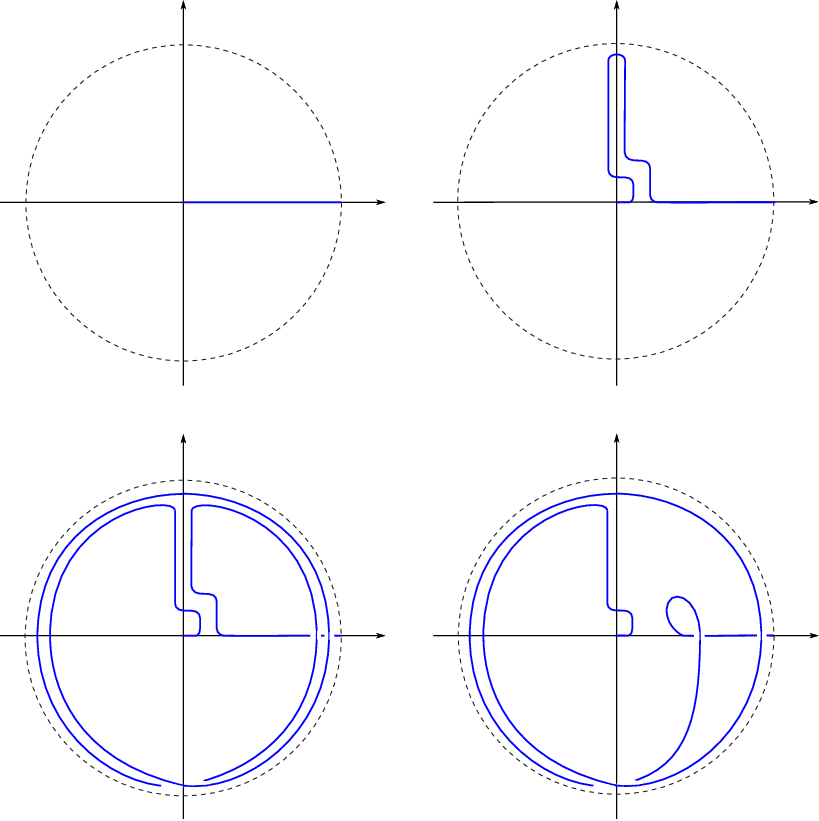}
\caption{A regular homotopy through Lagrangian discs described in terms of the fibration $f \colon \Sigma \to \dot{D}^2_{1/2}.$}
\label{fig:homotopy-disc}
\end{figure}

The standard Legendrian disc has a Lagrangian projection to $\CP^n$ that can be identified with the real part inside an affine chart $\C^n.$ This is the vanishing thimble for the standard Lefschetz fibration on $\C^n,$ and hence we identify it with the Lagrangian which lives above the positive $x$-axis inside $\dot{D}^2_{1/2}$ under the Lefschetz fibration $f \colon \Sigma \to \dot{D}^2_{1/2}.$

We must produce a Legendrian isotopy which is fixed near the boundary of the aforementioned Legendrian to exhibit a loose chart. To do this we follow the same strategy as in the proof of Theorem \ref{thm:frontsoflifts} given in Section \ref{sec:frontsoflifts}. The sequence of moves is shown in Figure \ref{fig:homotopy-disc}, where going from the top right curve to the bottom left curve requires that we pass over the line at infinity. Note that the Legendrian lifts do not intersect the boundary of the disc, which is fixed, even though the Lagrangian projections clearly do intersect. In order to see this it suffices to compute the symplectic action at the points that intersect the boundary, and to check that it is not an integer multiple of $\pi.$

In the bottom right of the figure \color{black} resulting from the homotopy, we can indeed easily see the loose chart.

\section{Further directions}
\label{furtherdirections}
Here we present some expectations that we hope to show in future work.

\subsection{Relations between the augmentation variety and the superpotential}

\label{sec:potential}

Recall that there exists precisely three pseudoholomorphic Maslov-two disc families in $\CP^2$ with boundary on the Clifford torus. Similarly, the augmentation variety of the Legendrian $\Lambda_{\OP{Cl}}$ is the zero-locus of the Laurent polynomial $\partial a=1+\lambda(1+\mu),$ which also is given by a count of precisely three pseudoholomorphic discs. In general we expect that we should be able to compute the augmentation variety in the Bott degenerate situation of $(E,\alpha),$ where the embedded Legendrian lives above the embedded Lagrangian torus, and in this manner obtain a relation between the augmentation variety and the count of such pseudoholomorphic Maslov-two discs. Recall that the count of the latter discs define the so-called superpotential of the Lagrangian, which also is a Laurent polynomial.

In the case of the monotone Clifford and the Chekanov torus $L_{\OP{Cl}},L_{\OP{Ch}} \subset (\CP^2,\omega_{\OP{FS}})$ such a correspondence can indeed be confirmed by means of a hands-on comparison. Recall that the superpotentials of the two tori are given by
\begin{align*}
& \mathfrak{P}_{L_{\OP{Cl}}}(u,v)=u(1+v)+1/u^2v \in u\cdot\C[u^{\pm 3},v^{\pm1}],\\
& \mathfrak{P}_{L_{\OP{Ch}}}(u,v)=u+(1+v)^2/u^2v \in u\cdot\C[u^{\pm 3},v^{\pm1}],
\end{align*}
for suitable spin structures, and relative a suitable basis of $\langle u,v\rangle =H_1(\T^2)$ where $u$ and $v$ correspond to generators of Maslov index equal to two and zero modulo six, respectively; see \cite{Auroux06}.

The precise relation between the augmentation polynomial and the
superpotential is given by the following conjecture. Assume that
$$H_0(\mathcal{A}(\Lambda),\partial)=\C[\mu^{\pm1},\lambda^{\pm1}]/\langle \mathcal{A}ug_{\Lambda}(\mu,\lambda) \rangle$$
where $\mathcal{A}ug_{\Lambda}(\mu,\lambda)$ is the so-called {\bf augmentation polynomial} (which under the assumptions is uniquely determined up to the multiplication by a unit).
\begin{conj}
  \label{correspondeceSuperpotentialAugPol} For an embedded monotone Lagrangian two-torus $L \subset (\CP^2,\omega_{\OP{FS}})$ the superpotential $\mathfrak{P}_L(u,v)$ can be recovered (up to the multiplication by a unit) from the augmentation polynomial $\mathcal{A}ug_{\Lambda}(\mu,\lambda)$ of the Legendrian lift $\Lambda$ of its canonical threefold Bohr--Sommerfeld cover from the equality
$$\mathfrak{P}_{L}(u,v)=\mathcal{A}ug_{\Lambda}(u^3,v) \in u\cdot\C[u^{\pm 3},v^{\pm1}]$$
of ideals, given that we use suitable choices of bases of $H_1(L)$ and $H_1(\Lambda),$ as well as capping paths, and that $\mathcal{A}ug_{\Lambda}(\mu,\lambda)$ has been normalised appropriately.
\end{conj}
\begin{remark}
The basis of $H_1(L)=\langle u,v\rangle$ should be chosen so that $u$ and $v$ are mapped to primitive classes that are the boundaries of discs of Maslov index two and zero\color{black}, respectively. The capping paths in $\Lambda$ used when computing the augmentation variety must then be chosen accordingly.
\end{remark}

Since Vianna's infinite family of monotone Lagrangian tori in \cite{Vianna16} can be distinguished by computations of their superpotentials, Conjecture \ref{correspondeceSuperpotentialAugPol} would imply that also their Legendrian lifts all live in different Legendrian isotopy classes.

\subsection{Sublooseness of monotone tori in the projective plane}
As described in Proposition \ref{prp:monotonefilling}, the Legendrian lift of any canonical twofold Bohr--Sommerfeld cover of a monotone Lagrangian torus inside $(\CP^1 \times \CP^1,\omega_{\OP{FS}}\oplus\omega_{\OP{FS}})$ can be seen to admit a monotone filling inside full complex line-bundle. Even if monotone fillability is weaker than exact fillability, in the present setting they should still allow us to conclude the existence of an augmentation with coefficients in $\F.$

Contrary to this, similarly as in the case of the Clifford and Chekanov torus, we believe that the canonical threefold Bohr--Sommerfeld cover of any monotone Lagrangian torus inside $(\CP^2,\omega_{\OP{FS}})$ always has a subloose Legendrian lift.

\bibliographystyle{plain}
\bibliography{references}

\end{document}